\documentclass[11pt, reqno]{amsart}
\usepackage{amssymb}
\usepackage{amsmath}
\usepackage{amscd}
\usepackage{amsfonts}
\usepackage{mathrsfs}
\usepackage{stmaryrd}
\usepackage{tikz}
\usetikzlibrary{arrows}
\usepackage[T1]{fontenc}
\allowdisplaybreaks

\makeatletter
\DeclareFontFamily{U}{tipa}{}
\DeclareFontShape{U}{tipa}{m}{n}{<->tipa10}{}
\newcommand{\arc@char}{{\usefont{U}{tipa}{m}{n}\symbol{62}}}%

\newcommand{\arc}[1]{\mathpalette\arc@arc{#1}}

\newcommand{\arc@arc}[2]{%
  \sbox0{$\m@th#1#2$}%
  \vbox{
    \hbox{\resizebox{\wd0}{\height}{\arc@char}}
    \nointerlineskip
    \box0
  }%
}
\makeatother

\usepackage{bbm}

\theoremstyle{definition}
\newtheorem{theorem}{Theorem}[section]

\newtheorem{thm}[theorem]{Theorem}
\newtheorem{prop}[theorem]{Proposition}

\newtheorem{defn}[theorem]{Definition}
\newtheorem{lemma}[theorem]{Lemma}
\newtheorem{cor}[theorem]{Corollary}
\newtheorem{prop-def}{Proposition-Definition}[section]

\newtheorem{rema}[theorem]{Remark}

\newtheorem{exam}[theorem]{Example}
\newtheorem{nota}[theorem]{Notation}

\newcommand{\N}{{\mathbb N}}
\newcommand{\C}{{\mathbb C}}
\newcommand{\Z}{{\mathbb Z}}

\newcommand{\one}{\mathbf{1}}
\renewcommand{\d}{\mathbf{d}}
\newcommand{\wt}{\mbox{\rm wt}\ }

\newcommand{\otimesdots}{\otimes \cdots \otimes}
\newcommand{\nn}{\nonumber}
 
\begin{document}

\setlength{\oddsidemargin}{0cm} \setlength{\evensidemargin}{0cm}
\baselineskip=18pt

\title{On the cohomology of meromorphic open-string vertex algebras}
\author{Fei Qi}

\begin{abstract}

This paper generalizes Huang's cohomology theory 
of grading-restricted vertex algebras to 
meromorphic open-string vertex algebras (MOSVAs hereafter), which are 
noncommutative generalizations of grading-restricted 
vertex algebras introduced by Huang. The vertex operators 
for these algebras satisfy associativity but do not necessarily satisfy the commutativity. 
Moreover, the MOSVA and its bimodules considered in this paper 
do not necessarily have finite-dimensional homogeneous subspaces, 
though we do require that they have  lower-bounded gradings.
The construction and results in this paper will be used in a joint paper 
by Huang and the author to give a cohomological criterion of the 
reductivity for modules for grading-restricted vertex algebras

% This paper generalizes Huang's cohomology theory of vertex algebras in \cite{Hcoh}. The vertex operators here are not assumed to have commutativity; they are assumed to form a meromorphic open-string vertex algebra (MOSVA hereafter), which is a noncommutative generalization of a vertex algebra proposed by Huang in \cite{H-MOSVA}. The cohomology theory is built for a MOSVA and its bimodules satisfying a pole-order condition proposed in \cite{Q-Rep} and is analogous to the Hochschild cohomology for an associative algebra and its bimodules. Moreover, the MOSVA and its bimodules are assumed to have a lower-bounded grading, but the homogeneous subspaces can be of infinite dimension. The theory developed in this paper will be used in \cite{HQ-Red} to give a cohomological criterion of reductivity for modules of vertex algebras. 
\end{abstract}

\maketitle
\tableofcontents

\section{Introduction}

In \cite{Hcoh}, Huang introduced the cohomology of a
grading-restricted vertex algebra. Vertex (operator) algebras are algebraic structures formed by vertex operators satisfying both commutativity and
associativity. They arose naturally in both mathematics
and physics (see \cite{BPZ}, \cite{B} and \cite{FLM}).
In the context of quantum field theory, the commutativity
of vertex operators is closely related to the locality of
meromorphic fields. The associativity of vertex operators
is a strong form of the operator product expansion. Thus,
a vertex algebra can be viewed as an analogue to commutative
associative algebras. The cohomology introduced in \cite{Hcoh} can be viewed as an analogue of Harrison cohomology. 

In the present paper, we introduce the cohomology of a meromorphic open-string vertex algebra. The notion of meromorphic open-string vertex algebra (MOSVA)
was introduced by Huang in \cite{H-MOSVA}. It is a noncommutative
generalization of the notion of vertex algebra, as the vertex operators
in a MOSVA satisfy associativity but do not necessarily satisfy
commutativity. Thus a MOSVA can be viewed as an analogue to associative algebras. The cohomology introduced in this paper can be viewed as an analogue of Hochschild cohomology.

% In \cite{H-MOSVA-Riemann}, for a Riemannian manifold
% $M$, Huang constructed a MOSVA using the parallel sections of the
% bundles $TM^{\otimes n}$ for all $n$. More importantly, eigenfunctions
% for the Laplacian generates left modules for the MOSVA. These structures
% should be related to the 2-dimension nonlinear sigma model with $M$
% as the target space.

In the study of associative algebras and commutative associative algebras,
the Hochschild cohomologies and the Harrison cohomologies play
important roles. Let $A$ be an associative algebra. For an $A$-bimodule 
$M$,
we use $\hat{H}^{n}(A, M)$ to denote the
$n$-th Hochschild cohomology of $A$ with
coefficients in $M$. When $A$ is commutative and $M$ is a module (viewed as
$A$-bimodule with the left and right $A$-module structures to be both 
the one
from the original $A$-module structure), we use $H^{n}(A, M)$ to denote the
the $n$-th Harrison cohomology of $A$ with coefficients in $M$.
Then the following results are well-known:
\begin{enumerate}

\item The first Hochschild cohomology $\hat{H}^1(A, M)$ is isomorphic to 
the
quotient of the space of derivations from $A$ to $M$ by the space of
inner derivations from $A$ to $M$. When $A$ is commutative, the first
Harrison cohomology $H^1(A, M)$ is isomorphic to the space of derivations from $A$ to $M$.

\item The second Hochschild cohomology $\hat{H}^2(A, A)$ is in one-to-one
correspondence with the set of first-order deformations of $A$.
When $A$ is commutative,
the second Harrison cohomology $H^2(A, A)$ is in one-to-one
correspondence with the set of the first-order deformations of $A$.

\item All the left $A$-modules are completely reducible if and only if
for every $A$-bimodules $B$ and every $n\in \Z_+$, the Hochschild
cohomology $\hat{H}^n(A, B)= 0$.
\end{enumerate}

In \cite{H1st-sec-coh}, using the cohomology established in \cite{Hcoh}, Huang established the analogues of the results (1)
and (2) for a grading-restricted vertex algebra $V$ and
grading-restricted $V$-modules. To define this cohomology, Huang introduced a larger complex in \cite{Hcoh}
such that the complex for the grading-restricted
vertex algebra is a subcomplex,
just as the Harrison complex is a subcomplex of the Hochschild complex
for the commutative associative algebra.
In particular, the larger complex can be viewed as the analogue of
the Hochschild complex. But this complex was defined in \cite{Hcoh}
only for a grading-restricted vertex algebra. 

In this paper, we give the definition of this larger complex for meromorphic open-string vertex algebras $V$ and $V$-bimodules $W$ that are not necessarily grading-restricted but satisfy the pole-order condition. The pole-order was introduced in \cite{Q-Rep} and is satisfied by all the existing examples of MOSVAs and modules. Using the cohomology of this larger complex, we can establish the analogues of results (1), (2) and (3). The results (1) and (3) will be presented in \cite{HQ-Red}. The result (2) will be presented in future papers.  

% In \cite{H1st-sec-coh}, Huang established the analogues of results (1)
% and (2) for a grading-restricted vertex algebra $V$ and
% grading-restricted $V$-modules. Just as the correct cohomology of
% any algebra, the first-order deformations for a grading-restricted
% vertex algebra $V$ given by elements of  Huang's second
% cohomology are still vertex algebras up to the first order. Thus
% this cohomology theory might be used to understand and study the
% deformation quantization of conformal field theories.

% This paper generalizes Huang's cohomology theory for grading-restricted
% vertex algebras to a cohomology theory for MOSVAs. The analogue of the
% result (1) will be given in \cite{HQ}
% and the analogue of result (2) will be presented in a future paper.
% The bimodules  that we consider in this paper that are not necessarily
% grading-restricted.
% The theory, at least the first cohomology part, is needed in the joint
% work \cite{HQ} by Huang and the author (announced first in 2015)
% on the relation between the first cohomology and the complete
% reducibility of left modules for a MOSVA. Since grading-restricted vertex
% algebras are MOSVAs, the results in \cite{HQ} also hold for 
% grading-restricted
% vertex algebras and thus give a criterion of the reductivity of the
% grading-restricted vertex algebra. The proofs of the results in \cite{HQ}
% depends heavily on the construction of a bimodule
% that in general is not necessarily
% grading restricted. Thus we have to develop the cohomology theory
% for bimodules that might not be grading restricted.

The paper is organized as follows:

In Section 2, we review the notions of the MOSVAs and modules
defined in \cite{H-MOSVA} and \cite{Q-Rep}. %We will restrict our attention to the MOSVAs and modules satisfying the pole-order condition in \cite{Q-Rep}, as it is satisfied by all the existing examples. 
Since the opposite vertex operator map of the vertex operator map
for a right module defined in \cite{Q-Rep} will be used extensively
in this paper, we also discuss some of its properties that are not 
presented
in \cite{Q-Rep}. Analytic continuation is the main technical part of this study. We will often quote those
lemmas in \cite{Q-Rep} on the analytic extensions in this section.

In Section 3, for a MOSVA $V$ and a $V$-bimodule $W = \coprod\limits_{n\in \C} W_{[n]}$, we discuss $\overline{W}$-valued rational functions, where $\overline{W}= \prod\limits_{n\in \C} W_{[n]}$ is the algebraic completion of $W$. %We restrict our attention to the $\overline{W}$-valued rational function that are sums of series with coefficients in $W$. Though it is possible to develop the theory with more general $(W')^*$-valued rational function, the resulting cohomology is different and might not be the correct one. 
We will also study series of $\overline{W}$-valued rational functions and prove that the associativity (of $Y_W^L$ and of $Y_W^{s(R)}$) and commutativity (of $Y_W^L$ and $Y_W^{s(R)}$) hold when acting on $\overline{W}$-valued rational functions satisfying certain convergence conditions.

In Section 4, we first study the maps from $V^{\otimes n}$ to
the space $\widetilde{W}_{z_1, ..., z_n}$ of $\overline{W}$-valued
rational function in $z_1, ..., z_n$ with the only possible poles at
$z_i = z_j$. Such maps satisfying $\d$-conjugation properties,
$D$-derivative properties and composable condition are used to
construct the cochain complex of the cohomology in this section. Our
composable condition is formulated differently from that in \cite{Hcoh},
but it can be shown that they are equivalent by analytic continuation.
The coboundary operators for the cochain complex is defined using
the $\overline{W}$-valued rational functions
that the relevant series converge to. These series might not have a
common region of convergence. But since the rational functions
that they converge to have the same domain, the coboundary
operator is well defined. This is the key (as observed by Huang)
for the cohomology theory to work.

\noindent \textbf{Acknowledgement.} The author is very grateful to Yi-Zhi Huang, who patiently discussed numerous technical details, provided lots insightful observations and corrected several mistakes in the earlier version of the paper.  

\section{Definitions and Immediate Consequences}

\subsection{Definition of the MOSVA}

We first recall the notions of meromorphic open-string vertex algebra and its modules given in \cite{H-MOSVA} and \cite{Q-Rep}. 

\begin{defn}\label{DefMOSVA}
{\rm A {\it meromorphic open-string vertex algebra} (hereafter MOSVA) is a $\Z$-graded vector space 
$V=\coprod_{n\in\Z} V_{(n)}$ (graded by {\it weights}) equipped with a {\it vertex operator map}
\begin{eqnarray*}
   Y_V:  V\otimes V &\to & V[[x,x^{-1}]]\\
	u\otimes v &\mapsto& Y_V(u,x)v%=\sum_{n\in \Z} u_n v x^{-n-1},
  \end{eqnarray*}
and a {\it vacuum} $\one\in V$, satisfying the following axioms:
\begin{enumerate}
\item Axioms for the grading:
\begin{enumerate}
\item {\it Lower bound condition}: When $n$ is sufficiently negative,
$V_{(n)}=0$.
\item {\it $\d$-commutator formula}: Let $\d_{V}: V\to V$
be defined by $\d_{V}v=nv$ for $v\in V_{(n)}$. Then for every $v\in V$
$$[\d_{V}, Y_{V}(v, x)]=x\frac{d}{dx}Y_{V}(v, x)+Y_{V}(\d_{V}v, x).$$
\end{enumerate}
%\item {\it Lower-truncation condition for vertex operators}:
%For $u, v\in V$, $Y_{V}(u, x)v$ contain only finitely many negative
%power terms, that is, $Y_{V}(u, x)v\in V((x))$ (the space of formal
%Laurent series in $x$ with coefficients in $V$ and with finitely
%many negative power terms).

\item Axioms for the vacuum: 
\begin{enumerate}
\item {\it Identity property}: Let $1_{V}$ be the identity operator on $V$. Then
$Y_{V}(\mathbf{1}, x)=1_{V}$. 
\item {\it Creation property}: For $u\in V$, $Y_{V}(u, x)\mathbf{1}\in V[[x]]$ and 
$\lim_{z\to 0}Y_{V}(u, x)\mathbf{1}=u$.
\end{enumerate}

\item {\it $D$-derivative property and $D$-commutator formula}:
Let $D_V: V\to V$ be the operator
given by
$$D_{V}v=\lim_{x\to 0}\frac{d}{dx}Y_{V}(v, x)\one$$
for $v\in V$. Then for $v\in V$,
$$\frac{d}{dx}Y_{V}(v, x)=Y_{V}(D_{V}v, x)=[D_{V}, Y_{V}(v, x)].$$

\item {\it Rationality}: Let $V'=\coprod_{n\in \Z}V_{(n)}^*$ be the graded dual of $V$. 
For $u_1, \cdots, u_n, v\in V, v'\in V'$, the series
$$\langle v', Y_V(u_1, z_1)\cdots Y_V(u_n, z_n)v\rangle$$
converges absolutely when $|z_1|>\cdots>|z_n|>0$ to a rational function in $z_1,\cdots, z_n$, 
with the only possible poles at $z_i=0, i=1, ... , n$ and $z_i=z_j, 1\leq i\neq j\leq n$. 
For $u_1, u_2, v\in V$ and $v' \in V'$, the series
$$\langle v', Y_V(Y_V(u_1,z_1-z_2)u_2,z_2)v\rangle$$
converges absolutely when $|z_2|>|z_1-z_2|>0$ to a rational function with the only possible poles at 
$z_1=0, z_2=0$ and $z_1=z_2$.

\item {\it Associativity}: For $u_{1}, u_{2}, v\in V$ and
$v'\in V'$, we have 
$$\langle v', Y_{V}(u_{1}, z_{1})Y_{V}(u_{2}, z_{2})v\rangle=\langle v', Y_{V}(Y_{V}(u_{1}, z_{1}-z_{2})u_{2}, z_{2})v\rangle$$
when $|z_{1}|>|z_{2}|>|z_{1}-z_{2}|>0$.
\end{enumerate}  }
Such a meromorphic open-string vertex algebra is denoted by $(V, Y_V, \one)$ or simply by 
$V$. 
\end{defn}

% \begin{defn}
% A meomorphic open-string vertex algebra $V$ is said to be \textit{grading-restricted} if 
% $\dim V_{(n)}<\infty$ for $n\in \Z$.
% \end{defn}

\begin{defn}
Let $V$ be a MOSVA. We say that $V$ satisfies the \textit{pole-order condition} if for every $u_1, v\in V$, there exists $C>0$, such that for every $v'\in V', u_2\in V$, the pole $z_1=0$ of the rational function determined by
$$\langle v', Y_V(u_1, z_1)Y_V(u_2, z_2)v\rangle$$
has order less than $C$. In other words, the order of the pole in the rational function is bounded above by a number that depends only on the choice of $u_1$ and $v$. 
\end{defn}

\begin{rema}\label{RemaPoleCond}
If $V$ has the pole-order condition for two vertex operators, then it is proved in \cite{Q-Rep} that for every $v'\in V', u_1, ..., u_n, v\in V$ and for the rational function determined by 
$$\langle v', Y_V(u_1, z_1)\cdots Y_V(u_n, z_n)v\rangle$$
the order of the pole $z_i=0$ is bounded above by a constant that depends only on $u_i$ and $w$, $i=1, ..., n$; and the order of the pole $z_i= z_j$ is bounded above by a constant that depends only on $u_i$ and $u_j$, $1\leq i < j \leq n$. 
\end{rema}

\subsection{Left modules for MOSVAs}

The notion of left modules for a meromorphic open-string vertex algebra was introduced in \cite{H-MOSVA}. The philosophy is similar to the modules for vertex algebras: all the defining properties of a MOSVA that make sense hold. 

\begin{defn}\label{DefMOSVA-L}
Let $(V, Y_{V}, \one)$ be a meromorphic open-string vertex algebra.
A \textit{left $V$-module} is a $\C$-graded vector space 
$W=\coprod_{m\in \C}W_{[m]}$ (graded by \textit{weights}), equipped with 
a \textit{vertex operator map}
\begin{eqnarray*}
Y_W^L: V\otimes W & \to & W[[x, x^{-1}]]\\
u\otimes w & \mapsto & Y_W^L(u, x)w,
\end{eqnarray*}
an operator $\d_{W}$ of weight $0$ and 
an operator $D_{W}$ of weight $1$, satisfying the 
following axioms:
\begin{enumerate}

\item Axioms for the grading: 
\begin{enumerate}
\item \textit{Lower bound condition}:  When $\text{Re}{(m)}$ is sufficiently negative,
$W_{[m]}=0$. 
\item  \textit{$\mathbf{d}$-grading condition}: for every $w\in W_{[m]}$, $\d_W w = m w$.
\item  \textit{$\mathbf{d}$-commutator formula}: For $u\in V$, 
$$[\mathbf{d}_{W}, Y_W^L(u,x)]= Y_W^L(\mathbf{d}_{V}u,x)+x\frac{d}{dx}Y_W^L(u,x).$$
\end{enumerate}

\item The \textit{identity property}:
$Y_W^L(\one,x)=1_{W}$.

\item The \textit{$D$-derivative property} and the  \textit{$D$-commutator formula}: 
For $u\in V$,
\begin{eqnarray*}
\frac{d}{dx}Y_W^L(u, x)
&=&Y_W^L(D_{V}u, x) \\
&=&[D_{W}, Y_W^L(u, x)].
\end{eqnarray*}

\item \textit{Rationality}: For $u_{1}, \dots, u_{n}\in V, w\in W$
and $w'\in W'$, the series 
$$
\langle w', Y_W^L(u_{1}, z_1)\cdots Y_W^L(u_{n}, z_n)v\rangle
$$
converges absolutely 
when $|z_1|>\cdots >|z_n|>0$ to a rational function in $z_{1}, \dots, z_{n}$
with the only possible poles at $z_{i}=0$ for $i=1, \dots, n$ and $z_{i}=z_{j}$ 
for $i\ne j$. For $u_{1}, u_{2}\in V, w\in W$
and $w'\in W'$, the series 
$$
\langle w', Y_W^L(Y_{V}(u_{1}, z_1-z_{2})u_{2}, z_2)v\rangle
$$
converges absolutely when $|z_{2}|>|z_{1}-z_{2}|>0$ to a rational function
with the only possible poles at $z_{1}=0$, $z_{2}=0$ and $z_{1}=z_{2}$. 

\item \textit{Associativity}: For $u_{1}, u_{2}\in V, w\in W$, 
$w'\in W'$, 
$$
\langle w', 
Y_W^L(u_{1},z_1)Y_W^L(u_{2},z_2)v\rangle
=
\langle w', 
Y_W^L(Y_{V}(u_{1},z_{1}-z_{2})u_{2},z_2)v\rangle
$$
when  $|z_{1}|>|z_{2}|>|z_{1}-z_{2}|>0$. 
\end{enumerate} 
We denote the left $V$-module just defined by $(W, Y_W^L, \d_{W}, D_{W})$ or simply $W$ when there is no confusion. 
\end{defn}

% \begin{defn}
% A left $V$-module is said to be \textit{grading-restricted} if 
% $\dim W_{[m]}<\infty$ for every $m\in \C$. 
% \end{defn}

% \begin{rema}
% Throughout the discussion in this paper, we will \textit{not} assume the $V$-modules to be grading-restricted. 
% \end{rema}

% \begin{prop}\label{ImmediateFacts-L}
% Let $V$ be a MOSVA and $W$ be a left $V$-module. then
% \begin{enumerate}
% \item For $u\in V$, $Y_W^L(u,x)$ can be regarded as a formal series in $\text{End}(W)[[x,x^{-1}]]$
% $$Y_W^L(u, x)= \sum_{n\in\Z} (Y_W^L)_n(u) x^{-n-1}$$
% where $(Y_W^L)_n(u):W \to W$ is a linear map for every $n\in \Z$. If $u$ is homogeneous, then $(Y_W^L)_n(u)$ is a map of weight $\wt u - n- 1$. 
% \item For fixed $u\in V, w\in W$, $Y_W^L(u, x)w$ is lower truncated, i.e, there are at most finitely many negative powers of $x$. 
% \item $D$-conjugation property: for $u\in V$, 
% $$Y_W^L(u, x+y) = Y_V(e^{yD_V}u, x) = e^{yD_W} Y_W^L(u, x) e^{-yD_W},$$
% in $\text{End}(W)[[x,y, x^{-1}]]$. 
% \item $\d$-conjugation property: for $u\in V$, 
% $$  e^{y\d_W} Y_W^L(u, x) e^{-y \d_W} = Y_W^L(e^{y\d_W} u, xy)$$
% in $\text{End}(W)[[x,x^{-1}, y, y^{-1}]]$. 
% \end{enumerate}
% \end{prop}
% \begin{proof}
% Similar to the argument of Proposition \ref{ImmediateFacts}. For the second statement, use of the fact that that $W_{[m]}= 0$ when Re $m << 0$. 
% \end{proof}

%================================================================

\subsection{Right modules for MOSVAs}

\begin{defn}
Let $(V, Y_{V}, \one)$ be a meromorphic open-string vertex algebra.
A \textit{right module for $V$} %or a \textit{M\"obius right $V$-module} 
is a $\C$-graded vector space 
$W=\coprod_{m\in \C}W_{[m]}$ (graded by \textit{weights}), equipped with 
a \textit{vertex operator map}
\begin{eqnarray*}
Y_W^R: W\otimes V&\to& W[[x, x^{-1}]]\\
w\otimes u&\mapsto& Y_W^R(w, x)u,
\end{eqnarray*}
an operator $\d_{W}$ and 
an operator $D_{W}$ of weight $1$, satisfying the 
following axioms:
\begin{enumerate}

\item Axioms for the grading
\begin{enumerate}
\item \textit{Lower bound condition}:  When $\text{Re }{m}$ is sufficiently negative, $W_{[m]}=0$. 
\item \textit{$\mathbf{d}$-grading condition}: for every $w\in W_{[m]}, \d_W w = m w$.
\item \textit{$\d$-commutator formula}: For $w\in W$, 
$$\d_{W}Y_W^R(w,x)-Y_W^R(w,x)\d_{V}= Y_W^R(\d_{W}w,x)+x\frac{d}{dx}Y_W^R(w,x).$$
\end{enumerate}

\item The \textit{Creation property}: For $w\in W$, $Y_W^R(w,x)\one\in W[[x]]$ and 
$\lim\limits_{x\to 0}Y_W^R(w,x)\one=w$.

\item The \textit{$D$-derivative property} and the  \textit{$D$-commutator formula}: 
For $u\in V$,
\begin{eqnarray*}
\frac{d}{dx}Y_W^R(w, x)
&=&Y_W^R(D_{W}w, x) \\
&=&D_{W}Y_W^R(w, x)-Y_W^R(w, x)D_{V}.
\end{eqnarray*}

\item \textit{Rationality}: For $u_{1}, \dots, u_{n}\in V, w\in W$
and $w'\in W'$, the series 
$$
\langle w', Y_W^R(w, z_1)Y_{V}(u_{1}, z_2)\cdots Y_{V}(u_{n-1}, z_n)u_{n}\rangle
$$
converges absolutely 
when $|z_1|>\cdots >|z_n|>0$ to a rational function in $z_{1}, \dots, z_{n}$
with the only possible poles at $z_{i}=0$ for $i=1, \dots, n$ and $z_{i}=z_{j}$ 
for $i\ne j$. For $u_{1}, u_{2}\in V, w\in W$
and $w'\in W'$, the series 
$$
\langle w', Y_W^R(Y_W^R(w, z_1-z_{2})u_{1}, z_2)u_{2}\rangle
$$
converges absolutely when $|z_{2}|>|z_{1}-z_{2}|>0$ to a rational function
with the only possible poles at $z_{1}=0$, $z_{2}=0$ and $z_{1}=z_{2}$. 

\item \textit{Associativity}: For $u_{1}, u_{2}\in V, w\in W$, 
$w'\in W'$, 
$$
\langle w', 
Y_W^R(w,z_1)Y_{V}(u_{1},z_2)u_{2}\rangle
=
\langle w', 
Y_W^R(Y_W^R(w,z_{1}-z_{2})u_{1},z_2)u_{2}\rangle
$$
when  $|z_{1}|>|z_{2}|>|z_{1}-z_{2}|>0$. 
\end{enumerate} 

% A right $V$-module is said to be \textit{grading-restricted} if $\dim W_{[n]}<\infty$ for $n\in \C$. %A M\"obius right $V$-module is said to be 
%a \textit{right $V$-module} if $W_{[n]}$ for $n\in \C$ are eigenspaces of $\d_{W}$. }

When there is no confusion, we also denote the right $V$-module just defined by $(W, Y_W^R, \d_{W}, D_{W})$ or simply $W$. 
\end{defn}

% \begin{prop}\label{ImmediateFacts-R}
% Let $V$ be a MOSVA and $W$ be a right $V$-module. then
% \begin{enumerate}
% \item For $u\in V$, $Y_W^R(\cdot,x)u$ can be regarded as a formal series in $\text{End}(W)[[x,x^{-1}]]$
% $$Y_W^R(\cdot, x)u= \sum_{n\in\Z} (Y_W^R)_n(u) x^{-n-1}$$
% where $(Y_W^R)_n(u): W \to W$ is a linear map for every $n\in \Z$. If $u$ is homogeneous, then $(Y_W^R)_n(u)$ is a map of weight $\wt u - n- 1$. 
% \item For fixed $u\in V$ and $w\in W$, $Y_W^R(w, x)u$ is lower truncated, i.e, there are at most finitely many negative powers of $x$. 
% \item For $w\in W$, 
% $$Y_W^R(w, x)\one = e^{xD_W} w$$ 
% \item $D$-conjugation property: for $w\in V$, 
% $$Y_W^R(w, x+y) = Y_W^R(e^{yD_W}w, x) = e^{yD_W} Y_W^R(w, x) e^{-yD_V},$$
% in $\text{End}(W)[[x,y, x^{-1}]]$. 
% \item $\d$-conjugation property: for $u\in V$, 
% $$  e^{y\d_W} Y_W^R(w, x) e^{-y \d_V} = Y_W^R(e^{y\d_W} w, xy)$$
% in $\text{End}(W)[[x,x^{-1}, y, y^{-1}]]$. 
% \end{enumerate}
% \end{prop}
% \begin{proof}
% The arguments for (1), (2), (4) and (5) are similar to those for Proposition \ref{ImmediateFacts}. To see (3), one first note that 
% $$D_W w = \lim_{x\to 0} Y_W^R(D_Ww, x)\one = \lim_{x\to 0} \frac d{dx} Y_W^R(w, x)$$
% (the first equality follows from the creation property, the second from $D$-derivative property), then apply the arguments in Proposition \ref{ImmediateFacts}. 
% \end{proof}

%================================================================

\subsection{Bimodules for MOSVAs}

\begin{defn}
Let $(V, Y_V, \one)$ be a meromorphic open-string vertex algebra. A \textit{$V$-bimodule} is a vector space equipped with a left $V$-module structure and right $V$-module structure such that these two strutcure are compatible. 
More precisely, a \textit{$V$-bimodule} is a $\C$-graded vector space 
$$W=\coprod_{n\in \C}W_{[n]}$$
equipped with a \textit{left vertex operator map}
\begin{eqnarray*}
Y_{W}^{L}: V\otimes W&\to& W[[x, x^{-1}]]\\
u\otimes w&\mapsto& Y_{W}^{L}(u, x)v,
\end{eqnarray*}
a \textit{right vertex operator map}
\begin{eqnarray*}
Y_{W}^{R}: W\otimes V&\to& W[[x, x^{-1}]]\\
w\otimes u&\mapsto& Y_{W}^{R}(w, x)u,
\end{eqnarray*}
and linear operators $d_W, D_W$ on $W$ satisfying the following conditions.
\begin{enumerate}

 \item $(W, Y_W^L, \d_W, D_W)$ is a left $V$-module.

 \item $(W, Y_W^R, \d_W, D_W)$ is a right $V$-module.

\item \textit{Compatibility}:
\begin{enumerate}
\item 
\textit{Rationality of left and right vertex operator maps}: For $u_1, ..., u_n, u_{n+1}, ..., u_{n+m}\in V$, $w\in W$, the series
$$\langle w', Y_W^L(u_1, z_1) \cdots Y_W^L(u_n, z_n) Y_W^R(w, z_{n+1}) Y_V(u_{n+1}, z_{n+2})\cdots Y_V(u_{n+m-1}, z_{n+m})u_{n+m}\rangle$$
converges absolutely in the region $|z_1|>|z_2|>\cdots >|z_n| > |z_{n+1}| > \cdots > |z_{n+m}|> 0$ to a rational function in $z_1, ..., z_n, z_{n+1}, ..., z_{n+m}$. 
%\textcolor{red}{Is it necessary or do we just need to check for two vertex operators? I believe it follows from the associativity. A big argument is needed though. }
\item \textit{Associativity for left and right vertex operator maps}: For $u, v\in V$, $w\in W$ and $w'\in W'$, the series
$$\langle w', Y_W^L(u, z_1)Y_W^R(w,z_2)v\rangle$$
$$\langle w', Y_W^R(Y_W^L(u,z_1-z_2)w, z_2)v\rangle$$
converges absolutely in the region $|z_1|>|z_2|>0$ and $|z_2|>|z_1-z_2|>0$,  respectively, to a common rational function in 
$z_1$ and $z_2$ with the only possible poles at $z_1, z_2=0$ and $z_1 = z_2$.
\end{enumerate}
\end{enumerate}
\end{defn}

The $V$-bimodule just defined is denoted 
by  $(W, Y_W^L, Y_W^R, \d_W, D_W)$ or simply by $W$ when there is no confusion. 

\begin{defn}
Let $V$ be a MOSVA that satisfies the pole-order condition. Let $W$ be a $V$-bimodule. We say $W$ satisfies the \textit{pole-order condition} if
\begin{enumerate}
\item $W$ satisfies the following pole-order condition as a left $V$-module: for every $u_1\in V, w\in W$, there exists $C>0$, such that for every $w'\in W', u_2\in V$, the pole $z_1=0$ of the rational functions determined by
$$\langle w', Y_W^L(u_1, z_1)Y_W^L(u_2, z_2)w\rangle$$
has order less than $C$. 
\item $W$ satisfies the following pole-order condition as a right $V$-module: for every $u_2\in V, w\in W$, there exists $C>0$, such that for every $w'\in W', u_1\in V$, the pole $z_1=0$ of the rational functions determined by
$$\langle w', Y_W^R(w, z_1)Y_V(u_1, z_2)u_2\rangle$$
has order less than $C$. 
\item For every $u_1, u_2\in V$, there exists $C>0$, such that for every $w'\in W', w\in W$, the pole $z_1=0$ of the rational functions determined by
$$\langle w', Y_W^L(u_1, z_1)Y_W^R(w, z_2)u_2\rangle$$
has order less than $C$. 
\end{enumerate}
\end{defn}

\begin{rema}
Similarly as Remark \ref{RemaPoleCond}, the pole-order condition for two vertex operators implies the pole-order condition for any numbers of vertex operators. 
\end{rema}

\begin{rema}
We will build the cohomology theory for MOSVAs and bimodules satisfying the pole-order condition. Though a cohomology theory can also be built without this condition, the resulting theory is different and has no applications at this moment. Hence we choose not to do that here.  
\end{rema}

%================================================================

\subsection{Bimodules in terms of opposite algebras}

Recall that in \cite{Q-Rep}, we proved that for a MOSVA $(V, Y_V, \one)$, the space $V$ with the following vertex operator
\begin{align*}
Y_V^s: V\otimes V &\to V[[x, x^{-1}]]\\
Y_V^s(u, x)v & =  e^{xD_V} Y_V(v, -x)u
\end{align*}
and the vacuum $\one \in V$ also forms a MOSVA, called the opposite MOSVA of $V$ and denoted $V^{op}$. We also proved that a right $V$-module $(W, Y_W^R, \d_W, D_W)$ is equivalent to a left $V^{op}$-module $(W, Y_W^{s(R)}, \d_W, D_W)$, where $Y_W^{s(R)}$ is defined by
$$Y_W^{s(R)}(v, x)w = e^{xD_W}Y_W^R(w, -x)v.$$

% For convenience, we list some properties of the $Y_W^{s(R)}$ here. 

% \begin{prop}\label{ImmediateFacts}
% Let $V$ be a MOSVA and $W$ be a right $V$-module. Then
% \begin{enumerate}
% \item For $u\in V$, $Y_W^{s(R)}(u,x)$ can be regarded as a formal series in $\text{End}(W)[[x,x^{-1}]]$
% $$Y_W^{s(R)}(u, x)= \sum_{n\in\Z} (Y_W^{s(R)})_n(u) x^{-n-1}$$
% where $(Y_W^{s(R)})_n(u): V \to V$ is a linear map for every $n\in \Z$. If $u$ is homogeneous, then $(Y_W^{s(R)})_n(u)$ is a map of weight $\wt u - n- 1$. 
% \item For fixed $u, v\in V$, $Y_V(u, x)v$ is lower truncated, i.e, there are at most finitely many negative powers of $x$. 
% \item $D$-conjugation property: for $u\in V$, 
% $$Y_W^{s(R)}(u, x+y) = Y_W^{s(R)}(e^{yD_V}u, x) = e^{yD_W} Y_W^{s(R)}(u, x) e^{-yD_W},$$
% in $\text{End}(V)[[x,x^{-1}, y]]$. 
% \item $\d$-conjugation property: for $u\in V$, 
% $$  e^{y\d_V} Y_W^{s(R)}(u, x) e^{-y \d_V} = Y_W^{s(R)}(e^{y\d_V} u, xy)$$
% in $\text{End}(V)[[x,x^{-1}, y, y^{-1}]]$.
% \end{enumerate}
% \end{prop}

% \begin{proof}
% It follows from Proposition *.*.* in \cite{Q-Rep}. 
% \end{proof}

\begin{thm}
Let $W$ be a $V$-bimodule. Then the compatibility condition can be formulated in terms of $Y_W^L$ and $Y_W^{s(R)}$ as follows
\begin{enumerate}
\item For every $n\in \Z_+$, $l=1, ..., n$, $u_1, ..., u_n\in V$, $w\in W$, $w'\in W'$,
$$\langle w', Y_W^L(u_1, z_1)\cdots Y_W^L(u_l, z_l)Y_W^{s(R)}(u_{l+1}, z_{l+1})\cdots Y_W^{s(R)}(u_n, z_n)w\rangle$$
converges absolutely when $|z_1|>\cdots > |z_n|>0$  to a rational function with the only possible poles at $z_i=0, i=1, ...,n$ and $z_i = z_j, 1\leq i < j\leq n$. 
\item For every $u_1, u_2\in V, w\in W, w'\in W'$, 
$$\langle w', Y_W^L(u_1, z_1)Y_W^{s(R)}(u_2, z_2)w\rangle$$
$$\langle w', Y_W^{s(R)}(u_2, z_2)Y_W^L(u_1, z_1)w\rangle$$
converges absolutely to a common rational function in the region $|z_1|>|z_2|>0$ and $|z_2|>|z_1|>0$, respectively. 
\end{enumerate}
\end{thm}

\begin{proof}
We only give a sketch here. From the compatibility condition of $Y_W^L$ and $Y_W^R$, 
$$\langle w', Y_W^L(u_1, z_1-z_{l+1})\cdots Y_W^L(u_l, z_l - z_{l+1}) Y_W^R(w, -z_{l+1}) Y_V(u_n, -z_{l+1}+z_n) \cdots Y_V(u_{l+2}, -z_{l+1}+z_{l+2})u_{l+1}\rangle$$
converges absolutely when 
$$|z_1-z_{l+1}|>\cdots > |z_l-z_{l+1}| > |z_{l+1}| > |z_{l+1}-z_n|> \cdots > |z_{l+1}-z_{l+2}| > 0$$ to a rational function with the only possible poles at $z_i = 0, i=1, ..., n$ and $z_i= z_j, 1\leq i < j \leq n$. Then one repeatedly uses associativity and Lemma 4.7 in \cite{Q-Rep} to argue that
$$\langle w', Y_W^L(u_1, z_1-z_{l+1})\cdots  Y_W^L(u_l, z_l - z_{l+1}) Y_W^R(\cdots Y_W^R(Y_W^R(w, -z_n)u_{n}, -z_{n-1}+z_n)u_{n-1}, \cdots, -z_{l+1}+z_{l+2})u_{l+1}\rangle$$
converges absolutely when 
\begin{align*}
|z_1-z_{l+1}|>\cdots > |z_l - z_{l+1}| > |z_{l+1}-z_{l+2}|+ \cdots + |z_{n-1}-z_n|+|z_n|;\\ 
|z_i - z_{i+1}|>|z_{i+1}- z_{i+2}|+ \cdots + |z_{n-1}-z_n| + |z_n| > 0, i = 1, ..., n-1.
\end{align*}
to the same rational function. If we further expand the negative powers of $z_i - z_{l+1}$ as a power series in $z_{l+1}$ for $i=1, ..., l$, and further expand the negative powers of $-z_i + z_{i+1}$ as a power series $z_{i+1}$ for $i = l+1, ..., n-1$, the resulting series in $z_1, ..., z_n$ is precisely 
$$\langle w', e^{-z_{l+1}D_W} Y_W^L(u_1, z_1)\cdots Y_W^L(u_l, z_l)e^{z_{l+1}D_W}Y_W^{R}(\cdots e^{z_n D_W} Y_W^R(w, -z_n)u_n, \cdots, -z_{l+1})u_{l+1}\rangle$$
One can repeatedly use Lemma 4.5 and Lemma 4.7 in \cite{Q-Rep} that this series converges absolutely when 
$$|z_1|>\cdots > |z_n|> 0$$
to the same rational function. Thus we proved that the series
$$\langle w', e^{-z_{l+1}D_W} Y_W^L(u_1, z_1)\cdots Y_W^L(u_l, z_l) Y_W^{s(R)}(u_{l+1}, z_{l+1})\cdots Y_W^{s(R)}(u_n, z_n)w\rangle$$
converges absolutely when $|z_1|> \cdots > |z_n|> 0$. The conclusion of (1) then follows from Remark 5.3 in \cite{Q-Rep}, which allows us to apply another $e^{z_{l+1}D_W}$ to the front and keep the convergence (though the rational function might change). 

For (2), note that 
$$\langle w', e^{z_2D_W} Y_W^L(u_1, z_1-z_2)Y_W^R(w, -z_2)u_2\rangle = \langle w', e^{z_2D_W} Y_W^R(Y_W^L(u_1, z_1)w, -z_2)u_2$$
when $|z_1-z_2|>|z_2|>|z_1|>0$. Both sides converge to the same rational function. If the negative powers of $z_1-z_2$ in the series on the left-hand-side are expanded as a power series in $z_2$, then the resulting series is precisely $\langle w', Y_W^L(u_1, z_1)Y_W^{s(R)}(u_2, z_2)w\rangle$ and converges absolutely in the region $|z_1|>|z_2|>0, |z_1-z_2|>0$ to the same rational function. We then use Lemma 4.5 in \cite{Q-Rep} to see that $\langle w', Y_W^L(u_1, z_1)Y_W^{s(R)}(u_2, z_2)w\rangle$ converges absolute when $|z_1|>|z_2|>0$ to the same rational function as the right-hand-side, while the right-hand-side is precisely $\langle w', Y_W^s(R)(u_2, z_2)Y_W^L(u_1, z_1)w\rangle$. Thus the conclusion is proved. 
\end{proof}

\begin{rema}
The pole-order condition can also be expressed in terms of $Y_W^{s(R)}$. More precisely, if $V$ is a MOSVA and $W$ is a $V$-bimodule, both of which satisfy the pole-order condition, then
\begin{enumerate}
\item For every $u_1 \in V, w\in W$, there exists $C>0$ such that for every $w'\in W', u_2\in V$, the pole $z_1=0$ of the rational function determined by 
$$\langle w', Y_W^{s(R)}(u_1, z_1)Y_W^{s(R)}(u_2, z_2)w\rangle$$
has order less than $C$. In fact, $C$ can be chosen to be the same upper bound of the order pole $z_1=0$ for $\langle w', Y_W^R(w, z_1)Y_V(u_2, z_2)u_1\rangle$.
\item For every $u_1, u_2\in V$, there exists $C>0$ such that for every $w'\in W', w\in W$, the pole $z_1=z_2$ of the rational function determined by 
$$\langle w', Y_W^{L}(u_1, z_1)Y_W^{s(R)}(u_2, z_2)w\rangle$$
has order less than $C$. In fact, $C$ can be chosen to be the same upper bound of the order pole $z_1=0$ for $\langle w', Y_W^L(u_1, z_1)Y_W^R(w, z_2)u_2\rangle$.
\end{enumerate}
\end{rema}

\begin{rema}\label{RemaPoleCond-1}
Similarly as Remark \ref{RemaPoleCond}, one can prove that for the rational function determined by 
$$\langle w', Y_W^L(u_1, z_1)\cdots Y_W^L(u_l, z_l)Y_W^{s(R)}(u_{l+1}, z_{l+1})\cdots Y_W^{s(R)}(u_n, z_n) w\rangle$$
the order of the pole $z_i=0$ is bounded above by a constant that depends only on $u_i$ and $w$, $i=1, ..., n$; and the order of the pole $z_i= z_j$ is bounded above by a constant that depends only on $u_i$ and $u_j$, $1\leq i < j \leq n$. 
\end{rema}

% \textcolor{red}{The result here is definitely necessary. But the exposition here is way too sketchy. In the thesis, we can certainly put everything in Chapter 3 without reintroducing the complex analysis. However in the paper, how to deal with these convergence becomes a problem.\\
% Maybe we can redirect the reader to the thesis for complete details. }

%================================================================

\section{$\overline{W}$-valued rational functions}

Throughout this section, $V$ is a MOSVA; $W = \coprod_{n\in \C} W_{[n]}$ is a $V$-bimodule that is not necessarily grading-restricted; $W'=\coprod_{n\in \C} W_{[n]}^*$ is the graded dual of $W$. We shall assume that all the pole-order conditions hold for $V$ and $W$. 

We use $\overline{W}$ to denote the algebraic completion $\prod_{n\in \C} W_{[n]}$ of $W$. Note that the dual $(W')^*$ of $W'$ does not coincide with $\overline{W}$. Also note that any homogeneous linear map $L: W\to W$ extends to a map $\overline{W}\to\overline{W}$ by the formal linearity
$$L(\overline{w})=L\left(\sum_{k\in \C} \pi_k \overline{w}\right)=\sum_{k\in\C} L(\pi_k \overline{w}) $$
where $\pi_k$ is the projection of $W$ onto $W_{[k]}$. More generally, any linear map $L: W \to W$ that is a finite linear combination of homogeneous linear maps can be extended to $L: \overline{W} \to \overline{W}$. For convenience, we will not introduce new notations to distinguish the extended map from the original map.

\subsection{$\overline{W}$-valued rational functions}

\begin{defn}\label{arcWRat} 
For $n\in \Z_+$, we consider the configuration space
$$F_n \C = \{(z_1, ..., z_n)\in \C^n: z_i \neq z_j, i\neq j\}$$
A \textit{$\overline{W}$-valued rational function in $z_1, ..., z_n$ with the only possible poles at $z_i = z_j, i\neq j$} is a map
$$\begin{aligned}
f:  F_n \C &\to \overline{W}\\
(z_1, ..., z_n)&\mapsto f(z_1, ..., z_n)
\end{aligned}$$
such that 
% For any $w'\in W'$, 
% $$\langle w', f(z_1, ..., z_n)\rangle$$
% is a rational function in $z_1, ..., z_n$ with the only possible poles at $z_i = z_j, i\neq j$.
\begin{enumerate}
\item For any $w'\in W'$, 
$$\langle w', f(z_1, ..., z_n)\rangle$$
is a rational function in $z_1, ..., z_n$ with the only possible poles at $z_i = z_j, i\neq j$.
\item There exists integers $p_{ij}, 1\leq i < j \leq n$ and a formal series $g(x_1, ..., x_n)\in W[[x_1, ..., x_n]]$, such that for every $w'\in W'$ and $(z_1, ..., z_n)\in F_n\C$,  
$$\prod_{1\leq i < j \leq n} (z_i-z_j)^{p_{ij}} \langle w', f(z_1, ..., z_n)\rangle = \langle w', g(z_1, ..., z_n)\rangle$$
as a polynomial function. 
\end{enumerate}
\end{defn}

For simplicity, we will simply call such maps \textit{$\overline{W}$-valued rational function} when there is no confusion. The space of all such functions will be denoted by $\widetilde{W}_{z_1,...,z_n}$. 

\begin{rema}\label{SecondCond}
From the second condition, we know that the order of poles of the rational function $\langle w', f(z_1, ..., z_n)\rangle$ is independent of the choice of $w'$. So
for every $w'\in W'$, 
$$(z_1, ..., z_n) \mapsto \prod_{1\leq i < j \leq n} (z_i-z_j)^{p_{ij}}\langle w', f(z_1, ..., z_n)\rangle$$
is a holomorphic (in fact, polynomial) function on $\C^n$,  which can be expanded as a multiple power series
$$\sum_{i_1, ..., i_n=0}^\infty a_{i_1...i_n}(w') z_1^{i_1}\cdots z_n^{i_n}$$
For each $i_1, ..., i_n\in \N$, $w'\mapsto a_{i_1...i_n}(w')$ is an element in $(W')^*$. The second condition further specifies that there exists $b_{i_1...i_n}\in W$, such that $a_{i_1...i_n}(w')= \langle w', b_{i_1...i_n}\rangle$. Thus, $\prod_{1\leq i < j \leq n}(z_i - z_j)^{p_{ij}}f(z_1, ..., z_n)$ can be expanded as
$$ \sum_{i_1, ..., i_n= 0}^\infty b_{i_1... i_n} z_1^{i_1}\cdots z_n^{i_n}\in W[[z_1, ..., z_n]]$$
and therefore, $f(z_1, ..., z_n)$ can be expanded as %the quotient of a power series and the products of $(z_i - z_j)$'s. So we have
$$\frac{\sum\limits_{i_1, ..., i_n= 0}^\infty b_{i_1... i_n} z_1^{i_1}\cdots z_n^{i_n}}{\prod\limits_{1\leq i < j \leq n}(z_i - z_j)^{p_{ij}}} \in W[[z_1, ..., z_n]][(z_1-z_2)^{-1}, ..., (z_{n-1}-z_n)^{-1}]$$
For $1\leq i < j \leq n$, one can further expand the negative powers of $z_i - z_j$ as a power series in $z_j$ and multiply them out. It is clear that in the resulting series
\begin{equation}\label{ExpandRatFunc}
\sum_{k_1, ..., k_n\in \Z} f_{k_1 ... k_n} z_1^{k_1}\cdots z_n^{k_n}
\end{equation}
each coefficient $f_{k_1 ... k_n}$ is a finite sum of various $b_{i_1...i_n}$'s. Thus $f_{k_1 ... k_n}\in W$. So (\ref{ExpandRatFunc}) is a series in $W[[z_1, z_1^{-1}, ..., z_n, z_n^{-1}]]$ that converges absolutely to $f(z_1, ..., z_n)$ in the region 
$$\{(z_1, ..., z_n)\in \C^n: |z_1|>\cdots > |z_n|\}$$
We will also consider the expansion of $f(z_1, ..., z_n)$ in other regions. In all the regions that arise in our applications, all the coefficients of the corresponding series sit in $W$. 
\end{rema}

In an earlier draft of the paper, $f$ takes value in the larger space $(W')^*$. The following observation by Huang says that we can use $\overline{W}$ instead of $(W')^*$:

\begin{prop}\label{nomorearcW-prop}
Let $f: F_n\C \to (W')^*$ satisfying the 	two conditions in Definition \ref{arcWRat}. In addition, assume that there exists a complex number $C$, such that for every $i_1 ,..., i_n \in \N$, the coefficient $b_{i_1, ..., i_n}$ of the power $x_1^{i_1}\cdots x_n^{i_n}$ in the series $g(x_1, ..., x_n)$ are homogeneous, with 
$$\wt b_{i_1...i_n} - i_1 - \cdots - i_n = C$$
Then $f$ takes value in $\overline{W}$. 
\end{prop}

\begin{proof}
We write 
$$g(z_1, ..., z_n) = \sum_{k=0}^\infty \sum_{i_1 + \cdots + i_n = k} b_{i_1, ..., i_n} z_1^{i_1}\cdots z_n^{i_n}$$
Note that for every $k\in \N$, 
$$\sum_{i_1 + \cdots + i_n = k} b_{i_1, ..., i_n} z_1^{i_1}\cdots z_n^{i_n}$$
is a finite sum of homogeneous elements in $W$ of weight $k + C$. As $k$ varies, the weight of these elements varies. Thus for fixed $(z_1, ..., z_n)\in F_n\C$, $g(z_1, ..., z_n)$ is an infinite sum of homogeneous elements in $W$, thus an element in $\overline{W}$. As $f(z_1, ..., z_n)$ is simply a quotient of $g(z_1,..., z_n)$ and products of $(z_i-z_j)$, the same holds for $f(z_1, ..., z_n)$. 
\end{proof}

\begin{rema}\label{nomorearcW}
As all the $(W')^*$-valued rational function that will be used in this paper are finite linear combinations of those in Proposition \ref{nomorearcW-prop}, We need only those $\overline{W}$-valued rational functions given in Definition \ref{arcWRat}. 
\end{rema}

\begin{rema}
It is possible to develop a cohomology theory with $(W')^*$-valued rational functions that do not satisfy the second condition. We choose not to do that because we do not need to consider such general MOSVAs and modules. 
\end{rema}

\begin{prop}
Let $n\in \Z_+$ and take $l = 0, ..., n$. Let $u_1, ..., u_n\in V$ and $w\in W$ satisfying the condition that $\forall u\in V, Y_W^L(u, x)w\in W[[x]]$, $Y_W^{s(R)}(u,x)w\in W[[x]]$. Then for every $w'\in W'$, the series
\begin{equation}\label{LeftlRightn-l-1}
\langle w', Y_W^L(u_1, z_1)\cdots Y_W^L(u_l, z_l)Y_W^{s(R)}(u_{l+1},z_{l+1})\cdots Y_W^{s(R)}(u_{n}, z_{n})w\rangle
\end{equation}
converges absolutely when $|z_1|>\cdots >|z_{n}|$ to a rational function with the only possible poles at $z_i=z_j, 1\leq i < j \leq n.$ 
\end{prop}

\begin{proof}
From rationality we know that the series (\ref{LeftlRightn-l-1}) 
converges absolutely when $|z_1|>\cdots >|z_{n}|>0$ to a rational function with the only possible poles at $z_i=0, i=1,...,n, z_i=z_j, 1\leq i < j \leq n.$ From the assumption, we see that the lowest power of $z_n$ is nonnegative. Therefore, $z_n$ is allowed to take zero. So $z_n=0$ s not a pole. 

Assume that $z_{n-1}=0$ is a pole. By associativity, the series 
\begin{align*}
%& \langle w', Y_W^L(u_1, z_1)\cdots Y_W^L(u_l, z_l)Y_W^{s(R)}(u_{l+1},z_{l+1})\cdots Y_W^{s(R)}(u_{n-1}, z_{n-1}) Y_W^{s(R)}(u_{n}, z_{n})w\rangle\\
%& \qquad = 
\langle w', Y_W^L(u_1, z_1)\cdots Y_W^L(u_l, z_l)Y_W^{s(R)}(u_{l+1},z_{l+1})\cdots Y_W^{s(R)}(Y_V^s(u_{n-1}, z_{n-1}-z_n)u_{n}, z_{n})w\rangle, 
\end{align*}
is obtained by the expanding some rational function in some certain region, during which one of the steps expands the negative powers of $z_{n-1} = z_n+(z_{n-1}-z_n)$ as a series with positive powers $z_{n-1}-z_n$. So there should be infinitely many negative powers of $z_{n}$. However, since $Y_W^{s(R)}(u, z_n)w$ has no negative powers of $z_n$, in particular, for $u=(Y_V)_k(u_{n-1})u_n$ with any $k \in \Z$. Thus this series has no negative powers of $z_{n}$. So it is impossible for $z_{n-1}=0$ to be a pole. 

Similarly, assume $z_{n-2}=0$ is a pole, we use the associativity again to see that 
\begin{align*}
\langle w', Y_W^L(u_1, z_1)\cdots Y_W^L(u_l, z_l)Y_W^{s(R)}(u_{l+1},z_{l+1})\cdots Y_W^{s(R)}(Y_V^s(u_{n-2}, z_{n-2}-z_{n})Y_V^s(u_{n-1}, z_{n-1}-z_n)u_n, z_n)w\rangle, 
\end{align*}
is obtained by the expanding some rational function in some certain region, during which one of the steps expands the negative powers of $z_{n-2} = z_n+(z_{n-2}-z_n)$ as a series with positive powers $z_{n-2}-z_n$. So there should be infinitely many negative powers of $z_n$, which is not possible. 

Similarly one can argue that $z_{n-3}=0$ is not a pole, ..., $z_{l+1}=0$ is not a pole. To see that $z_l$ is not a pole, we use the commutativity of $Y_W^L$ and $Y_W^{s(R)}$ to move all the $Y_W^L$ to the right. The resulting series
\begin{align*}
\langle w', Y_W^{s(R)}(u_{l+1},z_{l+1})\cdots Y_W^{s(R)}(u_{n}, z_{n})Y_W^L(u_1, z_1)\cdots Y_W^L(u_l, z_l)w\rangle
\end{align*}
converges absolutely when $|z_{l+1}|>\cdots > |z_n|>|z_1|>\cdots >|z_l|>0$ to the same rational function as (\ref{LeftlRightn-l-1}) does. As $Y_W^L(u, z_l)w$ has no negative powers of $z_l$ for every $u$, we then see that $z_l$ is allowed to take zero and thus $z_l=0$ is not a pole. Then we apply associativity of $Y_W^L$ and argue similarly that $z_{l-1}=0$ is not a pole, ..., $z_1$ is not a pole. 
% If $l=n$, then instead of using the associativity of $Y_W^{s(R)}$ we use the associativity of $Y_W^L$. The argument is similar. If $l=n-1$, then from the commutativity of $Y_W^L$ and $Y_W^{s(R)}$, we see that the series 
% \begin{align*}
% \langle w', Y_W^L(u_1, z_1)\cdots Y_W^L(u_{n-1}, z_{n-1})Y_W^{s(R)}(u_{n}, z_{n})w\rangle\\
% = \langle w', Y_W^L(u_1, z_1)\cdots Y_W^{s(R)}(u_{n}, z_{n})Y_W^L(u_{n-1}, z_{n-1})w\rangle\\
% \end{align*}
% has no negative powers of $z_{n-1}$. By further moving $Y_W^{s(R)}$ forward and apply the associativity of $Y_W^L$, we see that the $z_{n-2}=0, ..., z_{1}=0$ are not poles. 
\end{proof}

% So the series (\ref{LeftlRightn-l-1}) converges absolutely in the region $|z_1|>\cdots > |z_n|$ to a rational function of the form
% $$\frac{g(z_1, ..., z_n)}{\prod\limits_{1\leq i < j \leq n}(z_i -z_j)^{p_{ij}} } $$
% for some integers $p_{ij}, 1\leq i < j \leq n$ (possibly depending on the choice of $w'$)

\begin{nota} We denote the rational function that the series
\begin{align*}
\langle w', Y_W^L(u_1, z_1)\cdots Y_W^L(u_l, z_l)Y_W^{s(R)}(u_{l+1},z_{l+1})\cdots Y_W^{s(R)}(u_{n}, z_{n})w\rangle\end{align*}
converges to by
$$
R\left(\langle w', Y_W^L(u_1, z_1)\cdots Y_W^L(u_l, z_l)Y_W^{s(R)}(u_{l+1},z_{l+1})\cdots Y_W^{s(R)}(u_{n}, z_{n})w\rangle\right).$$
By the previous proposition, it is of the form
$$\frac{h(z_1, ..., z_n)}{\prod\limits_{1\leq i < j \leq n}(z_i -z_j)^{p_{ij}} } $$ 
for some polynomial $h(z_1, ..., z_n)$ and some integers $p_{ij}, 1\leq i < j \leq n$. The polynomial depends on the choice of $w'\in W', w\in W, u_1, ..., u_n\in V$. But since $V$ and $W$ satisfies the pole-order condition, for each $1\leq i < j \leq n$, the integer $p_{ij}$ depends only on $u_i$ and $u_j$. It is important that $R\left(\langle w', Y_W^L(u_1, z_1)\cdots Y_W^L(u_l, z_l)Y_W^{s(R)}(u_{l+1},z_{l+1})\cdots Y_W^{s(R)}(u_{n}, z_{n})w\rangle\right)$ is defined whenever $z_i\neq z_j, 1\leq i < j \leq n$. The inequality $|z_1|>\cdots > |z_n|$ is not necessary.
\end{nota}

\begin{nota}
For every $(z_1, ..., z_{n})\in F_{n}\C$, the linear functional 
$$w'\mapsto R\left(\begin{aligned}\langle w', Y_W^L(u_1, z_1)\cdots Y_W^L(u_l, z_l)Y_W^{s(R)}(u_{l+1},z_{l+1})
\cdots Y_W^{s(R)}(u_{n}, z_{n})w\rangle\end{aligned}\right)$$ 
determines an element in $(W')^*$ that will be denoted by 
$$E\left(\begin{aligned} Y_W^L(u_1, z_1)\cdots Y_W^L(u_l, z_l)Y_W^{s(R)}(u_{l+1},z_{l+1})\cdots Y_W^{s(R)}(u_{n}, z_{n})w\end{aligned}\right). $$
As will be seen soon, this element is indeed in $\overline{W}$. 
It is important that this element of $\overline{W}$ is defined whenever $z_i \neq z_j, 1\leq i < j \leq n$. The inequality $|z_1|>\cdots > |z_n|$ is not necessary. 
\end{nota}

\begin{rema}
The $E$-notation was introduced by Huang in \cite{Hcoh}. Instead of dealing with the series, we are dealing with the holomorphic function obtained by the analytic extension of the sum of the series. With the $E$-notation, the commutativity of $Y_W^L$ and $Y_W^{s(R)}$ can now be expressedas
$$E(Y_W^L(u_1, z_1)Y_W^{s(R)}(u_2, z_2)w) = E(Y_W^{s(R)}(u_2, z_2)Y_W^L(u_1, z_1)w)$$
Notice that the series in the left-hand-side only makes sense in $|z_1|>|z_2|>0$, and the series in the right-hand-side only makes sense in $|z_2|>|z_1|>0$. So we will not be able to find $z_1, z_2\in \C$ such that $Y_W^L(u_1, z_1)Y_W^{s(R)}(u_2, z_2)w$ and $Y_W^{s(R)}(u_2, z_2)Y_W^L(u_1, z_1)w$ are equal as elements in $\overline{W}$. However, as they both converge to a common rational function that determines an element in $\overline{W}$ defined for every $(z_1, z_2)\in F_2\C$. 
\end{rema}

% \begin{rema}
% As we will see, the use of analytic extension is crucial in the study of the cohomology theory. 
% \end{rema}

\begin{exam}\label{exam-n-cochain}
Let $V$ be a MOSVA and $W$ be a $V$-bimodule. Assume that both $V$ and $W$ satisfies the pole-order condition. 
Fix $n\in \Z_+$ and $l\in \N$ such that $0\leq l \leq n$. 
For every $w\in W$ such that for every $u\in V$, $Y_W^L(u,x)w\in W[[x]], Y_W^{s(R)}(u, x)w\in W[[x]]$, and for every $u_1, ..., u_n\in V$, the map from $F_n\C$ to $\overline{W}$ defined by
\begin{equation} \label{ELeftlRightn-l}
(z_1, ..., z_{n})\mapsto E\left(\begin{aligned}Y_W^L(u_1, z_1)\cdots Y_W^L(u_l, z_l)Y_W^{s(R)}(u_{l+1},z_{l+1})
\cdots Y_W^{s(R)}(u_{n}, z_{n})w\end{aligned}\right)
\end{equation}
is a $\overline{W}$-valued rational function in $z_1, ..., z_n$ with the only possible poles at $z_i= z_j$, $1\leq i < j \leq n$.
\end{exam}

\begin{proof}
The first condition is seen from the discussions above. The second condition follows from Remark \ref{RemaPoleCond-1}. For homogeneous $u_1, ..., u_n\in V, w\in W$, one can computes that the power series
\begin{align*}
\prod_{1\leq i < j \leq n} (z_i - z_j)^{p_{ij}} Y_W^L(u_1, z_1)\cdots Y_W^L(u_l, z_l)Y_W^{s(R)}(u_{l+1}, z_{l+1})\cdots Y_W^{s(R)}(u_n, z_n)w\end{align*}
satisfies the the conditions in Proposition \ref{nomorearcW-prop} ($C$ can be chosen as $\sum\limits_{i=1}^n \wt u_i + \wt w - \sum\limits_{1\leq i < j \leq n} p_{ij}$.) Thus from Remark \ref{nomorearcW}, we see that the rational function in question takes value in $\overline{W}$. 
\end{proof}

\begin{nota}
We will use the notation 
$$E_W^{(l, n-l)}(u_1\otimes \cdots \otimes u_n; w)$$ 
to denote the rational function (\ref{ELeftlRightn-l}), with $n\in \Z_+, l\in \N$, $u_1, ..., u_n\in V$ and $w\in W$ chosen the same way as in the previous example. In particular, $E_W^{(l, n-l)}(u_1\otimes \cdots \otimes u_n; w)\in \widetilde{W}_{z_1 ... z_n}$. When $W = V$ and $w = \one$, we will use the notation
$E_V^{(n)}(u_1\otimes \cdots \otimes u_n)$ for $E_V^{(n, 0)}(u_1\otimes \cdots \otimes u_n; \one)$ without explicitly mentioning $\one$. 
\end{nota}

%=======================================================================

\subsection{Series of $\protect\overline{W}$-valued rational functions} In this paper we will be frequently dealing with series of $\overline{W}$-valued rational function. Here we illustrate some examples. Let $(z_1, ..., z_n)\in F_n \C$. Let $u_1, ..., u_n \in V$ and $w\in W$ such that $Y_W^L(u, x)w\in W[[x]]$ and $Y_W^R(w, x)u\in W[[x]]$ for every $u\in V$. Let $v\in V$ and $x$ be a formal variable. Note that the components $(Y_W^L)_n(v)$ of the vertex operator $Y_W^L(v, x)$ are sums of homogeneous linear operators on $W$ that extends naturally to $\overline{W}$. In particular, they acts on 
\begin{align*}
\overline{w} & = (E_W^{(l, n-l)}(u_1 \otimes \cdots \otimes u_n; w))(z_1, ..., z_n)\\ 
& = E(Y_W^L(u_1, z_1)\cdots Y_W^L(u_l, z_l)Y_W^{s(R)}(u_{l+1}, z_{l+1})\cdots Y_W^{s(R)}(u_n, z_n)w)
\end{align*} 
Thus the vertex operator $Y_W^L(u, x)$ acting on $\overline{w}$ is the following \textit{single }series of elements in $\overline{W}$:
\begin{align*}
Y_W^L(v, x) \overline{w} & = \sum_{n\in \Z} (Y_W^L)_n(u) \overline{w} x^{-n-1}\\
&= \sum_{n\in \Z} (Y_W^L)_n(u) E(Y_W^L(u_1, z_1)\cdots Y_W^L(u_l, z_l)Y_W^{s(R)}(u_{l+1}, z_{l+1})\cdots Y_W^{s(R)}(u_n, z_n)w)x^{-n-1}
\end{align*}
If we pair the above with $w'\in W'$, then the coefficient of $x^{-n-1}$ in $\langle w', Y_W^L(v, x)\overline{w}\rangle$ is just
$$\langle w', (Y_W^L)_n(v)E(Y_W^L(u_1, z_1)\cdots Y_W^L(u_l, z_l)Y_W^{s(R)}(u_{l+1}, z_{l+1})\cdots Y_W^{s(R)}(u_n, z_n)w)\rangle$$
which is a rational function in $z_1, ..., z_n$ with the only possible poles at $z_i = z_j, 1\leq i < j \leq n$. Moreover, if $n$ is sufficiently negative, the coefficient is zero. Thus the series $\langle w', Y_W^L(v, x)\overline{w}\rangle$ has at most finitely many positive powers.

\begin{prop}
Let $u_1, ..., u_n\in V, w\in W$ be chosen as above. Then the single series 
$$Y_W^L(v, z)E(Y_W^L(u_1, z_1)\cdots Y_W^L(u_l, z_l)Y_W^{s(R)}(u_{l+1}, z_{l+1})\cdots Y_W^{s(R)}(u_n, z_n)w)$$
converges absolutely when
$$|z|>|z_i|, i = 1, ..., n$$
to the $\overline{W}$-valued rational function
$$E(Y_W^L(v, z)Y_W^L(u_1, z_1)\cdots  Y_W^L(u_l, z_l)Y_W^{s(R)}(u_{l+1}, z_{l+1})\cdots Y_W^{s(R)}(u_n, z_n)w)$$
\end{prop}

\begin{proof}
For every $w'\in W'$, we know that the series
$$\langle w', Y_W^L(v, z)Y_W^L(u_1, z_1)\cdots Y_W^L(u_l, z_l)Y_W^{s(R)}(u_{l+1}, z_{l+1})\cdots Y_W^{s(R)}(u_n, z_n)w)\rangle$$
converges absolutely when $|z|>|z_1|>\cdots > |z_n|$ to the rational function with the only possible poles at $z=z_i, i= 1,..., n, z_i = z_j, 1\leq i < j \leq n$. For each fixed $n\in \Z$, the coefficient of $z^{-n-1}$ is precisely the sum of the series
$$\langle w', (Y_W^L)_n(v) Y_W^L(u_1, z_1)\cdots Y_W^L(u_l, z_l)Y_W^{s(R)}(u_{l+1}, z_{l+1})\cdots Y_W^{s(R)}(u_n, z_n)w)\rangle$$
in $z_1, ..., z_n$, which is the same as the coefficient of $Y_W^L(v, z)\overline{w}$. From the upper truncation of $z$, we know that the series is obtained by expanding the negative powers of $z-z_i$ as a power series in $z_i$. Thus it converges absolutely whenever $|z|>|z_i|, i = 1, ..., n$. 
\end{proof}

\begin{prop}\label{ComposableExample-1}
Let $u_1, ..., u_n\in V, w\in W$ be chosen as above. Let $m\in \Z_+, v_1, ..., v_m\in V$. Then for each $p= 0, ..., m$, the series 
\begin{align*}
& Y_W^L(v_1, z_1)\cdots Y_W^L(v_p, z_p)Y_W^{s(R)}(v_{p+1}, z_{p+1})\cdots Y_W^{s(R)}(v_m, z_m)\cdot  \\
& \qquad E(Y_W^L(u_1, z_{m+1})\cdots Y_W^L(u_l, z_{m+l})Y_W^{s(R)}(u_{l+1}, z_{m+l+1})\cdots Y_W^{s(R)}(u_n, z_{m+n})w)
\end{align*}
converges absolutely when
$$|z_1|>\cdots > |z_m| > |z_i|, i = m+1, ..., m+n.$$
to the $\overline{W}$-valued rational function
\begin{align*}
& E(Y_W^L(v_1, z_1)\cdots Y_W^L(v_p, z_p)Y_W^{s(R)}(v_{p+1}, z_{p+1})\cdots Y_W^{s(R)}(v_m, z_m)\cdot  \\
& \qquad Y_W^L(u_1, z_{m+1})\cdots Y_W^L(u_l, z_{m+l})Y_W^{s(R)}(u_{l+1}, z_{m+l+1})\cdots Y_W^{s(R)}(u_n, z_{m+n})w)
\end{align*}
\end{prop}

\begin{proof}
It suffices to notice that for each $w'\in W'$, the series
\begin{align*}
&\langle w', Y_W^L(v_1, z_1)\cdots Y_W^L(v_p, z_p)Y_W^{s(R)}(v_{p+1}, z_{p+1})\cdots Y_W^{s(R)}(v_m, z_m)\cdot  \\
& \qquad E(Y_W^L(u_1, z_{m+1})\cdots Y_W^L(u_l, z_{m+l})Y_W^{s(R)}(u_{l+1}, z_{m+l+1})\cdots Y_W^{s(R)}(u_n, z_{m+n})w)\rangle
\end{align*}
coincides with the expansion of the rational function 
\begin{align*}
& R(\langle w', Y_W^L(v_1, z_1)\cdots Y_W^L(v_p, z_p)Y_W^{s(R)}(v_{p+1}, z_{p+1})\cdots Y_W^{s(R)}(v_m, z_m)\cdot  \\
& \qquad Y_W^L(u_1, z_{m+1})\cdots Y_W^L(u_l, z_{m+l})Y_W^{s(R)}(u_{l+1}, z_{m+l+1})\cdots Y_W^{s(R)}(u_n, z_{m+n})w)\rangle
\end{align*}
in the region $\{(z_1, ..., z_{m+n}): |z_1|>\cdots > |z_m|> |z_i|, i = m+1, ..., m+n\}$. 
\end{proof}

\begin{rema}
In terms of the $E$-notation, we have
\begin{align*}
& E(Y_W^L(v, z)E(Y_W^L(u_1, z_1)\cdots Y_W^L(u_l, z_l)Y_W^{s(R)}(u_{l+1}, z_{l+1})\cdots Y_W^{s(R)}(u_n, z_n)w))\\
& \qquad = E(Y_W^L(v, z)Y_W^L(u_1, z_1)\cdots Y_W^L(u_l, z_l)Y_W^{s(R)}(u_{l+1}, z_{l+1})\cdots Y_W^{s(R)}(u_n, z_n)w))
\end{align*}
and
\begin{align*}
& E(Y_W^L(v_1, z_1)\cdots Y_W^L(v_p, z_p)Y_W^{s(R)}(v_{p+1}, z_{p+1})\cdots Y_W^{s(R)}(v_m, z_m)\cdot  \\
& \qquad E(Y_W^L(u_1, z_{m+1})\cdots Y_W^L(u_l, z_{m+l})Y_W^{s(R)}(u_{l+1}, z_{m+l+1})\cdots Y_W^{s(R)}(u_n, z_{m+n})w))\\
= & E(Y_W^L(v_1, z_1)\cdots Y_W^L(v_p, z_p)Y_W^{s(R)}(v_{p+1}, z_{p+1})\cdots Y_W^{s(R)}(v_m, z_m)\cdot  \\
& \qquad Y_W^L(u_1, z_{m+1})\cdots Y_W^L(u_l, z_{m+l})Y_W^{s(R)}(u_{l+1}, z_{m+l+1})\cdots Y_W^{s(R)}(u_n, z_{m+n})w)
\end{align*}
\end{rema}

Here is another type of series of $\overline{W}$-valued rational functions that will be considered. Let $u_1, ..., u_{n+1}\in V, w\in W$ such that $Y_W^L(u, x)w\in W[[x]]$ and $Y_W^R(w, x)u\in W[[x]]$. Let $(\zeta, z_3, ..., z_n)\in F_n \C$. 
$$E_W^{(l, n-l)}(Y_V(u_1, z_1-\zeta) Y_V(u_2, z_2-\zeta)\one \otimes u_3 \otimes \cdots \otimes u_{n+1}; w)(\zeta, z_3, ..., z_{n+1})$$
which expands as 
$$\sum_{k_1, k_2\in \Z} E_W^{(l, n-l)}((Y_V)_{k_1}(u_1)(Y_V)_{k_2}(u_2)\one \otimes u_3 \otimes \cdots \otimes u_{n+1}; w)(\zeta, z_3, ..., z_{n+1}) (z_1-\zeta)^{-k_1-1}(z_2-\zeta)^{-k_2-1}$$
For each $k_1, k_2\in \Z$, the coefficients of $(z_1-\zeta)^{-k_1-1}(z_2-\zeta)^{-k_2-1}$ is a $\overline{W}$-valued rational function in $\zeta, z_3, ..., z_{n+1}$. 

\begin{prop}
Let $u_1, ..., u_{n+1}\in V$ and $w\in W$ be chosen as above. Then the series 
\begin{equation}\label{ComposableExample-2}
E_W^{(l, n-l)}(Y_V(u_1, z_1-\zeta) Y_V(u_2, z_2-\zeta)\one \otimes u_3 \otimes \cdots \otimes u_{n+1}; w)(\zeta, z_3, ..., z_{n+1})
\end{equation}
converges absolutely when 
\begin{align*}
|z_3-\zeta|>|z_1-\zeta|>|z_2-\zeta|
\end{align*}
to the $\overline{W}$-valued rational function 
$$E(Y_W^L(u_1, z_1) \cdots Y_W^L(u_{l+1}, z_{l+1}) Y_W^{s(R)}(u_{l+2}, z_{l+2})\cdots Y_W^{s(R)}(u_{n+1}, z_{n+1})w)$$
\end{prop}

\begin{proof}
For every $w'\in W'$, we know that the series 
$$\langle w', Y_W^L(u_1, z_1) \cdots Y_W^L(u_{l+1}, z_{l+1}) Y_W^{s(R)}(u_{l+2}, z_{l+2})\cdots Y_W^{s(R)}(u_{n+1}, z_{n+1})w\rangle$$
converges absolutely when $|z_1|>\cdots > |z_{n+1}|$ to a rational function with the only possible poles at $z_i = z_j, 1\leq i < j \leq n+1$. By associativity and Lemma 4.5 in \cite{Q-Rep}, we know that the series
$$\langle w', Y_W^L(Y_V(u_1, z_1-\zeta)Y_V(u_2, z_2-\zeta)\one, \zeta)Y_W^L(u_3, z_3)\cdots Y_W^L(u_{l+1}, z_{l+1})Y_W^{s(R)}(u_{l+2}, z_{l+2})\cdots Y_W^{s(R)}(u_{n+1}, z_{n+1})w\rangle$$
with variables $z_1-\zeta, z_2-\zeta, \zeta, z_3, ..., z_n$ that expands as
\begin{align*}
\sum_{k_1, ..., k_{n+2}\in \Z} \langle w',& (Y_W^L)_{k_3}((Y_V)_{k_1}(u_1)(Y_V)_{k_2}(u_2)\one)(Y_W^L)_{k_4}(u_3)\cdots (Y_W^L)_{k_{l+2}}(u_{l+1}) \\
& \cdot (Y_W^{s(R)})_{k_{l+3}}(u_{l+2})\cdots (Y_W^{s(R)})_{k_{n+2}}(u_{n+1})w\rangle (z_1-\zeta)^{-k_1-1} (z_2-\zeta)^{-k_2-1} \zeta^{-k_3-1} z_3^{-k_4-1}\cdots z_{n+1}^{-k_{n+2}-1}
\end{align*}
is obtained from the following expansion of the rational function:
\begin{enumerate}
\item Expand the negative powers of $z_1-z_2 = z_1-\zeta - (z_2-\zeta)$ as a power series of $z_2-\zeta$. 
\item For $s=1, 2$ and $j = 3, ..., n$, expand the negative powers of $z_s-z_j = \zeta + (z_s - \zeta + z_j)$ as a power series of $z_s - \zeta + z_j$, then further expand the positive powers of $z_s-\zeta+z_j$ as polynomials of $z_s-\zeta$ and $z_j$. Note that this expansion is the same as first expand the negative powers of $z_s-z_j=-(z_j - \zeta) + (z_s-\zeta)$ as power series of $(z_s-\zeta)$, then further expand all the negative powers of $z_j - \zeta$ as power series of $z_j$. 
\item For $3\leq i < j \leq n$, expand the negative powers of $z_i - z_j$ as power series of $z_j$. 
\end{enumerate}
Thus the series converges absolutely when
\begin{align*}
& |z_1-\zeta|>|z_2-\zeta|, |\zeta|>|z_3|> \cdots > |z_n|, \\
& |z_j - \zeta|>|z_1-\zeta|, j = 3, ..., n.  
\end{align*}
The result then follow by noticing that the coefficients of the series (\ref{ComposableExample-2}), paired with $w'$, are precisely the partial sums of the above series with respect to $k_3, ..., k_{n+2}$. In particular, the series (\ref{ComposableExample-2}) is obtained from the following expansions of the rational function 
$$R(\langle w', Y_W^L(u_1, z_1) \cdots Y_W^L(u_{l+1}, z_{l+1}) Y_W^{s(R)}(u_{l+2}, z_{l+2})\cdots Y_W^{s(R)}(u_{n+1}, z_{n+1})w\rangle)$$
\begin{enumerate}
\item Expand the negative powers of $z_1-z_2 = z_1-\zeta - (z_2-\zeta)$ as a power series of $z_2-\zeta$. 
\item For $s=1, 2$ and $j = 3, ..., n$, expand the negative powers of $z_s-z_j=-(z_j - \zeta) + (z_s-\zeta)$ as power series of $(z_s-\zeta)$.
\end{enumerate}
Thus the series (\ref{ComposableExample-2}) converges absolutely when $|z_3-\zeta|>|z_1-\zeta|>|z_2-\zeta|.$
\end{proof}

\begin{prop}
Let $m, n \in \Z_+$. Let $\alpha_1, ..., \alpha_n$ be chosen such that $\alpha_1+\cdots + \alpha_n = m+n$. Then the series
$$E_W^{(l, n-l)}(Y_V(u_1^{(1)}, z_1^{(1)}-\zeta_1)\cdots Y_V(u_{\alpha_1}^{(1)}, z_{\alpha_1}^{(1)}-\zeta_1)\one \otimes \cdots \otimes Y_V(u_1^{(n)}, z_1^{(n)}-\zeta_n)\cdots Y_V(u_{\alpha_n}^{(n)}, z_{\alpha_n}^{(n)}-\zeta_n)\one)(\zeta_1, ..., \zeta_n)$$
converges absolutely when
\begin{align*}
& |\zeta_i - \zeta_j| > |z_s^{(i)}-\zeta_i| + |z_t^{(j)}-\zeta_j|, 1\leq i < j \leq n, s = 1, ..., \alpha_i, t = 1, ..., \alpha_j\\
& |z_s^{(i)} - \zeta_i| > |z_t^{(i)}- \zeta_i|, i = 1,..., n, 1\leq s < t \leq \alpha_i. 
\end{align*}
to the $\overline{W}$-valued rational function 
\begin{align*}
& E(Y_W^L(u_1^{(1)}, z_1^{(1)})\cdots Y_W^L(u_{\alpha_1}^{(1)}, z_{\alpha_1}^{(1)}) \cdots Y_W^L(u_1^{(l)}, z_1^{(l)})\cdots Y_W^L(u_{\alpha_l}^{(l)}, z_{\alpha_l}^{(l)})\\
& \qquad \cdot Y_W^{s(R)}(u_1^{(l+1)}, z_1^{(l+1)})\cdots Y_W^{s(R)}(u_{\alpha_{l+1}}^{(l+1)}, z_{\alpha_{l+1}}^{(l+1)}) \cdots Y_W^{s(R)}(u_1^{(n)}, z_1^{(n)})\cdots Y_W^{s(R)}(u_{\alpha_n}^{(n)}, z_{\alpha_n}^{(n)})w)
\end{align*}
\end{prop}

\begin{proof}
It suffices argue similarly that the series, paired with any $w'\in W'$, is obtained from the expansion of the corresponding rational function in the region. 
\end{proof}

We end this section by a proposition dealing with the mixture of the above two types of series of $\overline{W}$-valued rational functions

\begin{prop}\label{ComposableExample}
Fix $m, n\in \Z_+$. Let $\alpha_{0}, \alpha_1, \cdots, \alpha_n$ be chosen such that $\alpha_0 + \alpha_1 + \cdots + \alpha_n = m+n$. Then for every $l_0=0, ..., \alpha_0$, the the series 
\begin{align*}
& Y_W^L(u_1^{(0)}, z_1^{(0)}) \cdots Y_W^L(u_{l_0}^{(0)}, z_{l_0}^{(0)}) Y_W^{s(R)}(u_{l_0+1}^{(0)}, z_{l_0+1}^{(0)}) \cdots Y_W^{s(R)}(u_{\alpha_0}^{(0)}, z_{\alpha_0}^{(0)}) \\
& \cdot E_W^{(l, n-l)}(Y_V(u_1^{(1)}, z_1^{(1)}-\zeta_1)\cdots Y_V(u_{\alpha_1}^{(1)}, z_{\alpha_1}^{(1)}-\zeta_1)\one \otimes \cdots \otimes Y_V(u_1^{(n)}, z_1^{(n)}-\zeta_n)\cdots Y_V(u_{\alpha_n}^{(n)}, z_{\alpha_n}^{(n)}-\zeta_n)\one)(\zeta_1, ..., \zeta_n)
\end{align*}
converges absolutely when
\begin{align*}
& |z_1^{(0)}|> \cdots > |z_{\alpha_0}^{(0)}| > |\zeta_i|+|z_t^{(i)}-\zeta_i|, i = 1, ..., n, t = 1, ..., \alpha_i\\
& |z_1^{(i)}-\zeta_i|> \cdots > |z_{\alpha_i}^{(i)}-\zeta_i|, i = 1, ..., n\\
& |\zeta_i - \zeta_j|> |z_s^{(i)} - \zeta_i| + |z_t^{(j)} -\zeta_j|, 1\leq i < j \leq n, s = 1, ..., \alpha_i, t = 1, ..., \alpha_j. 
\end{align*}
to the $\overline{W}$-valued rational function
\begin{align*}
& E(Y_W^L(u_1^{(0)}, z_1^{(0)}) \cdots Y_W^L(u_{l_0}^{(0)}, z_{l_0}^{(0)}) Y_W^{s(R)}(u_{l_0+1}^{(0)}, z_{l_0+1}^{(0)}) \cdots Y_W^{s(R)}(u_{\alpha_0}^{(0)}, z_{\alpha_0}^{(0)})  \\
& \qquad \cdot Y_W^L(u_1^{(1)}, z_1^{(1)})\cdots Y_W^L(u_{\alpha_1}^{(1)}, z_{\alpha_1}^{(1)}) \cdots Y_W^L(u_1^{(l)}, z_1^{(l)})\cdots Y_W^L(u_{\alpha_l}^{(l)}, z_{\alpha_l}^{(l)})\\
& \qquad \cdot Y_W^{s(R)}(u_1^{(l+1)}, z_1^{(l+1)})\cdots Y_W^{s(R)}(u_{\alpha_{l+1}}^{(l+1)}, z_{\alpha_{l+1}}^{(l+1)}) \cdots Y_W^{s(R)}(u_1^{(n)}, z_1^{(n)})\cdots Y_W^{s(R)}(u_{\alpha_n}^{(n)}, z_{\alpha_n}^{(n)})w)
\end{align*}
\end{prop}

\begin{proof}
It suffices to notice that the series in question, paired with any $w'\in W'$, is obtained from the following expansions of the corresponding rational function:
\begin{enumerate}
\item For $s = 1, ..., \alpha_0, i = 0, ..., n, t = 1, ..., \alpha_i$, expand the negative powers $z_s^{(0)}-z_t^{(i)}$ as power series of $z_t^{(i)}$. When $i \geq 1$, one further expands the positive powers of $z_t^{(i)} = \zeta_i + (z_t^{(i)}-\zeta_i)$ as polynomials of $\zeta_i$ and $z_t^{(i)}$. 
\item For $i = 1, ..., n, 1\leq s < t \leq \alpha_i$, expand the negative powers of $z_s^{(i)}-z_t^{(i)} =z_s^{(i)}-\zeta_i -(z_t^{(i)}-\zeta_i)$ as power series of $(z_t^{(i)}-\zeta_i)$. 
\item For $1\leq i < j \leq n, s=1, ..., \alpha_i, t = 1, ..., \alpha_j$, expand the negative powers of $z_s^{(i)} - z_t^{(j)} = (\zeta_i - \zeta_j)+ (z_s^{(i)}-\zeta_i-z_t^{(j)} + \zeta_j)$ as power series of $(z_s^{(i)}-\zeta_i-z_t^{(j)} + \zeta_j)$, then further expand the positive powers of $(z_s^{(i)}-\zeta_i-z_t^{(j)} + \zeta_j)$ as polynomials of $(z_s^{(i)} - \zeta_i)$ and $(z_t^{(j)}-\zeta_j)$. 
\end{enumerate}
Thus the series in question converges absolutely when 
\begin{align*}
& |z_1^{(0)}|> \cdots > |z_{\alpha_0}^{(0)}| > |\zeta_i|+|z_t^{(i)}-\zeta_i|, i = 1, ..., n, t = 1, ..., \alpha_i\\
& |z_1^{(i)}-\zeta_i|> \cdots > |z_{\alpha_i}^{(i)}-\zeta_i|, i = 1, ..., n\\
& |\zeta_i - \zeta_j|> |z_s^{(i)} - \zeta_i| + |z_t^{(j)} -\zeta_j|, 1\leq i < j \leq n, s = 1, ..., \alpha_i, t = 1, ..., \alpha_j. 
\end{align*}

\end{proof}

%=======================================================================

\subsection{Associativity and commutativity extended to $\protect\overline{W}$-valued rational functions} In this subsection we consider the vertex operator action on more general $\overline{W}$-valued rational functions. Let $(z_1, ..., z_n)\in F_n \C$. Let $v\in V$ and $x$ be a formal variable. Let $f$ be a $\overline{W}$-valued rational function. Then $(Y_W^L)_n(v)$ acts on the $\overline{W}$-element
\begin{align*}
\overline{w} & = f(z_1, ..., z_n)
\end{align*} 
Thus the vertex operator $Y_W^L(u, x)$ acting on $\overline{w}$ is the following \textit{single }series of elements in $\overline{W}$:
\begin{align*}
Y_W^L(v, x) \overline{w} & = \sum_{n\in \Z} (Y_W^L)_n(u) \overline{w} x^{-n-1}\\
&= \sum_{n\in \Z} (Y_W^L)_n(u) f(z_1, ..., z_n) x^{-n-1}
\end{align*}
If we pair the above with $w'\in W'$, then the coefficient of $x^{-n-1}$ in $\langle w', Y_W^L(v, x)\overline{w}\rangle$ is just
$$\langle w', (Y_W^L)_n(u)f(z_1, ..., z_n))\rangle$$
which is rational function in $z_1, ..., z_n$ with the only possible poles at $z_i = z_j, 1\leq i < j \leq n$. Moreover, if $n$ is sufficiently negative, the coefficient is zero. Thus the series $\langle w', Y_W^L(v, x)\overline{w}\rangle$ has at most finitely many positive powers. 

Similarly, for $v_1, v_2\in V$ and formal variables $x_1, x_2$, the series
$$Y_W^L(v_1, x_1) Y_W^L(v_2, x_2)\overline{w} $$
and 
$$Y_W^L(Y_V(v_1, x_0)v_2, x_2)\overline{w}$$
are understood as \textit{double} series of elements in $\overline{W}$
\begin{align*}
 Y_W^L(v_1, x_1) Y_W^L(v_2, x_2)\overline{w} =&  \sum_{k_1,  k_2\in \Z} (Y_W^L)_{k_1}(v_1) (Y_W^L)_{k_2}(v_2)\overline{w} x_1^{-k_1-1} x_2^{-k_2-1}\\
= &  \sum_{k_1,  k_2\in \Z} (Y_W^L)_{k_1}(v_1) (Y_W^L)_{k_2}(v_2)f(z_1, ..., z_n) x_1^{-k_1-1} x_2^{-k_2-1}\\
 Y_W^L(Y_V(v_1, x_0)v_2, x_2)\overline{w} = &  \sum_{k_1,  k_2\in \Z} (Y_W^L)_{k_1}((Y_V)_{k_1}(v_1) v_2)\overline{w} x_0^{-k_1-1} x_2^{-k_2-1}\\
= &  \sum_{k_1,  k_2\in \Z} (Y_W^L)_{k_1}((Y_V)_{k_1}(v_1) v_2) f(z_1, ..., z_n) x_0^{-k_1-1} x_2^{-k_2-1}
\end{align*}
In general, we don't know if these two series converge. But if $\overline{w}$ is chosen appropriately and one of them converges absolutely under certain conditions, then the other also converges absolutely. More precisely, 

\begin{prop}
Let $v_1, v_2\in V$. Let $f\in \widetilde{W}_{z_3, ..., z_{n+2}}$ such that for every $(z_1, ..., z_{n+2})\in F_{n+2}\C$ with $|z_1|>|z_2|>|z_i|, i =3, ..., n+2$, the series
$$Y_W^L(v_1, z_1)Y_W^L(v_2, z_2)f(z_3, ..., z_n)$$
converges absolutely to a $\overline{W}$-valued rational function. Then for every $(z_1, ..., z_{n+2})\in F_{n+2} \C$ such that $|z_2|>|z_1-z_2|+|z_i|, i = 3, ..., n+2$, the series 
$$Y_W^L(Y_V(v_1, z_1-z_2)v_2, z_2)f(z_3, ..., z_n)$$
also converges absolutely to the same $\overline{W}$-valued rational function. 
\end{prop}

\begin{proof}
By Definition \ref{arcWRat} and Remark \ref{SecondCond}, we know that $f(z_3, ..., z_n)$ can be expanded in the region 
$$\{(z_3, ..., z_{n+2})\in \C^n: |z_3|>\cdots > |z_n|\}$$
as an absolutely convergent series 
$$f(z_3, ..., z_{n+2}) = \sum_{k_3, ..., k_{n+2}\in \Z} f_{k_3 ..., k_{n+2}} z_3^{k_3}\cdots z_{n+2}^{k_{n+2}}$$
in $W[[z_3, z_3^{-1}, ..., z_{n+2}, z_{n+2}^{-1}]]$. This expansion is obtained by expanding each negative power of $z_i - z_j$ as a power series in $z_j$, for $3\leq i < j \leq n+2$. Thus the series is lower-truncated in $z_{n+2}$. The coefficient of each fixed power of $z_{n+2}$, as a series in $z_3, ..., z_{n+1}$, is lower-truncated in $z_{n+1}$. In general, for each $i=3, ..., n+1$ and each $k_{i+1}, ..., k_{n+2}\in \Z$, the coefficient of $z_{i+1}^{k_{i+1}}\cdots z_{n+2}^{k_{n+2}}$, as a series in $z_3, ..., z_i$, is lower-truncated in $z_i$. If we pick $M_i\in \Z$ such that the lowest power of $z_i$ is $-M_i$, then we can recover the coefficient of the series from the following formula 
$$f_{k_3 ... k_{n+2}} = \lim_{z_3 = 0} \cdots \lim_{z_{n+2}=0} \left(\frac \partial {\partial z_{3}}\right)^{k_{3} + M_{3}}\cdots \left(\frac \partial {\partial z_{n+2}}\right)^{k_{n+2} + M_{n+2}}(z_3^{M_3}\cdots z_{n+2}^{M_{n+2}}f(z_3, ..., z_n))$$
Now by assumption, the series 
\begin{align*}
& Y_W^L(u_1, z_1)Y_W^L(u_2, z_2)f(z_3, ..., z_{n+2})\\
= & \sum_{k_1, k_2 \in \Z} (Y_W^L)_{k_1}(u_1) (Y_W^L)_{k_2}(u_2) f(z_3, ..., z_{n+2}) z_1^{-k_1-1}z_2^{-k_2-1}
\end{align*}
converges absolutely when $|z_1|>|z_2|>|z_i|, i = 3, ..., n+2$ to a $\overline{W}$-valued rational function. This means the following iterated series  
\begin{align*}
& Y_W^L(u_1, z_1)Y_W^L(u_2, z_2)\left(\sum_{k_3, ..., k_{n+2}\in \Z} f_{k_3 ..., k_{n+2}} z_3^{k_3}\cdots z_{n+2}^{k_{n+2}}\right)\\
=& \sum_{k_1, k_2\in \Z}\left(\sum_{k_3, ..., k_{n+2}\in \Z} (Y_W^L)_{k_1}(u_1)(Y_W^L)_{k_2}(u_2) f_{k_3 ..., k_{n+2}} z_3^{k_3}\cdots z_{n+2}^{k_{n+2}}\right)z_1^{-k_1-1}z_2^{-k_2-1},
\end{align*}
viewed as a double series in $z_1, z_2$ whose coefficients are $$\sum\limits_{k_3, ..., k_{n+2}\in \Z} (Y_W^L)_{k_1}(u_1)(Y_W^L)_{k_2}(u_2) f_{k_3 ..., k_{n+2}} z_3^{k_3}\cdots z_{n+2}^{k_{n+2}},$$ 
converges absolutely when $|z_1|>\cdots > |z_{n+2}|$ to the same $\overline{W}$-valued rational function. Moreover, the power of $z_2$ is lower-truncated. And for each fixed power of $z_2$, the power of $z_1$ in coefficient series is also lower-truncated. Thus by Lemma 4.5 in \cite{Q-Rep}, the series
$$\sum_{k_1,  ..., k_{n+2}\in \Z} (Y_W^L)_{k_1}(u_1)(Y_W^L)_{k_2}(u_2) f_{k_3 ..., k_{n+2}} z_1^{-k_1-1}z_2^{-k_2-1}z_3^{k_3}\cdots z_{n+2}^{k_{n+2}} $$
is precisely the expansion of the $\overline{W}$-valued rational function 
\begin{equation}\label{Assoc-Rat-1}
E(Y_W^L(u_1, z_1)Y_W^L(u_2, z_2)f(z_3, ..., z_{n+2}))
\end{equation}
in the region 
$$\{(z_1, ..., z_{n+2})\in \C^{n+2}: |z_1|>\cdots > |z_{n+2}|\}$$
In particular, the series converges absolutely in this region. 

By associativity, for fixed $k_3, ..., k_{n+2}$, when $|z_1|>|z_2|>|z_1-z_2|>0$, we have
\begin{align*}
Y_W^L(u_1, z_1)Y_W^L(u_2, z_2) f_{k_3 ..., k_{n+2}} &=Y_W^L(Y_V(u_1, z_1-z_2)u_2, z_2) f_{k_3 ..., k_{n+2}} \\
& = \sum_{k_1, k_2\in \Z}(Y_W^L)_{k_1}((Y_V)_{k_2}(u_1)u_2)(z_1-z_2)^{-k_1-1}z_2^{-k_2-1} 
\end{align*}
Thus the series 
\begin{align*}
& Y_W^L(Y_V(u_1, z_1-z_2)u_2, z_2)\sum_{k_3, ..., k_{n+2}\in \Z} f_{k_3 ..., k_{n+2}} z_3^{k_3}\cdots z_{n+2}^{k_{n+2}}\\
=& \sum_{k_3, ..., k_{n+2}\in \Z} \left(\sum_{k_1, k_2\in \Z}(Y_W^L)_{k_1}((Y_V)_{k_2}(u_1)u_2)f_{k_3 ..., k_{n+2}} (z_1-z_2)^{-k_1-1}z_2^{-k_2-1} \right) z_3^{k_3}\cdots z_{n+2}^{k_{n+2}}
\end{align*}
viewed as a series in $z_3, ..., z_{n+2}$ whose coefficients are $$\sum_{k_1, k_2\in \Z}(Y_W^L)_{k_1}((Y_V)_{k_2}(u_1)u_2)f_{k_3 ..., k_{n+2}}(z_1-z_2)^{-k_1-1}z_2^{-k_2-1},$$ converges absolutely when 
$$|z_1|>|z_2|>\cdots > |z_{n+2}|, |z_2|>|z_1-z_2|>0. $$
Moreover, for every $i = n+2, ..., 3$ and every $k_{i+1}, ..., k_{n+2}$, the coefficient series of $z_{i+1}^{k_{i+1}}\cdots z_{n+2}^{k_{n+2}}$ is lower-truncated in $z_i$. One then sees from the Lemma 4.5 in \cite{Q-Rep} that the series 
\begin{align*}
 & Y_W^L(Y_V(u_1, z_1-z_2)u_2, z_2)\sum_{k_3, ..., k_{n+2}\in \Z} f_{k_3 ..., k_{n+2}} z_3^{k_3}\cdots z_{n+2}^{k_{n+2}} \\
=& \sum_{k_1, k_2, k_3, ..., k_{n+2}\in \Z} (Y_W^L)_{k_1}((Y_V)_{k_2}(u_1)u_2)f_{k_3 ..., k_{n+2}}(z_1-z_2)^{-k_1-1}z_2^{-k_2-1} z_3^{k_3}\cdots z_{n+2}^{k_{n+2}}
\end{align*}
is the expansion of the $\overline{W}$-valued rational function (\ref{Assoc-Rat-1}) in the region
$$\{(z_1, ..., z_{n+2})\in \C^{n+2}: |z_2|>|z_1-z_2|+|z_3|, |z_1-z_2|>0, |z_3|>\cdots > |z_{n+2}|\}$$
In particular, the series converges absolutely in the region. We then sum up all $k_3, ..., k_{n+1}$, to see that the double series 
\begin{align*}
& Y_W^L(Y_V(u_1, z_1-z_2)u_2, z_2)f(z_3, ..., z_{n+2})
= & \sum_{k_1, k_2\in \Z} (Y_W^L)_{k_1}((Y_V)_{k_2}(u_1)u_2)f(z_3, ..., z_{n+2})
\end{align*}
of elements in $\overline{W}$ is precisely the expansion of the $\overline{W}$-rational function (\ref{Assoc-Rat-1}) in the region 
$$\{(z_1, ..., z_{n+2}): |z_2|>|z_1-z_2|+|z_i|, i = 3, ..., n+2\}$$
In particular, the double series converges absolutely in the region. 
\end{proof}

\begin{cor}
For $u_1, u_2\in V$ and $f\in \widetilde{W}_{z_3, ..., z_{n+2}}$ chosen as above, we have
$$Y_W^L(u_1, z_1)Y_W^L(u_2, z_2)f(z_3, ..., z_{n+2}) = Y_W^L(Y_V(u_1, z_1-z_2)u_2, z_2)f(z_3, ..., z_{n+2})$$
for every $(z_1, ..., z_{n+2})\in F_{n+2}\C$ such that $|z_1|>|z_2|> |z_1-z_2|+|z_i|, i = 3, ..., n+2$. Moreover, we have 
$$E(Y_W^L(u_1, z_1)Y_W^L(u_2, z_2)f(z_3, ..., z_{n+2})) = E(Y_W^L(Y_V(u_1, z_1-z_2)u_2, z_2)f(z_3, ..., z_{n+2}))$$
where both sides are regarded as $\overline{W}$-valued rational functions in $\widetilde{W}_{z_1, ..., z_{n+2}}$.
\end{cor}

One can generalize the above conclusions to the product of any numbers of $Y_W^L$ and $Y_W^{s(R)}$ vertex operators. For convenience, we list the conclusions we will need in this paper in the following theorem. 

\begin{thm}\label{ConvRef}
\begin{enumerate}
\item Let $u_1, u_2\in V$, $f\in \widetilde{W}_{z_3, ..., z_{n+2}}$ such that 
$$Y_W^L(u_1, z_1)Y_W^{s(R)}(u_2, z_2)f(z_3, ..., z_{n+2})$$
converges absolutely to a $\overline{W}$-valued rational function for every $(z_1,..., z_{n+2})\in F_{n+2}\C$ such that $|z_1|>|z_2|>|z_i|, i=3, ..., n+2$. Then the series
$$Y_W^{s(R)}(u_2, z_2)Y_W^L(u_1, z_1)f(z_3, ..., z_{n+2})$$
also converges absolutely to the same $\overline{W}$-valued rational function for every $(z_1,..., z_{n+2})\in F_{n+2}\C$ such that $|z_2|>|z_1|>|z_i|, i = 3, ..., n+2$. Moreover, we have
$$E(Y_W^L(u_1, z_1)Y_W^{s(R)}(u_2, z_2)f(z_3, ..., z_{n+2}))=E(Y_W^{s(R)}(u_2, z_2)Y_W^L(u_1, z_1)f(z_3, ..., z_{n+2}))$$
as elements in $\widetilde{W}_{z_1, ..., z_{n+2}}$. 
\item Let $u_1, ..., u_m \in V$, $f\in \widetilde{W}_{z_{m+1}, ..., z_{m+n}}$ such that 
$$Y_W^L(u_1, z_1)\cdots Y_W^L(u_m, z_m)f(z_{m+1}, ..., z_{m+n})$$
converges absolutely to a $\overline{W}$-valued rational function for every $(z_1, ..., z_{m+n})\in F_{m+n}\C$ such that $|z_1|>\cdots > |z_m|>|z_i|, i = m+1, ..., m+n$. Then the series
$$Y_W^L(Y_V(u_1, z_1-\zeta)\cdots Y_V(u_m, z_m-\zeta)\one, \zeta)f(z_{m+1}, ..., z_{m+n})$$
also converges absolutely to the same $\overline{W}$-valued rational function whenever $(z_1, ..., z_{m+n})\in F_{m+n}\C, |\zeta|>|z_1-\zeta|+|z_i|, i = m+1, ..., m+n, |z_1-\zeta|>\cdots > |z_m-\zeta|$. Moreover, we have
\begin{align*}
& E(Y_W^L(u_1, z_m)\cdots Y_W^L(u_m, z_m) f(z_{m+1}, ..., z_{m+n})) \\ 
& \qquad = E(Y_W^L(Y_V(u_1, z_1-\zeta)\cdots Y_V(u_m, z_m-\zeta)\one, \zeta)f(z_{m+1}, ..., z_{m+n}))
\end{align*}
as elements in $\widetilde{W}_{z_1, ..., z_{m+n}}$. 
\item Let $u_1, ..., u_m \in V$, $f\in \widetilde{W}_{z_{m+1}, ..., z_{m+n}}$ such that 
$$Y_W^{s(R)}(u_m, z_m)\cdots Y_W^{s(R)}(u_1, z_1)f(z_{m+1}, ..., z_{m+n})$$
converges absolutely to a $\overline{W}$-valued rational function for every $(z_1, ..., z_{m+n})\in F_{m+n}\C$ such that $|z_m|>\cdots > |z_1|>|z_i|, i = m+1, ..., m+n$. Then the series
$$Y_W^{s(R)}(Y_V(u_1, z_1-\zeta)\cdots Y_V(u_m, z_m-\zeta)\one, \zeta)f(z_{m+1}, ..., z_{m+n})$$
also converges absolutely to the same $\overline{W}$-valued rational function whenever $(z_1, ..., z_{m+n})\in F_{m+n}\C, |\zeta|>|z_1-\zeta|+|z_i|, i = m+1, ..., m+n, |z_1-\zeta|>\cdots > |z_m-\zeta|$. Moreover, we have
\begin{align*}
& E(Y_W^{s(R)}(u_m, z_m)\cdots Y_W^{s(R)}(u_1, z_1) f(z_{m+1}, ..., z_{m+n})) \\ 
& \qquad = E(Y_W^{s(R)}(Y_V(u_1, z_1-\zeta)\cdots Y_V(u_m, z_m-\zeta)\one, \zeta)f(z_{m+1}, ..., z_{m+n}))
\end{align*}
as elements in $\widetilde{W}_{z_1, ..., z_{m+n}}$. 
\end{enumerate}
\end{thm}

\section{The cochain complex and the cohomology group}

Let $n$ be a fixed positive integer. We will define cochain complexes from linear maps $V^{\otimes n} \to \widetilde{W}_{z_1, ..., z_n}$ satisfying some natural properties.

\subsection{Linear maps $V^{\otimes n} \to \widetilde{W}_{z_1...z_n}$ satisfying $D$-derivative and $\d$-conjugation properties}

\begin{defn}
A linear map $\Phi: V^{\otimes n} \to \widetilde{W}_{z_1, ..., z_n}$ is said to have the $D$-derivative property if
\begin{enumerate}
\item For $i=1,..., n$, $v_1, ..., v_n\in V, w'\in W'$, 
$$ \langle w', (\Phi(v_1\otimes \cdots\otimes v_{i-1} \otimes D_V v_i \otimes v_{i+1}\otimes \cdots  \otimes v_n))(z_1, ..., z_n)\rangle =\frac{\partial}{\partial z_i} \langle w', (\Phi(v_1\otimes \cdots \otimes v_n))(z_1, ..., z_n)\rangle $$
\item For $v_1, ..., v_n\in V, w'\in W'$, 
$$\langle w', D_W (\Phi(v_1\otimes \cdots v_n))(z_1, ..., z_n)\rangle=\left(\frac{\partial}{\partial z_1} + \cdots + \frac{\partial}{\partial z_n} \right) \langle w', (\Phi(v_1\otimes \cdots \otimes v_n))(z_1, ..., z_n)\rangle $$
\end{enumerate}
\end{defn}

\begin{defn}
A linear map $\Phi: V^{\otimes n} \to \widetilde{W}_{z_1, ..., z_n}$ is said to have the $\d$-conjugation property if for $v_1, ..., v_n\in V, w'\in W', (z_1, ..., z_n)\in F_n\C$ and $z\in \C^\times$ so that $(zz_1, ..., zz_n)\in F_n \C$, 
$$\langle w', z^{\d_W}(\Phi(v_1\otimes \cdots \otimes v_n))(z_1, ..., z_n)\rangle = \langle w', (\Phi(z^{\d_V}v_1 \otimes \cdots z^{\d_V}v_n))(zz_1, ..., zz_n)\rangle$$
\end{defn}

\begin{prop}\label{D-conjugation}
Let $\Phi: V^{\otimes n} \to \widetilde{W}_{z_1, ..., z_n}$ be a linear map satisfying the $D$-derivative property. 
\begin{enumerate}
\item For $v_1, ..., v_n\in V, w'\in W', (z_1, ..., z_n)\in F_n \C, z\in \C$ and $1\leq i \leq n$ such that $(z_1, ..., z_{i-1}, z_i + z, z_{i+1}, ..., z_n)\in F_n \C$, the power series expansion of 
$$\langle w', (\Phi(v_1\otimes \cdots \otimes v_n))(z_1, ..., z_{i-1}, z_i + z, z_{i+1}, ..., z_n)\rangle$$
in positive powers of $z$ is equal to the power series
$$\langle w', (\Phi(v_1\otimes \cdots \otimes v_{i-1} \otimes e^{zD_V} v_i \otimes v_{i+1} \otimes \cdots \otimes v_n))(z_1, ..., z_n)\rangle$$
in $z$, which converges absolutely when $|z|<\min\limits_{1\leq i< j\leq n}|z_i-z_j|$.
\item For $v_1, ..., v_n\in V, w'\in W', (z_1, ..., z_n)\in F_n \C, z\in \C$ so that $(z_1+z, ..., z_n + z)\in F_n\C$, the power series expansion 
$$ \langle w', (\Phi(v_1\otimes\cdots\otimes v_n))(z_1+z,...,z_n+z)\rangle$$
in positive powers of $z$ is equal to the power series
$$\langle w', e^{zD_W} (\Phi(v_1\otimes \cdots \otimes v_n))(z_1, ..., z_n)\rangle. $$ 
which converges absolutely when $|z| < \min\limits_{1\leq i\leq n} |z_i|$
\end{enumerate}
\end{prop}

\begin{proof}
The argument of $D$-conjugation property carries over. 
\end{proof}

\begin{defn}
For every $n\in \mathbb{N}$, we define $\hat{C}_0^n(V, W)$ to be the set of all linear maps from $V^{\otimes n}\to \widetilde{W}_{z_1, ..., z_n}$ that satisfies $D$-derivative property and $\d$-conjugation property. 
\end{defn}

\begin{exam}
Let $l=0, 1, ..., n$. 
For every $w\in W$ such that for every $u\in V$, $Y_W^L(u,x)w\in W[[x]], Y_W^{s(R)}(u, x)w\in W[[x]]$, one checks easily that the map
$$u_1\otimes \cdots \otimes u_n \mapsto E_W^{(l, n-l)}(u_1\otimes\cdots u_n; w)$$
is a linear map $V^{\otimes n} \to \widetilde{W}_{z_1 ... z_n}$ that has the $D$-derivative property and $\d$-conjugation property. 
\end{exam}

\begin{nota}
We will use the notation $E_{W, w}^{(l, n-l)}$ to denote the map in the previous example. 
\end{nota}

Let $\Phi \in \hat{C}_0^n(V, W)$, $u^{(1)}, ..., u^{(n)} \in V$. Consider the following series of $\overline{W}$-valued rational functions:
\begin{align}
& \qquad \Phi (Y_V(u^{(1)}, z^{(1)}-\zeta_1)\one \otimes Y_V(u^{(2)}, z^{(2)}-\zeta_2)\one \otimes \cdots \otimes Y_V(u^{(n)}, z^{(n)} - \zeta_n) \one)(\zeta_1, ..., \zeta_n) \label{SeriesRatFunc-1}
\end{align}
which is a series in variables $z^{(i)}-\zeta_i, i = 1, ..., n$ with
$$\sum_{k_1, ..., k_n\in \Z} \Phi((Y_V)_{k_1}(u^{(1)})\one \otimes \cdots \otimes (Y_V)_{k_n}(u^{(n)})\one)(\zeta_1, ..., \zeta_n) (z^{(1)}-\zeta_1)^{-k_1-1}\cdots (z^{(n)}-\zeta_n)^{-k_n-1}$$
For each $k_1, ..., k_n\in \Z$, the coefficient of $(z^{(1)}-\zeta_1)^{-k_1-1}\cdots (z^{(n)}-\zeta_n)^{-k_n-1}$ is a $\overline{W}$-valued rational function with variables $\zeta_1, ..., \zeta_n$. If paired with $w'\in W'$, then for the complex series 
$$\sum_{k_1, ..., k_n\in \Z} \langle w', \Phi((Y_V)_{k_1}(u^{(1)})\one \otimes \cdots \otimes (Y_V)_{k_n}(u^{(n)})\one)(\zeta_1, ..., \zeta_n)\rangle (z^{(1)}-\zeta_1)^{-k_1-1}\cdots (z^{(n)}-\zeta_n)^{-k_n-1}$$
the coefficient of $(z^{(1)}-\zeta_1)^{-k_1-1}\cdots (z^{(n)}-\zeta_n)^{-k_n-1}$ is a rational function with possible poles at $\zeta_i = \zeta_j$ for $1\leq i < j \leq n$. 

\begin{prop}\label{0comp}
The series (\ref{SeriesRatFunc-1}) converges absolutely when 
$$|\zeta_i - \zeta_j|>|z^{(i)}-\zeta_i| + |z^{(j)}-\zeta_j|$$
to $\Phi(u_1\otimes \cdots \otimes u_n)(z^{(1)}, ..., z^{(n)})$. 
\end{prop}

\begin{proof}
From the creation property, we know that the series is the same as
$$\langle w', \Phi(e^{(z^{(1)}-\zeta_1) D_V}u_1 \otimes \cdots \otimes e^{(z^{(n)}-\zeta_n) D_V}u_n)(\zeta_1, ..., \zeta_n)\rangle$$
We repeatedly use Proposition \ref{D-conjugation} to see that the series converges absolutely to the rational function 
$$R(\langle w', \Phi(u_1\otimes\cdots\otimes u_n)(\zeta_1+z^{(1)}-\zeta_1, ..., \zeta_n+z^{(n)} - \zeta_n)\rangle) $$
when $|z^{(s)}-\zeta_s| < |\zeta_i - \zeta_j|, s=1, ..., n, s \leq i < j \leq n; |z^{(s)}-\zeta_s| < |z^{(t)}-\zeta_j|, s = 2, ..., n, 1\leq t < s \leq j \leq n$. 
Note that the rational function is the same as 
$$R(\langle w', \Phi(u_1\otimes \cdots \otimes u_n)(z^{(1)}, ..., z^{(n)})\rangle)$$
that has the only possible poles at $z^{(i)} = z^{(j)}, 1\leq i < j \leq n$ and does not depend on $\zeta_1, ..., \zeta_n.$ We then apply Lemma 4.5 in \cite{Q-Rep} to see that the series (\ref{SeriesRatFunc-1}) coincides with the expansion of the rational function by expanding the negative powers of $z^{(i)} - z^{(j)} = \zeta_i - \zeta_j + (z^{(i)} - \zeta_i - z^{(j)} + \zeta_j), 1\leq i < j \leq n$. Thus the series (\ref{SeriesRatFunc-1}) converges absolutely when 
$$|\zeta_i - \zeta_j| > |z^{(i)}- \zeta_i| + |z^{(j)} - \zeta_j|, 1\leq i < j \leq n$$
\end{proof}

%================================================================

\subsection{Linear maps $V^{\otimes n} \to \widetilde{W}_{z_1...z_n}$ composable with vertex operators}

\begin{defn}\label{Composable}
Let $\Phi: V^{\otimes n} \to \widetilde{W}_{z_1, ..., z_n}$ be a linear map. Let $m\in \Z_+$. $\Phi$ is said to be composable with $m$ vertex operators if for every $\alpha_0, \alpha_1, ..., \alpha_n\in \Z_+$ such that $\alpha_0+\cdots + \alpha_n = m+n$, every $l_0= 0, ..., \alpha_0$, and every $v_{1}^{(0)}, ..., v_{\alpha_0}^{(0)}, ..., v_1^{(n)}, ..., v_{\alpha_n}^{(n)}\in V$, the series of $\overline{W}$-valued rational functions
\begin{align*}
& \begin{aligned} 
Y_W^L(u_1^{(0)}, z_1^{(0)})& \cdots Y_W^L(u_{l_0}^{(0)}, z_{l_0}^{(0)})Y_W^{s(R)}(u_{l_0+1}^{(0)}, z_{l_0+1}^{(0)}) \cdots Y_W^{s(R)}(u_{\alpha_0}^{(0)}, z_{\alpha_0}^{(0)})\\
\cdot \Phi(& Y_V(v_1^{(1)}, z_1^{(1)}-\zeta_1)\cdots Y_V(v_{\alpha_1}^{(1)}, z_{\alpha_1}^{(1)}-\zeta_1)\one \\
\otimes & \cdots \\
\otimes & Y_V(v_1^{(n)}, z_1^{(n)}-\zeta_n)\cdots Y_V(v_{\alpha_n}^{(n)}, z_{\alpha_n}^{(n)}-\zeta_n)\one)(\zeta_1, ..., \zeta_n)
\end{aligned}\\
& \begin{aligned}
= \sum_{\substack{k_1^{(0)}, ..., k_{\alpha_0}^{(0)},\\ ..., k_1^{(n)}, ..., k_{\alpha_n}^{(n)}\in \Z}}(Y_W^L)&_{k_1^{(0)}}(u_1^{(0)})\cdots (Y_W^L)_{k_{l_0}^{(0)}}(u_{l_0}^{(0)})(Y_W^{s(R)})_{k_{l_0+1}^{(0)}}(u_{l_0+1}^{(0)})\cdots (Y_W^{s(R)})_{k_{\alpha_0}^{(0)}}(u_{\alpha_0}^{(0)})\\
\cdot \Phi(&(Y_V)_{k_1^{(1)}}(u_1^{(1)}) \cdots (Y_V)_{k_{\alpha_1}^{(1)}}(u_{\alpha_1}^{(1)})\one \\
\otimes & (Y_V)_{k_1^{(2)}}(u_1^{(2)}) \cdots (Y_V)_{k_{\alpha_2}^{(2)}}(u_{\alpha_2}^{(2)})\one  \\
\otimes & \cdots \\
\otimes & (Y_V)_{k_1^{(n)}}(u_1^{(n)}) \cdots (Y_V)_{k_{\alpha_n}^{(n)}}(u_{\alpha_n}^{(n)})\one )(\zeta_1, ..., \zeta_n)\rangle \\
& \prod_{i=1}^{\alpha_0}(z_i^{(0)})^{-k_i^{(0)}-1}\prod_{i=1}^n \prod_{j=1}^{\alpha_i}(z_j^{(i)}-\zeta_i)^{-k_j^{(i)}-1}
\end{aligned}
\end{align*}
% $$\begin{aligned}
% \sum_{r_1, ..., r_n \in \Z} \langle w',  \Phi(&\pi_{r_1} (E_V^{(\alpha_1)}(v_1^{(1)}\otimes \cdots \otimes v_{\alpha_1}^{(1)})) (z_1^{(1)}-\zeta_1, ...,  z_{\alpha_1}^{(1)}-\zeta_1) \\
% \otimes & \pi_{r_2}(E_V^{(\alpha_2)}(v_1^{(2)}\otimes \cdots \otimes v_{\alpha_2}^{(2)})) (z_1^{(2)}-\zeta_2, ...,  z_{\alpha_2}^{(2)}-\zeta_2), \\
% \otimes & \cdots \\
% \otimes & \pi_{r_n}(E_V^{(\alpha_n)}(v_1^{(n)}\otimes \cdots \otimes v_{\alpha_n}^{(n)})) (z_1^{(n)}-\zeta_n, ...,  z_{\alpha_n}^{(n)}-\zeta_n))(\zeta_1, ..., \zeta_n)\rangle 
% \end{aligned}$$
converges absolutely when 
\begin{align*}
& |z_1^{(0)}|> \cdots > |z_{\alpha_0}^{(0)}| > |\zeta_i|+|z_1^{(i)}-\zeta_i|, i = 1, ..., n;\\
& |z_1^{(i)}-\zeta_i|>\cdots > |z_{\alpha_i}^{(i)}-\zeta_i|, i = 1, ..., n; \\
& |z_s^{(i)}-\zeta_i - z_t^{(j)}+\zeta_j|<|\zeta_i-\zeta_j|, 1\leq i < j \leq	n, 1\leq s \leq \alpha_i, 1\leq t \leq \alpha_j.
\end{align*}
and the sum can be analytically extended to a rational function in $z_1^{(1)}, ..., z_{\alpha_1}^{(1)}, ..., z_1^{(n)}, ..., z_{\alpha_n}^{(n)}$ that is independent of $\zeta_1, ..., \zeta_n$ and has the only possible poles at $z_s^{(i)}= z_t^{(j)}$, for $1\leq i < j \leq n, s=1 ,..., \alpha_i, t = 1, ..., \alpha_j$. We require in addition that for each $i, j, s, t$, the order of the pole $z_s^{(i)}=z_t^{(j)}$ is bounded above by a constant that depends only on $u_s^{(i)}$ and $u_t^{(j)}$. 
% \item For $l=0, ..., m$, $v_1, ..., v_{m+n}\in V$ and $w' \in W$, the complex series
% $$\langle w', Y_W^L(v_1, z_1)\cdots Y_W^L(v_l, z_l)Y_W^{s(R)}(v_{l+1}, s_{l+1})\cdots Y_W^{s(R)}(v_n, z_n) \arc{w_\Phi}\rangle$$
% where $\arc{w_\Phi}$ is the $\overline{W}$-element defined by 
% $$(\Phi ({v_{m + 1}} \otimes  \cdots  \otimes {v_{ n}}))({z_{m + 1}}, ..., z_n), $$
% converges absolutely when $|z_1|>\cdots > |z_n|>|z_i|, i \geq n+1$, and the sum can be analytically extended to a rational function with the only possible poles at $z_i=z_j, 1\leq i < j \leq n+m$
%$$\langle w', Y_W^L(u_1, z_1)\cdots Y_W^L(u_l, z_l)Y_W^{s(R)}(u_{l+1}, z_{l+1})\cdots Y_W^{s(R)}(u_m, z_m)\arc{w_\Phi}\rangle$$
% where $\arc{w_\Phi}$ is the $\overline{W}$-element defined by 
% $$(\Phi ({v_{m + 1}} \otimes  \cdots  \otimes {v_{ n}}))({z_{m + 1}}, ..., z_n), $$
% converges absolutely in the region
% $$\{(z_1, ..., z_{n+m})\in \C^{n+m}: |z_i|>|z_j|, i=1,...,m, j=m+1, ..., m+n\}$$
%$$\left\{(z_1, ..., z_{n})\in \C^{n}: \begin{aligned}
%&|z_i|>|z_{l+n+1}|+|z_{l+n+1}-z_{l+j}|, 1\leq i \leq l, 1\leq j \leq n;\\
%& |z_{l+n+1}|>|z_{l+i}-z_{l+n+1}|+|z_{l+n+j}|, 1 \leq i \leq n; 1 \leq j \leq m-l;
%\end{aligned}\right\}$$
% and the sum can be analytically extended to a rational function in $z_1, ..., z_{m+n}$ with the only possible poles at $z_i = z_j$.
\end{defn}

\begin{defn}
We denote by $\hat{C}_m^n(V, W)$ the set of linear maps $V^{\otimes n} \to \widetilde{W}_{z_1, ..., z_n}$ in $\hat{C}_0^n(V, W)$ and are composable with $m$ vertex operators. It is easy to see that
$$\hat{C}_0^n(V, W) \supseteq \hat{C}_1^n(V, W) \supseteq \hat{C}_2^n(V, W) \supseteq \cdots  $$
We denote by $\hat{C}_\infty^n(V, W)$ the intersection of all $\hat{C}_m^n(V, W)$ for $m=0, 1, 2,...$. 
\end{defn}

\begin{exam}
Fix $n\in \Z_+$ and $l = 0, ..., n$. For $w\in W$ satisfying $\forall u\in V, Y_W^L(u, x)w\in W[[x]]$ and $Y_W^R(w, x)u\in W[[x]]$, the map $E_{W;w}^{(l, n-l)}$ is an element in $\hat{C}_\infty^{n}(V, W)$. 
\end{exam}

\begin{proof}
This was proved in Proposition \ref{ComposableExample}. 
\end{proof}

\begin{rema}
If $V$ is a grading-restricted vertex algebra and $W$ is a grading-restricted $V$-module, then $W$ can be viewed as a $V$-bimodule when we regard $V$ as a MOSVA. One can check easily that Definition \ref{Composable} is equivalent to Definition 3.5 in \cite{Hcoh}. The sets $\hat{C}_m^n(V, W)$ we are defining here is precisely the same as the $\hat{C}^n_m (V, W)$ in \cite{Hcoh}.  
\end{rema}

\begin{rema}
Let $\Phi \in \hat{C}_1^n(V, W)$. Then Definition \ref{Composable} implies the following
\begin{enumerate}
\item For every $\beta = 1, ..., n$, $v^{(1)}, ..., v_1^{(\beta)}, v_2^{(\beta)}, ..., v^{(n)}\in V$, the series
$$\Phi(Y_V(v^{(1)}, z^{(1)}-\zeta_1)\one \otimes \cdots \otimes Y_V(v_1^{(\beta)}, z_1^{(\beta)}-\zeta_\beta)Y_V(v_2^{(\beta)}, z_2^{(\beta)} - \zeta_\beta)\one \otimes \cdots \otimes Y_V(v^{(n)}, z^{(n)}-\zeta_n)\one)(\zeta_1, ..., \zeta_n) $$
which expands as 
\begin{align*}\sum_{k^{(1)}, ..., k_1^{(\beta)}, k_2^{(\beta)}, ..., k^{(n)}\in \Z}\Phi((Y_V)_{k^{(1)}}(v^{(1)})\one \otimes \cdots \otimes (Y_V)_{k_1^{(\beta)}}(v_1^{(\beta)})(Y_V)_{k_2^{(\beta)}}\one \otimes \cdots \otimes (Y_V)_{k^{(n)}}(v^{(n)})\one)(\zeta_1, ..., \zeta_n)\\
\cdot (z^{(1)}-\zeta_1)^{-k^{(1)}-1}\cdots (z_1^{(\beta)} - \zeta_\beta)^{-k_1^{(\beta)}-1}(z_2^{(\beta)} - \zeta_\beta)^{-k_2^{(\beta)} -1}\cdots (z^{(n)}-\zeta_n)^{-k^{(n)}-1}
\end{align*}
converges absolutely when
\begin{align*}
& |z_1^{(\beta)}-\zeta_\beta| > |z_2^{(\beta)}-\zeta_\beta|, |\zeta_\beta - \zeta_i| > |z_s^{(\beta)}-\zeta_\beta| + |z^{(i)} - \zeta_i|, s = 1, 2, 1\leq i \leq n, i\neq \beta. \\
& |\zeta_i- \zeta_j| > |z^{(i)}- \zeta_i| + |z^{(j)}-\zeta_j|, 
1\leq i < j \leq n, i, j \neq \beta, 
\end{align*}
to a rational function that depends only on $z^{(1)}, ..., z_1^{(i)}, z_2^{(i)}, ..., z^{(n)}$, with the only possible poles at $z_1^{(\beta)} = z_2^{(\beta)}; z_s^{(\beta)} = z^{(i)}, s=1, 2, i =1, ..., \beta-1, \beta+1, ..., n; z^{(i)}=z^{(j)}, 1\leq i < j \leq n, i, j\neq \beta$. 
From arguments similar to those in Proposition \ref{0comp}, this is equivalent to say that the following series 
\begin{align*}
 \Phi(v^{(1)}\otimes \cdots \otimes Y_V(v_1^{(\beta)}, z_1^{(\beta)}-\zeta_\beta)Y_V(v_2^{(\beta)}, z_2^{(\beta)}-\zeta_\beta)\one \otimes \cdots \otimes v^{(n)})(z^{(1)}, ..., \zeta_\beta ,..., z^{(n)})\\
 \begin{aligned}
= \sum_{k_1^{(\beta)}, k_2^{(\beta)}\in \Z}\Phi(v^{(1)} \otimes \cdots \otimes (Y_V)_{k_1^{(\beta)}}(v_1^{(\beta)})(Y_V)_{k_2^{(\beta)}}\one \otimes \cdots \otimes v^{(n)})(z^{(1)}, ..., \zeta_\beta, ..., z^{(n)})\\
(z_1^{(\beta)} - \zeta_\beta)^{-k_1^{(\beta)}-1}(z_2^{(\beta)} - \zeta_\beta)^{-k_2^{(\beta)} -1}
\end{aligned}
\end{align*}
of $\overline{W}$-valued rational functions converges absolutely when 
$$|z_1^{(\beta)}-\zeta_\beta| > |z_2^{(\beta)}-\zeta_\beta|, |z^{(i)}-z_s^{(\beta)}| > |z_s^{\beta} - \zeta_\beta|, i=1, ..., \beta-1, \beta+1, ..., n, s = 1, 2. $$
\item For $u_1, ..., u_{n+1}\in V$, the series
\begin{align*}
& Y_W^L(u_1, z_1)\Phi(u_2\otimes \cdots \otimes u_{n+1})(z_2, ..., z_{n+1})\\
= & \sum_{k\in \Z} (Y_W^L)_{k}(u_1)\Phi(u_2\otimes \cdots \otimes u_{n+1})(z_2, ..., z_{n+1}) z_1^{-k-1}
\end{align*}
of $\overline{W}$-elements converges absolutely when $|z_1|> |z_i|>0, i = 2, ..., n+1$ to an $\overline{W}$-valued rational function (here the operator $(Y_W^L)_k(u_1)$ is extended to $\overline{W} \to \overline{W}$.) 
\item For $u_1, ..., u_{n+1}\in V$, the series
\begin{align*}
& Y_W^{s(R)}(u_{n+1}, z_{n+1})\Phi(u_1\otimes \cdots \otimes u_{n})(z_1, ..., z_{n})\\
= & \sum_{k\in \Z} (Y_W^{s(R)})_{k}(u_{n+1})\Phi(u_1\otimes \cdots \otimes u_{n})(z_1, ..., z_{n}) z_{n+1}^{-k-1}
\end{align*}
of $\overline{W}$-elements converges absolutely when $|z_{n+1}|> |z_i|>0, i = 1, ..., n$ to an $\overline{W}$-valued rational function (here the operator $(Y_W^{s(R)})_k(u_{n+1})$ is extended to $\overline{W} \to \overline{W}$.) 
\end{enumerate}

\end{rema}

\begin{defn}
Let $\Phi \in \hat{C}_m^n(V, W)$. 
\begin{enumerate}
\item For every $i=1, ..., n$, we define the map  
$$\Phi\circ_i E_V^{(2)}: V^{\otimes (n+1)}\to \widetilde{W}_{z_1, ..., z_{n+1}}$$
by setting  
$$(\Phi\circ_i E_V^{(2)})(v_1\otimes \cdots \otimes v_{n+1})$$
to be the $\overline{W}$-valued rational function
$$E(\Phi(v_1 \otimes \cdots \otimes Y_V(v_i, z_i-\zeta) Y_V(v_{i+1}, z_{i+1}-\zeta)\one \otimes v_{i+2} \otimes \cdots \otimes v_{n+1})(z_1, ..., z_{i-1}, \zeta, z_{i+2}, ..., z_{n+1}))$$
\item We define the map
$$E_W^{(1, 0)}\circ_{2}\Phi:  V^{\otimes (n+1)}\to \widetilde{W}_{z_1, ..., z_{n+1}}$$
by setting  
$$(E_W^{(1, 0)}\circ_{2}\Phi)(v_1\otimes \cdots \otimes v_{n+1})$$
to be the $\overline{W}$-valued rational function
$$E(Y_W^L(v_1, z_1)\Phi(v_{2}\otimes \cdots \otimes v_{n+1})(z_{2}, ..., z_{n+1}))$$
\item We define the map
$$E_W^{(0, 1)}\circ_{2}\Phi:  V^{\otimes (n+1)}\to \widetilde{W}_{z_1, ..., z_{n+1}}$$
by setting  
$$(E_W^{(0, 1)}\circ_{2}\Phi)(v_1\otimes \cdots \otimes v_{n+1})$$
to be the $\overline{W}$-valued rational function 
$$E(Y_W^{s(R)}(v_{n+1}, z_{n+1})\Phi(v_{1}\otimes \cdots \otimes v_{n})(z_{1}, ..., z_{n}))$$
\end{enumerate}
\end{defn}

\begin{prop}\label{deltasummands}
The maps $\Phi\circ_i E_V^{(2)}, E_W^{(1, 0)}\circ_2 \Phi$ and $E_W^{(0, 1)}\circ_2 \Phi$ are elements of $\hat{C}^{n+1}_{m-1}(V, W)$
\end{prop}

\begin{proof}
Let $\alpha_0, ..., \alpha_{n+1}\in \N$ such that $\alpha_0 + \cdots + \alpha_{n+1} = m+n$. Let $l_0 = 0, ..., \alpha_0$. Take $v_s^{(j)}\in V, j = 0, 1, , ..., n+1, s = 1, ..., \alpha_j$. 
\begin{enumerate}
\item For the first conclusion, we first note that the associativity implies that
\begin{align*}
& Y_V(v_1^{(i)}, z_1^{(i)}-\zeta)\cdots Y_V(v_{\alpha_i}^{(i)}, z_{\alpha_i}^{(i)}-\zeta)Y_V(v_1^{(i+1)}, z_1^{(i+1)}-\zeta)\cdots Y_V(v_{\alpha_{i+1}}^{(i+1)}, z_{\alpha_{i+1}}^{(i+1)}-\zeta)\one \\
= & Y_V(Y_V(v_1^{(i)}, z_1^{(i)}-\zeta_i)\cdots Y_V(v_{\alpha_i}^{(i)}, z_{\alpha_i}^{(i)}-\zeta_i)\one, \zeta_i -\zeta)\\
& \qquad \cdot Y_V(Y_V(v_1^{(i+1)}, z_1^{(i+1)}-\zeta_{i+1})\cdots Y_V(v_{\alpha_{i+1}}^{(i+1)}, z_{\alpha_{i+1}}^{(i+1)}-\zeta_{i+1})\one, \zeta_{i+1}-\zeta)\one
\end{align*}
when 
\begin{align*}
&|z_1^{(i)}-\zeta| > \cdots > |z_{\alpha_i}^{(i)}-\zeta| > |z_1^{(i+1)}-\zeta| > \cdots > |z_{\alpha_{i+1}}^{(i+1)}-\zeta|;\\
&|z_s^{(i)}-\zeta_i|>|z_t^{(i)}-\zeta_i|, 1\leq s < t\leq \alpha_i; 
|z_s^{(i+1)}-\zeta_{i+1}|>|z_t^{(i+1)}-\zeta_{i+1}|, 1\leq s < t\leq \alpha_{i+1};\\
&|\zeta_i - \zeta| > |\zeta_{i+1} - \zeta| + |z_1^{(i)}- \zeta_i|+ |z_1^{(i+1)} - \zeta_{i+1}|.
\end{align*}
Let $A$ denote the region in $\C^n$ such that the coordinates satisfies the inequalities above. We also note that since $\Phi\in \hat{C}_m^n(V, W)$, the series 
\begin{align}
Y_W^L(u_1^{(0)},& z_1^{(0)}) \cdots Y_W^L(u_{l_0}^{(0)}, z_{l_0}^{(0)})Y_W^{s(R)}(u_{l_0+1}^{(0)}, z_{l_0+1}^{(0)}) \cdots Y_W^{s(R)}(u_{\alpha_0}^{(0)}, z_{\alpha_0}^{(0)}) \nn \\
\cdot \Phi(& Y_V(v_1^{(1)}, z_1^{(1)}-\zeta_1)\cdots Y_V(v_{\alpha_1}^{(1)}, z_{\alpha_1}^{(1)}-\zeta_1)\one \nn\\
\otimes & \cdots \nn\\
\otimes & Y_V(v_1^{(i)}, z_1^{(i)}-\zeta)\cdots Y_V(v_{\alpha_i}^{(i)}, z_{\alpha_i}^{(i)}-\zeta)Y_V(v_1^{(i+1)}, z_1^{(i+1)}-\zeta)\cdots Y_V(v_{\alpha_{i+1}}^{(i+1)}, z_{\alpha_{i+1}}^{(i+1)}-\zeta)\one  \nn\\
\otimes & \cdots \nn\\
\otimes & Y_V(v_1^{(n+1)}, z_1^{(n+1)}-\zeta_{n+1})\cdots Y_V(v_{\alpha_{n+1}}^{(n+1)}, z_{\alpha_{n+1}}^{(n+1)}-\zeta_{n+1})\one)(\zeta_1, ..., \zeta_{i-1}, \zeta, 
\zeta_{i+2}, ...,  \zeta_{n+1}) \label{Arg-Series-1}
\end{align}
converges absolutely to the $\overline{W}$-valued rational function 
\begin{align} 
E(Y_W^L&(u_1^{(0)}, z_1^{(0)}) \cdots Y_W^L(u_{l_0}^{(0)}, z_{l_0}^{(0)})Y_W^{s(R)}(u_{l_0+1}^{(0)}, z_{l_0+1}^{(0)}) \cdots Y_W^{s(R)}(u_{\alpha_0}^{(0)}, z_{\alpha_0}^{(0)})\nn \\
\cdot \Phi(& Y_V(v_1^{(1)}, z_1^{(1)}-\zeta_1)\cdots Y_V(v_{\alpha_1}^{(1)}, z_{\alpha_1}^{(1)}-\zeta_1)\one \nn\\
\otimes & \cdots \nn\\
\otimes & Y_V(v_1^{(i)}, z_1^{(i)}-\zeta)\cdots Y_V(v_{\alpha_i}^{(i)}, z_{\alpha_i}^{(i)}-\zeta)Y_V(v_1^{(i+1)}, z_1^{(i+1)}-\zeta)\cdots Y_V(v_{\alpha_{i+1}}^{(i+1)}, z_{\alpha_{i+1}}^{(i+1)}-\zeta)\one \nn\\
\otimes & \cdots \nn\\
\otimes & Y_V(v_1^{(n+1)}, z_1^{(n+1)}-\zeta_{n+1})\cdots Y_V(v_{\alpha_{n+1}}^{(n+1)}, z_{\alpha_{n+1}}^{(n+1)}-\zeta_{n+1})\one)(\zeta_1, ..., \zeta_{i-1}, \zeta, \zeta_{i+2}, ..., \zeta_{n+1}))\label{Arg-Series-1-Limit}
\end{align}
that does not depend on $\zeta_1, ..., \zeta_{i-1}, \zeta,\zeta_{i+2} ..., \zeta_{n+1}$ in a region $B$ in $\C^n$. The region $B$ is defined by the following inequalities: 
\begin{align*}
& |z_1^{(0)}|> \cdots > |z_{\alpha_0}^{(0)}| > |\zeta_j|+|z_t^{(j)}-\zeta_j|, j = 1, ..., i-1, i+2, ...,  n, t = 1, ..., \alpha_j;\\
& |z_{\alpha_0}^{(0)}| > |\zeta|+|z_t^{(j)}-\zeta| , j = i, i+1; \\
& |z_1^{(j)}-\zeta_j|>\cdots > |z_{\alpha_j}^{(j)}-\zeta_j|, j = 1, ..., i-1, i+2, ..., n; \\
& |z_1^{(i)}- \zeta|> \cdots > |z_{\alpha_i}^{(i)} - \zeta| > |z_1^{(i+1)} - \zeta| > \cdots > |z_{\alpha_{i+1}}^{(i+1)}- \zeta|;\\
& |z_s^{(j)}-\zeta_j - z_t^{(k)}+\zeta_k|<|\zeta_j-\zeta_k|, 1\leq j < k \leq n, j \neq i, j\neq i+1, k \neq i, k\neq i+1, \\
& \hspace{5.5cm} 1\leq s \leq \alpha_j, 1\leq t \leq \alpha_k;\\
& |z_s^{(j)}-\zeta - z_t^{(k)}+\zeta_k|<|\zeta-\zeta_k|, i \leq j \leq i+1, 1\leq k \leq n, k \neq i, k\neq i+1, \\
& \hspace{5.2cm} 1\leq s \leq \alpha_j, 1\leq t \leq \alpha_k. 
\end{align*}
Let $f$ denote the $\overline{W}$-valued rational function defined by (\ref{Arg-Series-1-Limit}). Note that neither the series (\ref{Arg-Series-1}) nor the sum (\ref{Arg-Series-1-Limit}) involves $\zeta_i, \zeta_{i+1}$. 

It is easy to check that $A\cap B$ is nonempty. So in $A\cap B$, the following series  
\begin{align}
Y_W^L(u_1^{(0)}, z_1^{(0)})& \cdots Y_W^L(u_{l_0}^{(0)}, z_{l_0}^{(0)})Y_W^{s(R)}(u_{l_0+1}^{(0)}, z_{l_0+1}^{(0)}) \cdots Y_W^{s(R)}(u_{\alpha_0}^{(0)}, z_{\alpha_0}^{(0)})\nn\\
\cdot \Phi(& Y_V(v_1^{(1)}, z_1^{(1)}-\zeta_1)\cdots Y_V(v_{\alpha_1}^{(1)}, z_{\alpha_1}^{(1)}-\zeta_1)\one \nn\\
\otimes & \cdots \nn\\
\otimes & Y_V(Y_V(v_1^{(i)}, z_1^{(i)}-\zeta_i)\cdots Y_V(v_{\alpha_i}^{(i)}, z_{\alpha_i}^{(i)}-\zeta_i)\one, \zeta_i -\zeta)\nn\\
& \qquad \cdot Y_V(Y_V(v_1^{(i+1)}, z_1^{(i+1)}-\zeta_{i+1})\cdots Y_V(v_{\alpha_{i+1}}^{(i+1)}, z_{\alpha_{i+1}}^{(i+1)}-\zeta_{i+1})\one, \zeta_{i+1}-\zeta)\one\nn\\
\otimes & \cdots \nn\\
\otimes & Y_V(v_1^{(n)}, z_1^{(n)}-\zeta_n)\cdots Y_V(v_{\alpha_n}^{(n)}, z_{\alpha_n}^{(n)}-\zeta_n)\one)(\zeta_1, ...,\zeta_{i-1}, \zeta, \zeta_{i+2}..., \zeta_n)\label{Arg-Series-2}
\end{align}
converges absolutely to the $\overline{W}$-valued rational function $f$. We analyze the series and find that it is the expansion of $f$ in the following way:
\begin{enumerate}
    \item For $j=0, 1, ..., n, 1\leq s < t\leq \alpha_0$, expand the negative powers of $z_s^{(0)}-z_t^{(j)}$ as power series in $z_t^{(j)}$. When $j\geq 1$ and $j\neq i, j\neq i+1$, further expand positive powers of $z_t^{(j)} = \zeta_j + (z_t^{(j)} - \zeta_j)$ as polynomials in $\zeta_j$ and $z_t^{(j)}-\zeta_j$. When $j=i$ or $j=i+1$, further expand positive powers of $z_t^{(j)} = \zeta + (\zeta_j-\zeta) + (z_t^{(j)} - \zeta_j)$ as polynomials in $\zeta, \zeta_j-\zeta$ and $z_t^{(j)}-\zeta_j$.
    \item For $j = 1,..., n, 1\leq s< t\leq \alpha_j$, expand negative powers of $z_s^{(j)}-z_t^{(j)}=(z_s^{(j)} - \zeta_j) - (z_t^{(j)}- \zeta_j)$ as power series in $z_t^{(j)}-\zeta_j$. 
    \item For $1\leq j < k \leq n, j,k \notin\{i, i+1\}, 1\leq s \leq \alpha_j, 1\leq t \leq \alpha_k$, expand the negative powers of $z_s^{(j)}-z_t^{(k)} = (\zeta_j-\zeta_k) + (z_s^{(j)} - \zeta_j - z_t^{(k)} + \zeta_k)$ as a power series of $(z_s^{(j)} - \zeta_j - z_t^{(k)} + \zeta_k)$, then further expand the positive powers of $(z_s^{(j)} - \zeta_j - z_t^{(k)} + \zeta_k)$ as polynomials of $(z_s^{(j)} - \zeta_j)$ and  $(z_t^{(k)} - \zeta_k)$. 
    \item For $j=1,..., i-1, i+2, ..., n, k = i, i+1, 1\leq s \leq \alpha_j, 1\leq t \leq \alpha_k$, expand the negative powers of $z_s^{(j)}-z_t^{(k)} = (\zeta_j-\zeta) + (z_s^{(j)} - \zeta_j - z_t^{(k)} + \zeta_k + \zeta- \zeta_k)$ as a power series of $(z_s^{(j)} - \zeta_j - z_t^{(k)} + \zeta_k + \zeta- \zeta_k)$, then further expand the positive powers of $(z_s^{(j)} - \zeta_j - z_t^{(k)} + \zeta_k + \zeta- \zeta_k)$ as polynomials of $(z_s^{(j)} - \zeta_j)$, $(z_t^{(k)} - \zeta_k)$ and $(\zeta - \zeta_k)$. 
    \item For $j=i, k=i+1, 1\leq s \leq \alpha_j, 1\leq t \leq \alpha_k$, expand the negative powers of $z_s^{(i)} - z_t^{(i+1)} = (\zeta_{i}-\zeta) + (\zeta - \zeta_{i+1} + z_s^{(i)} - \zeta_i - z_t^{(i+1)} + \zeta_{i+1})$. then further expand the positive powers of $ (\zeta - \zeta_{i+1} + z_s^{(i)} - \zeta_i - z_t^{(i+1)} + \zeta_{i+1})$ as polynomials of $ (\zeta - \zeta_{i+1})$, $(z_s^{(i)} - \zeta_i)$ and $(z_t^{(i+1)} - \zeta_{i+1})$. 
\end{enumerate}
So the series (\ref{Arg-Series-2}) converges absolutely to (\ref{Arg-Series-1-Limit}) in the region $C$ defined by
\begin{align*}
& |z_1^{(0)}|> \cdots > |z_{\alpha_0}^{(0)}| > |\zeta_j|+|z_1^{(j)}-\zeta_j|, j = 1, ..., i-1, i+2, ...,  n;\\
& |z_{\alpha_0}^{(0)}| > |\zeta|+|\zeta_j - \zeta|+ |z_1^{(j)}-\zeta_j| , j = i, i+1; \\
& |z_1^{(j)}-\zeta_j|>\cdots > |z_{\alpha_j}^{(j)}-\zeta_j|, j = 1, ..., n; \\
& |z_s^{(j)}-\zeta_j - z_t^{(k)}+\zeta_k|<|\zeta_j-\zeta_k|, 1\leq j < k \leq n, j \neq i, j\neq i+1, k \neq i, k\neq i+1; \\
& \hspace{5.5cm} 1\leq s \leq \alpha_j, 1\leq t \leq \alpha_k;\\
& |z_s^{(j)} - \zeta_j - z_t^{(k)} + \zeta| < |\zeta- \zeta_j|, 1\leq j\leq n, j\neq i, j\neq i+1, i\leq k \leq i+1,\\
& \hspace{5.2cm} 1\leq s \leq \alpha_j, 1\leq t \leq \alpha_k;\\
& |\zeta+z_s^{(i)} - \zeta_i - z_t^{(i+1)}| < |\zeta_i-\zeta|. 
\end{align*}
Now we note that by definition of $\Phi\circ_i E_V^{(2)}$, 
\begin{align*}
(\Phi  \circ_i E_V^{(2)})&( Y_V(v_1^{(1)}, z_1^{(1)}-\zeta_1)\cdots Y_V(v_{\alpha_1}^{(1)}, z_{\alpha_1}^{(1)}-\zeta_1)\one \\
\otimes & \cdots \\
\otimes & Y_V(v_1^{(i)}, z_1^{(i)}-\zeta_i)\cdots Y_V(v_{\alpha_i}^{(i)}, z_{\alpha_i}^{(i)}-\zeta_i)\one \\
\otimes &  Y_V(v_1^{(i+1)}, z_1^{(i+1)}-\zeta_{i+1})\cdots Y_V(v_{\alpha_{i+1}}^{(i+1)}, z_{\alpha_{i+1}}^{(i+1)}-\zeta_{i+1})\one\\
\otimes & \cdots \\
\otimes & Y_V(v_1^{(n+1)}, z_1^{(n+1)}-\zeta_{n+1})\cdots Y_V(v_{\alpha_{n+1}}^{(n+1)}, z_{\alpha_{n+1}}^{(n+1)}-\zeta_{n+1})\one)(\zeta_1, ..., \zeta_{n+1})
\end{align*}
is the sum of the series 
\begin{align*}
\Phi(& Y_V(v_1^{(1)}, z_1^{(1)}-\zeta_1)\cdots Y_V(v_{\alpha_1}^{(1)}, z_{\alpha_1}^{(1)}-\zeta_1)\one \\
\otimes & \cdots \\
\otimes & Y_V(Y_V(v_1^{(i)}, z_1^{(i)}-\zeta_i)\cdots Y_V(v_{\alpha_i}^{(i)}, z_{\alpha_i}^{(i)}-\zeta_i)\one, \zeta_i -\zeta)\\
& \qquad \cdot Y_V(Y_V(v_1^{(i+1)}, z_1^{(i+1)}-\zeta_{i+1})\cdots Y_V(v_{\alpha_{i+1}}^{(i+1)}, z_{\alpha_{i+1}}^{(i+1)}-\zeta_{i+1})\one, \zeta_{i+1}-\zeta)\one\\
\otimes & \cdots \\
\otimes & Y_V(v_1^{(n)}, z_1^{(n)}-\zeta_n)\cdots Y_V(v_{\alpha_n}^{(n)}, z_{\alpha_n}^{(n)}-\zeta_n)\one)(\zeta_1,..., \zeta, ..., \zeta_n)
\end{align*}
with respect to the variables involving $\zeta$. Therefore, the series 
\begin{align}
& \begin{aligned}
Y_W^L(u_1^{(0)} , z_1^{(0)})\cdots &Y_W^L(u_{l_0}^{(0)}, z_{l_0}^{(0)})Y_W^{s(R)}(u_{l_0+1}^{(0)}, z_{l_0+1}^{(0)})\cdots Y_W^{s(R)}(u_{\alpha_0}^{(0)}, z_{\alpha_0}^{(0)})\\
\cdot (\Phi \circ_i E_V^{(2)})&(Y_V(v_1^{(1)}, z_1^{(1)}-\zeta_1)\cdots Y_V(v_{\alpha_1}^{(1)}, z_{\alpha_1}^{(1)}-\zeta_1)\one \\ 
 &\otimes \cdots \otimes Y_V(v_1^{(n+1)}, z_1^{(n+1)}-\zeta_{n+1})\cdots Y_V(v_{\alpha_{n+1}}^{(1)}, z_{\alpha_{n+1}}^{(n+1)}-\zeta_{n+1})\one )(\zeta_1, ..., \zeta_{n+1})
\end{aligned}\nn\\
& \begin{aligned} =\sum_{\substack{k_1^{(0)}, ..., k_{\alpha_0}^{(0)}\\ ... k_1^{(n+1)}, ..., k_{\alpha_{n+1}}^{(n+1)} \in \Z } } (Y_W^L)_{k_1^{(0)}}(u_1^{(0)}) & \cdots (Y_W^L)_{k_{l_0}^{(0)}}(u_{l_0}^{(0)}) (Y_W^{s(R)})_{k_{l_0+1}}(u_{l_0+1}^{(0)}) \cdots (Y_W^{s(R)})_{k_{\alpha_0}}(u_{\alpha_0}^{(0)})\\ 
\cdot (\Phi\circ_i E_V^{(2)})(& (Y_V)_{k_{1}^{(1)}}(v_{1}^{(1)})\cdots (Y_V)_{k_{\alpha_1}^{(1)}}(v_{\alpha_1}^{(1)})\one \\ & 
\otimes \cdots \otimes (Y_V)_{k_{1}^{(n+1)}}(v_{1}^{(n+1)})\cdots (Y_V)_{k_{\alpha_{n+1}}^{(n+1)}}(v_{\alpha_{n+1}}^{(n+1)})\one)(\zeta_1, ..., \zeta_{n+1})\\
 \cdot \prod_{j=1}^{\alpha_0}  (z_j^{(0)})^{-k_j^{(0)}-1} & \prod_{j=1}^n  (z_1^{(j)}-\zeta_j)^{-k_1^{(j)}-1} \cdots  (z_{\alpha_j}^{(j)}-\zeta_j)^{-k_{\alpha_j}^{(j)}-1} 
\end{aligned}\label{Arg-Series-3}
\end{align}
is the expansion of the $\overline{W}$-rational function $f$ which do not involve any expansion of $\zeta$. More precisely, it is the expansion of $f$ in the following way:
\begin{enumerate}
    \item For $j=0, 1, ..., n, 1\leq s < t\leq \alpha_0$, expand the negative powers of $z_s^{(0)}-z_t^{(j)}$ as power series in $z_t^{(j)}$. When $j\geq 1$, further expand positive powers of $z_t^{(j)} = \zeta_j + (z_t^{(j)} - \zeta_j)$ as polynomials in $\zeta_j$ and $z_t^{(j)}-\zeta_j$. 
    \item For $j = 1,..., n, 1\leq s< t\leq \alpha_j$, expand negative powers of $z_s^{(j)}-z_t^{(j)}=(z_s^{(j)} - \zeta_j) - (z_t^{(j)}- \zeta_j)$ as power series in $z_t^{(j)}-\zeta_j$. 
    \item For $1\leq j < k \leq n, 1\leq s \leq \alpha_j, 1\leq t \leq \alpha_k$, expand the negative powers of $z_s^{(j)}-z_t^{(k)} = (\zeta_j-\zeta_k) + (z_s^{(j)} - \zeta_j - z_t^{(k)} + \zeta_k)$ as a power series of $(z_s^{(j)} - \zeta_j - z_t^{(k)} + \zeta_k)$, then further expand the positive powers of $(z_s^{(j)} - \zeta_j - z_t^{(k)} + \zeta_k)$ as polynomials of $(z_s^{(j)} - \zeta_j)$ and  $(z_t^{(k)} - \zeta_k)$. 
\end{enumerate}
In other words, the series (\ref{Arg-Series-3}) converges absolutely to the $\overline{W}$-valued rational function (\ref{Arg-Series-1-Limit}) in the region
\begin{align*}
& |z_1^{(0)}|> \cdots > |z_{\alpha_0}^{(0)}| > |\zeta_i|+|z_1^{(i)}-\zeta_i|, i = 1, ..., n+1;\\
& |z_1^{(i)}-\zeta_i|>\cdots > |z_{\alpha_i}^{(i)}-\zeta_i|, i = 1, ..., n+1; \\
& |z_s^{(i)}-\zeta_i - z_t^{(j)}+\zeta_j|<|\zeta_i-\zeta_j|, 1\leq i < j \leq n+1, 1\leq s \leq \alpha_i, 1\leq t \leq \alpha_j.
\end{align*}
This proves that $\Phi \circ_i E_V^{(2)}$ is in $\hat{C}_{m-1}^{n+1}(V, W)$. 
\item First, from the assumption that $\Phi$ is composable with $m$ vertex operators, we know that the series 
\begin{align*} 
Y_W^L(v_1^{(0)}, z_1^{(0)})& \cdots Y_W^L(v_{l_0}^{(0)}, z_{l_0}^{(0)})
\cdot Y_W^L(v_1^{(1)}, z_1^{(1)}) \cdots Y_W^L(v_{\alpha_1}^{(1)}, z_{\alpha_1}^{(1)})\cdot  Y_W^{s(R)}(v_{l_0+1}^{(0)}, z_{l_0+1}^{(0)}) \cdots Y_W^{s(R)}(v_{\alpha_0}^{(0)}, z_{\alpha_0}^{(0)})\\
\cdot \Phi(& Y_V(v_1^{(2)}, z_1^{(2)}-\zeta_2)\cdots Y_V(v_{\alpha_2}^{(2)}, z_{\alpha_2}^{(1)}-\zeta_2)\one \\
\otimes & \cdots \\
\otimes & Y_V(v_1^{(n+1)}, z_1^{(n+1)}-\zeta_{n+1})\cdots Y_V(v_{\alpha_{n+1}}^{(n+1)}, z_{\alpha_{n+1}}^{(n+1)}-\zeta_{n+1})\one)(\zeta_2, ...,  \zeta_{n+1})
\end{align*}
converges absolutely to a $\overline{W}$-valued rational function \begin{align} 
E(& Y_W^L(v_1^{(0)}, z_1^{(0)}) \cdots Y_W^L(v_{l_0}^{(0)}, z_{l_0}^{(0)})
\cdot Y_W^L(v_1^{(1)}, z_1^{(1)}) \cdots Y_W^L(v_{\alpha_1}^{(1)}, z_{\alpha_1}^{(1)})\nn\\ 
&\cdot  Y_W^{s(R)}(v_{l_0+1}^{(0)}, z_{l_0+1}^{(0)}) \cdots Y_W^{s(R)}(v_{\alpha_0}^{(0)}, z_{\alpha_0}^{(0)})\nn\\
& \cdot \Phi( Y_V(v_1^{(2)}, z_1^{(2)}-\zeta_2)\cdots Y_V(v_{\alpha_2}^{(2)}, z_{\alpha_2}^{(1)}-\zeta_2)\one \nn\\
&\qquad \otimes  \cdots \nn\\
&\qquad \otimes  Y_V(v_1^{(n+1)}, z_1^{(n+1)}-\zeta_{n+1})\cdots Y_V(v_{\alpha_{n+1}}^{(n+1)}, z_{\alpha_{n+1}}^{(n+1)}-\zeta_{n+1})\one)(\zeta_2, ...,  \zeta_{n+1}))\label{Arg-Series-4-Limit}
\end{align}
that does not depend on $\zeta_2, ..., \zeta_{n+1}$ in the region where
\begin{align*}
    & |z_1^{(0)}| > \cdots > |z_{l_0}^{(0)}| > |z_1^{(1)}| > \cdots > |z_{\alpha_1}^{(1)}| \\ & \qquad > |z_{l_0+1}^{(0)}| > \cdots > |z_{\alpha_0}^{(0)}|>|\zeta_i|+|z_{1}^{(i)} - \zeta_i|, i=2, ..., n+1;  \\
    & |z_1^{(i)}-\zeta_i|>\cdots > |z_{\alpha_i}^{(i)}-\zeta_i|, i = 2, ..., n+1; \\
    & |z_s^{(i)}-\zeta_i - z_t^{(j)}+\zeta_j|<|\zeta_i-\zeta_j|, 2\leq i < j \leq n+1, 1\leq s \leq \alpha_i, 1\leq t \leq \alpha_j.
\end{align*}
By commutativity of $Y_W^L$ and $Y_W^{s(R)}$, Theorem \ref{ConvRef} Part (1), and the expansion of $\overline{W}$-valued rational functions, the series 
\begin{align*} 
Y_W^L(v_1^{(0)}, z_1^{(0)})& \cdots Y_W^L(v_{l_0}^{(0)}, z_{l_0}^{(0)})
\cdot  Y_W^{s(R)}(v_{l_0+1}^{(0)}, z_{l_0+1}^{(0)}) \cdots Y_W^{s(R)}(v_{\alpha_0}^{(0)}, z_{\alpha_0}^{(0)})\cdot Y_W^L(v_1^{(1)}, z_1^{(1)}) \cdots Y_W^L(v_{\alpha_1}^{(1)}, z_{\alpha_1}^{(1)})\\
\cdot \Phi(& Y_V(v_1^{(2)}, z_1^{(2)}-\zeta_2)\cdots Y_V(v_{\alpha_2}^{(2)}, z_{\alpha_2}^{(1)}-\zeta_2)\one \\
\otimes & \cdots \\
\otimes & Y_V(v_1^{(n+1)}, z_1^{(n+1)}-\zeta_{n+1})\cdots Y_V(v_{\alpha_{n+1}}^{(n+1)}, z_{\alpha_{n+1}}^{(n+1)}-\zeta_{n+1})\one)(\zeta_2, ...,  \zeta_{n+1})
\end{align*}
converges absolutely to (\ref{Arg-Series-4-Limit}) in the region 
\begin{align*}
    & |z_1^{(0)}| > \cdots > |z_{l_0}^{(0)}| > |z_{l_0+1}^{(0)}| > \cdots > |z_{\alpha_0}^{(0)}|\\
    & \qquad > |z_1^{(1)}| > \cdots > |z_{\alpha_1}^{(1)}|  >|\zeta_i|+|z_{1}^{(i)} - \zeta_i|, i=2, ..., n+1;  \\
    & |z_1^{(i)}-\zeta_i|>\cdots > |z_{\alpha_i}^{(i)}-\zeta_i|, i = 2, ..., n+1; \\
    & |z_s^{(i)}-\zeta_i - z_t^{(j)}+\zeta_j|<|\zeta_i-\zeta_j|, 2\leq i < j \leq n+1, 1\leq s \leq \alpha_i, 1\leq t \leq \alpha_j.
\end{align*}
For convenience, we denote this region by $A$. 

By associativity of $Y_W^L$, Theorem \ref{ConvRef} Part (2), and the expansion of $\overline{W}$-valued rational functions, we have the following equality 
\begin{align*}
& Y_W^L(v_1^{(1)}, z_1^{(1)})\cdots  Y_W^L(v_{\alpha_1}^{(1)}, z_{\alpha_1}^{(1)})\overline{w} = Y_W^L(Y_V(v_1^{(1)}, z_1^{(1)}-\zeta_1)\cdots Y_V(v_{\alpha_1}^{(1)}, z_{\alpha_1}^{(1)}-\zeta_1)\one, \zeta_1)\overline{w},
\end{align*}
where $\overline{w}$ is the expansion of a $\overline{W}$-valued rational function as the following series
\begin{align*}
\overline{w} = \Phi(& Y_V(v_1^{(2)}, z_1^{(2)}-\zeta_2)\cdots Y_V(v_{\alpha_2}^{(2)}, z_{\alpha_2}^{(2)}-\zeta_2)\one \\
\otimes & \cdots \\
\otimes & Y_V(v_1^{(n+1)}, z_1^{(n+1)}-\zeta_{n+1})\cdots Y_V(v_{\alpha_{n+1}}^{(n+1)}, z_{\alpha_{n+1}}^{(n+1)}-\zeta_{n+1})\one)(\zeta_2, ...,  \zeta_{n+1})
\end{align*}
This equation holds in the region where
\begin{align*}
    & |z_1^{(1)}|>\cdots > |z_{\alpha_1}^{(1)}| > |\zeta_1| >  |\zeta_i| + |z_1^{(i)}-\zeta_i- z_1^{(1)} + \zeta_1|, i =2, ..., n+1; \\
    & |z_1^{(i)}-\zeta_i| > \cdots > |z_{\alpha_i}^{(i)}-\zeta_i|, i = 1, ..., n+1.  
\end{align*}
For convenience, we denote this region by $B$. It is easy to check that $A\cap B$ is nonempty, in which the series
\begin{align*} 
Y_W^L(v_1^{(0)}, z_1^{(0)})& \cdots Y_W^L(v_{l_0}^{(0)}, z_{l_0}^{(0)})
\cdot  Y_W^{s(R)}(v_{l_0+1}^{(0)}, z_{l_0+1}^{(0)}) \cdots Y_W^{s(R)}(v_{\alpha_0}^{(0)}, z_{\alpha_0}^{(0)}) \\
\cdot  Y_W^L&(Y_V(v_1^{(1)}, z_1^{(1)}-\zeta_1)\cdots Y_V(v_{\alpha_1}^{(1)}, z_{\alpha_1}^{(1)}-\zeta_1)\one, \zeta_1)\\
\cdot \Phi(& Y_V(v_1^{(2)}, z_1^{(2)}-\zeta_2)\cdots Y_V(v_{\alpha_2}^{(2)}, z_{\alpha_2}^{(1)}-\zeta_2)\one \\
\otimes & \cdots \\
\otimes & Y_V(v_1^{(n+1)}, z_1^{(n+1)}-\zeta_{n+1})\cdots Y_V(v_{\alpha_{n+1}}^{(n+1)}, z_{\alpha_{n+1}}^{(n+1)}-\zeta_{n+1})\one)(\zeta_2, ...,  \zeta_{n+1})
\end{align*}
converges absolutely to (\ref{Arg-Series-4-Limit}).

Now we use the definition of $E_W^{(1, 0)}\circ_2 \Phi$, i.e., 
\begin{align*}
& (E_W^{(1, 0)}\circ_2 \Phi)(Y_V(v_1^{(1)}, z_1^{(1)}-\zeta_1) \cdots Y_V(v_{\alpha_1}^{(1)}, z_{\alpha_1}^{(1)} - \zeta_1)\one \otimes \cdots \\ & \qquad\qquad \quad  \otimes  Y_V(v_1^{(n+1)}, z_1^{(n+1)}-\zeta_{n+1})\cdots Y_V(v_{\alpha_{n+1}}^{(n+1)}, z_{\alpha_{n+1}}^{(n+1)}-\zeta_{n+1})\one)(\zeta_1, ..., \zeta_{n+1}) \\
= & Y_W^L(Y_V(v_1^{(1)}, z_1^{(1)}-\zeta_1) \cdots Y_V(v_{\alpha_1}^{(1)}, z_{\alpha_1}^{(1)} - \zeta_1)\one, \zeta_1) \\
& \cdot \Phi(Y_V(v_1^{(2)}, z_1^{(2)}-\zeta_2) \cdots Y_V(v_{\alpha_2}^{(2)}, z_{\alpha_2}^{(2)} - \zeta_2)\one \otimes \cdots \otimes Y_V(v_1^{(n+1)}, z_1^{(n+1)}-\zeta_{n+1})\cdots Y_V(v_{\alpha_{n+1}}^{(n+1)}, z_{\alpha_{n+1}}^{(n+1)}-\zeta_{n+1})\one )\\
& \qquad \qquad (\zeta_2, ..., \zeta_{n+1})
\end{align*}
where the right hand side is the expansion of the $\overline{W}$-valued rational function in the following region:
\begin{align*}
& |\zeta_1| > |\zeta_i| + |z_1^{(i)}- \zeta_i-z_1^{(1)} + \zeta_1|, i = 2, ..., n+1;\\
& |z_1^{(i)}-\zeta_i|>\cdots > |z_{\alpha_i}^{(i)}-\zeta_i|, i = 1, ..., n+1; \\
& |z_s^{(i)}-\zeta_i - z_t^{(j)}+\zeta_j|<|\zeta_i-\zeta_j|, 1\leq i < j \leq n+1, 1\leq s \leq \alpha_i, 1\leq t \leq \alpha_j.
\end{align*}
For convenience, we denote this region by $C$. It is easy to check that the region $A \cap B\cap C$ is nonempty. So in this region, the series
\begin{align}
Y_W^L(v_1^{(0)} , z_1^{(0)})\cdots &Y_W^L(v_{l_0}^{(0)}, z_{l_0}^{(0)})Y_W^{s(R)}(v_{l_0+1}^{(0)}, z_{l_0+1}^{(0)})\cdots Y_W^{s(R)}(v_{\alpha_0}^{(0)}, z_{\alpha_0}^{(0)})\nn\\
\cdot (E_W^{(1, 0)}\circ_2 \Phi)&(Y_V(v_1^{(1)}, z_1^{(1)}-\zeta_1)\cdots Y_V(v_{\alpha_1}^{(1)}, z_{\alpha_1}^{(1)}-\zeta_1)\one\nn \\ 
 &\otimes \cdots \otimes Y_V(v_1^{(n+1)}, z_1^{(n+1)}-\zeta_{n+1})\cdots Y_V(v_{\alpha_{n+1}}^{(1)}, z_{\alpha_{n+1}}^{(n+1)}-\zeta_{n+1})\one )(\zeta_1, ..., \zeta_{n+1})\label{Arg-Series-4}
\end{align}
converges absolutely to the same (\ref{Arg-Series-4-Limit}). Now we analyze series carefully, and conclude that the series is obtained from expanding the $\overline{W}$-rational function in the following way:
\begin{enumerate}
\item for $1 \leq i < j \leq \alpha_0$, the negative powers of $z_j^{(0)} - z_k^{(0)}$ is expanded in positive powers of $z_k^{(0)}$;
\item for $1\leq j \leq \alpha_0, i = 1, ..., n+1, 1\leq t \leq \alpha_i$, the negative powers of $z_j^{(0)} - z_t^{(i)}$ is first expanded in positive powers of $z_t^{(i)} = \zeta_i + z_t^{(i)} - \zeta_i$, then further expand the positive powers as polynomials in $\zeta_i$ and $z_t^{(i)} - \zeta_i$;
\item for $ i = 1, ..., n+1, 1\leq j < k \leq \alpha_i$, the negative powers of $z_j^{(i)} - z_k^{(i)} = (z_j^{(i)} - \zeta_i) - (z_k^{(i)} - \zeta_i)$ is expanded as positive powers of $(z_k^{(i)}-\zeta_i)$. 
\item for $1\leq i < j \leq n+1, 1\leq s \leq \alpha_i, 1\leq t \leq \alpha_j$, the negative powers of $z_s^{(i)} - z_t^{(j)}=(\zeta_i -\zeta_j) + (z_s^{(i)} - \zeta_i - z_t^{(j)} + \zeta_j)$ is expanded as positive powers of $(z_s^{(i)} - \zeta_i - z_t^{(j)} + \zeta_j)$, then further expand as polynomials of $z_s^{(i)} - \zeta_i$ and $z_t^{(j)} - \zeta_j$. 
\end{enumerate} 
Thus the series (\ref{Arg-Series-4}) converges absolutely to (\ref{Arg-Series-4-Limit}) in the region 
\begin{align*}
    &|z_1^{(0)}|> \cdots > |z_{\alpha_0}^{(0)}| > |\zeta_1| + |z_1^{(1)} - \zeta_1|;\\
    &|z_1^{(i)}-\zeta_i| > \cdots > |z_{\alpha_i}^{(i)}- \zeta_i|, i = 1, ..., n+1;\\
    &|\zeta_i - \zeta_j|> |z_s^{(i)} - \zeta_i-z_t^{(j)}+\zeta_j|, 1\leq i < j \leq n+1, 1\leq s \leq \alpha_i, 1\leq t \leq \alpha_j.
\end{align*}
This proves that $E_W^{(1, 0)}\circ_2 \Phi$ is in $\hat{C}_{m-1}^{n+1}(V, W)$. 

\item The third conclusion is proved similarly to the second. We should not repeat the whole argument, but just summarize the key steps.
\begin{enumerate}
    \item We use the associativity of $Y_W^{s(R)}$ and Theorem \ref{ConvRef} Part (2) to see that
\begin{align*}
& Y_W^{s(R)}(v_{\alpha_{n+1}}^{(n+1)}, z_{\alpha_{n+1}}^{(n+1)})\cdots  Y_W^{s(R)}(v_1^{(n+1)}, z_1^{(n+1)})\overline{w} \\
& \qquad = Y_W^{s(R)}(Y_V(v_{1}^{(n+1)}, z_1^{(n+1)}-\zeta_{n+1})\cdots Y_V(v_{\alpha_{n+1}}^{(n+1)}, z_{\alpha_{n+1}}^{(n+1)}-\zeta_{n+1})\one, \zeta_{n+1})\overline{w}
\end{align*}
where
\begin{align*}
\overline{w} = E(\Phi(& Y_V(v_1^{(1)}, z_1^{(1)}-\zeta_1)\cdots Y_V(v_{\alpha_1}^{(1)}, z_{\alpha_1}^{(1)}-\zeta_1)\one \\
\otimes & \cdots \\
\otimes & Y_V(v_1^{(n)}, z_1^{(n)}-\zeta_{n})\cdots Y_V(v_{\alpha_{n}}^{(n)}, z_{\alpha_{n}}^{(n)}-\zeta_{n})\one)(\zeta_1, ...,  \zeta_{n}))\in  \overline{W}
\end{align*}
and $z_1^{(1)}, ..., z_{\alpha_{n+1}}^{(n+1)}\in \C$ such that
\begin{align*}
&|z_{\alpha_{n+1}}^{(n+1)}| > \cdots > |z_1^{(n+1)}|>|z_s^{(i)}|, i = 1, ..., n, s = 1, ..., \alpha_i;\\
&|\zeta_{n+1}| > |z_1^{(n+1)} - \zeta_{n+1}| + |z_s^{(i)}|, i = 1, ..., n, s = 1, ..., \alpha_i;\\
&|z_1^{(n+1)}-\zeta_{n+1}| > \cdots > |z_{\alpha_{n+1}}^{(n+1)}-\zeta_{n+1}|;\\
&z_s^{(i)} \neq z_t^{(j)}, 1\leq i \leq j \leq n+1, s = 1,..., \alpha_i, t = 1, ..., \alpha_j, s \neq t \text{ when }i = j 
\end{align*}
for every $\overline{w}\in \overline{W}$ and $z_1^{(n+1)}, ..., z_{\alpha_{n+1}}^{(n+1)}\in \C$ such that the left-hand-side converges absolutely. 
\item With the commutativity of $Y_W^L$ and $Y_W^{s(R)}$ operators, Theorem \ref{ConvRef} Part (1) and analytic continuation similar as those in previous proofs, we can prove that the series
\begin{align*}
& \begin{aligned}
Y_W^L(u_1^{(0)} , z_1^{(0)})\cdots &Y_W^L(u_{l_0}^{(0)}, z_{l_0}^{(0)})Y_W^{s(R)}(u_{l_0+1}^{(0)}, z_{l_0+1}^{(0)})\cdots Y_W^{s(R)}(u_{\alpha_0}^{(0)}, z_{\alpha_0}^{(0)})\\
\cdot (E_W^{(0,1)}\circ_2 \Phi)&(Y_V(v_1^{(1)}, z_1^{(1)}-\zeta_1)\cdots Y_V(v_{\alpha_1}^{(1)}, z_{\alpha_1}^{(1)}-\zeta_1)\one \\ 
 &\otimes \cdots \otimes Y_V(v_1^{(n+1)}, z_1^{(n+1)}-\zeta_{n+1})\cdots Y_V(v_{\alpha_{n+1}}^{(1)}, z_{\alpha_{n+1}}^{(n+1)}-\zeta_{n+1})\one )(\zeta_1, ..., \zeta_{n+1})
\end{aligned}\\
& \begin{aligned} =\sum_{\substack{k_1^{(0)}, ..., k_{\alpha_0}^{(0)}\\ ... k_1^{(n+1)}, ..., k_{\alpha_{n+1}}^{(n+1)} \in \Z } } (Y_W^L)_{k_1^{(0)}}(u_1) & \cdots (Y_W^L)_{k_{l_0}^{(0)}}(u_{l_0}) (Y_W^{s(R)})_{k_{l_0+1}}(u_{l_0+1}) \cdots (Y_W^{s(R)})_{k_{\alpha_0}}(u_{\alpha_0})\\ 
\cdot (E_W^{(0,1)}\circ_2 \Phi)(& (Y_V)_{k_{1}^{(1)}}(v_{1}^{(1)})\cdots (Y_V)_{k_{\alpha_1}^{(1)}}(v_{\alpha_1}^{(1)})\one \\ & 
\otimes \cdots \otimes (Y_V)_{k_{1}^{(n+1)}}(v_{1}^{(n+1)})\cdots (Y_V)_{k_{\alpha_{n+1}}^{(n+1)}}(v_{\alpha_{n+1}}^{(n+1)})\one)(\zeta_1, ..., \zeta_{n+1})\\
 \cdot \prod_{j=1}^{\alpha_0}  z_j^{(0)} \prod_{j=1}^n  ( & z_1^{(j)}-\zeta_j)^{-k_1^{(j)}-1} \cdots  (z_{\alpha_j}^{(j)}-\zeta_j)^{-k_{\alpha_j}^{(j)}-1} 
\end{aligned}
\end{align*}
is the expansion of the $\overline{W}$-rational function 
\begin{align*} 
E(Y_W^L(u_1^{(0)}, z_1^{(0)})& \cdots Y_W^L(u_{l_0}^{(0)}, z_{l_0}^{(0)})
\cdot Y_W^L(v_1^{(1)}, z_1^{(1)}) \cdots Y_W^{s(R)}(u_{l_0+1}^{(0)}, z_{l_0+1}^{(0)})\\
 & \cdot Y_W^{s(R)}(v_{1}^{(n+1)}, z_{1}^{(n+1)}) \cdots Y_W^{s(R)}(v_{\alpha_{n+1}}^{(n+1)}, z_{\alpha_{n+1}}^{(n+1)})\\
& \begin{aligned}\cdot \Phi(& Y_V(v_1^{(1)}, z_1^{(1)}-\zeta_1)\cdots Y_V(v_{\alpha_1}^{(1)}, z_{\alpha_1}^{(1)}-\zeta_1)\one \\
\otimes & \cdots \\
\otimes & Y_V(v_1^{(n)}, z_1^{(n)}-\zeta_{n})\cdots Y_V(v_{\alpha_{n}}^{(n)}, z_{\alpha_{n}}^{(n)}-\zeta_{n})\one)(\zeta_1, ...,  \zeta_{n}))
\end{aligned}
\end{align*}
in the region
\begin{align*}
& |z_1^{(0)}|> \cdots > |z_{\alpha_0}^{(0)}| > |\zeta_i|+|z_t^{(i)}-\zeta_i|, i = 1, ..., n+1, t = 1, ..., \alpha_i;\\
& |z_1^{(i)}-\zeta_i|>\cdots > |z_{\alpha_i}^{(i)}-\zeta_i|, i = 1, ..., n+1; \\
& |z_s^{(i)}-\zeta_i - z_t^{(j)}+\zeta_j|<|\zeta_i-\zeta_j|, 1\leq i < j \leq n+1, 1\leq s \leq \alpha_i, 1\leq t \leq \alpha_j.
\end{align*}
\end{enumerate}
\end{enumerate}
\end{proof}

\subsection{The coboundary operators and the cochain complex}

For $m, n \in \Z_+$, we define the coboundary operator as follows
$$\hat{\delta}_m^n: \hat{C}_m^n(V, W) \to \hat{C}_{m-1}^{n+1}(V, W)$$
by
$$ \hat{\delta}_m^n\Phi =  E_W^{(1, 0)}\circ_2 \Phi + \sum_{i=1}^{n} (-1)^i\Phi\circ_i E_V^{(2)} + (-1)^{n+1} E_W^{(0,1)}\circ_2 \Phi
$$
More explicitly, $\hat{\delta}^n_m \Phi$ is a map from $V^{\otimes (n+1)}$ to $\widetilde{W}_{z_1, ..., z_{n+1}}$ satisfying
\begin{align*}
&(({\hat{\delta}}_m^n\Phi)(v_1\otimes \cdots \otimes v_{n+1}))(z_1, ..., z_{n+1}) \\
= &  E(Y_W^L(v_1, z_1) (\Phi(v_2\otimes \cdots \otimes v_{n+1}))(z_2, ..., z_{n+1})) \\
- &  E((\Phi(Y_V(v_1, z_1-\zeta_1)Y_V(v_2, z_2-\zeta_1)\one\otimes v_3 \otimes \cdots \otimes v_{n+1}))(\zeta_1, z_3, ..., z_{n+1}))\\
+ &  E((\Phi(v_1\otimes Y_V(v_2, z_2-\zeta_2)Y_V(v_3, z_3-\zeta_2)\one \otimes v_4 \otimes \cdots \otimes v_{n+1}) (z_1, \zeta_2, z_4, ..., z_{n+1}))\\
-&  \cdots \cdots \cdots \cdots \cdots \cdots\\
+ & (-1)^i E\left(\begin{aligned} (\Phi & (v_1 \otimes\cdots \otimes v_{i-1} \otimes Y_V(v_i, z_i - \zeta_i)Y_V(v_{i+1}, z_{i+1}-\zeta_i)\one \otimes v_{i+2} \otimes \cdots \otimes v_{n+1}))\\
&(z_1, ..., z_{i-1}, \zeta_i, z_{i+2}, ..., z_{n+1})\end{aligned}\right)\\
+ & \cdots \cdots \cdots \cdots \cdots \cdots\\
+ & (-1)^n E(\Phi(v_1 \otimes \cdots \otimes v_{n-1} \otimes Y_V(v_n, z_n - \zeta_n)Y_V(z_{n+1}-\zeta_n)\one))(z_1, ..., z_{n-1}, \zeta_n)) \\
+ & (-1)^{n+1} E(Y_W^{s(R)}(u_{n+1}, z_{n+1})(\Phi(v_1\otimes \cdots v_n))(z_1, ..., z_n))
\end{align*}
One can also write 
\begin{align*}
&((\hat{\delta}_m^n(\Phi))(v_1\otimes \cdots \otimes v_{n+1}))(z_1, ..., z_{n+1}) \\
= &  E(Y_W^L(v_1, z_1) (\Phi(v_2\otimes \cdots \otimes v_{n+1}))(z_2, ..., z_{n+1})) \\
+ & \sum_{i=1}^n(-1)^i E\left(\begin{aligned} (\Phi & (v_1 \otimes\cdots \otimes v_{i-1} \otimes Y_V(v_i, z_i - \zeta_i)Y_V(v_{i+1}, z_{i+1}-\zeta_i)\one \otimes v_{i+2} \otimes \cdots \otimes v_{n+1}))\\
&(z_1, ..., z_{i-1}, \zeta_i, z_{i+2}, ..., z_{n+1})\end{aligned}\right)\\
+ & (-1)^{n+1} E(Y_W^{s(R)}(u_{n+1}, z_{n+1})(\Phi(v_1\otimes \cdots v_n))(z_1, ..., z_n))
\end{align*}
provided that $i=1$ and $i=n$ term in the sum is well-understood. 

When $n=0$, $\Phi$ is represented by a vector $w$. In this case,
$$((\hat{\delta}_m^0 \Phi)(v_1))(z_1) = E(Y_W^L(v_1, z_1)w) - E(Y_W^{s(R)}(v_1, z_1)w)$$

When $n=1$, we have
\begin{align*}& (\hat{\delta}^1_m(\Phi)(v_1\otimes v_2))(z_1, z_2) \\
&   = E(Y_W^L(v_1, z_1)(\Phi(v_2))(z_2)) -  E (\Phi(Y_V(u_1, z_1-\zeta)Y_V(u_2, z_2-\zeta)\one))(\zeta) + E(Y^{s(R)}_W(v_2, z_2)(\Phi(v_1))(z_1))
\end{align*}
\begin{rema}
It is crucial that in all the explicit summations above, we are not adding series, but adding the analytic extensions of the sums of these series, which are $\overline{W}$-valued rational functions, aka., $\overline{W}$-elements that depends on $z_1, ..., z_{n+1}$. Those series refuse to be added up directly because the region of convergence of the first series and that of the last series do not intersect. 
\end{rema}

% To prove that the sequence
% $$\hat{C}_m^n(V, W) \to \hat{C}_{m-1}^{n+1}(V, W) \to \cdots \to \hat{C}_0^{n+m}(V, W)$$
% forms a cochain complex, we need the following proposition 

\begin{thm}
For every $m\in \Z_+, n\in \N$, $\hat{\delta}_m^n (\hat{C}_m^n(V, W)) \subseteq \hat{C}_{m-1}^{n+1}(V, W)$.
\end{thm}

\begin{proof}
This follows from Proposition \ref{deltasummands}. 
\end{proof}

\begin{thm}\label{delta-square-zero}
For $m, n\in \Z_+$, $\hat{\delta}_{m-1}^{n+1}\circ \hat{\delta}_m^n = 0$
\end{thm}

\begin{proof}
Let $\Phi\in \hat{C}^n_m(V, W)$. We compute as follows:
\begin{align*}
& \hat{\delta}^{n+1}_{m-1} (\hat{\delta}_m^n \Phi)\\
= &  E_W^{(1, 0)}\circ_2 \hat{\delta}^n_m\Phi + \sum_{i=1}^{n+1} (-1)^i(\hat{\delta}^n_m\Phi)\circ_i E_V^{(2)} + E_W^{(0,1)}\circ_0 \hat{\delta}^n_m\Phi \\
= & E_W^{(1, 0)}\circ_2 (E_W^{(1, 0)} \circ_2 \Phi + \sum_{j=1}^{n} \Phi\circ_j E_V^{(2)} + E_W^{(0,1)} \circ_0 \Phi) \\
& + \sum_{i=1}^{n+1}(-1)^i\left((E_W^{(1, 0)} \circ_2 \Phi)\circ_i E_V^{(2)} + \sum_{j=1}^{n} (-1)^j (\Phi\circ_j E_V^{(2)})\circ_i E_V^{(2)} + (-1)^{n+1}(E_W^{(0, 1)} \circ_1 \Phi)\circ_i E_V^{(2)} \right)\\ 
& + (-1)^{n+2} \left( E_W^{(0, 1)} \circ_1 (E_W^{(1,0)} \circ_2 \Phi)+ \sum_{j=1}^n (-1)^j E_W^{(0,1)} \circ_1 (\Phi \circ_j E_V^{(2)}) + (-1)^{n+1} E_W^{(0,1)} \circ_1 (E_W^{(0,1)} \circ_1 \Phi)\right)
\end{align*}
We rearrange the terms and indexes to write $\hat{\delta}_{m-1}^{n+1}(\hat{\delta}_m^n \Phi)$ as
\begin{align*}
& E_W^{(1,0)}\circ_2 (E_W^{(1,0)}\circ_2 \Phi) + \sum_{i=1}^{n} (-1)^{i} E_W^{(1,0)} \circ_2 (\Phi\circ_i E_V^{(2)}) +\sum_{i=1}^{n+1}(-1)^i(E_W^{(1,0)} \circ_2 \Phi)\circ_i E_V^{(2)} & \textrm{(I)}\\
& + (-1)^{n+1} E_W^{(1,0)}\circ_2 (E_W^{(0,1)} \circ_1 \Phi) + (-1)^{n+2}E_W^{(0,1)} \circ_1 (E_W^{(1,0)} \circ_2 \Phi) & \textrm{(II)}\\  
& + \sum_{i=1}^{n+1} \sum_{j=1}^{n} (-1)^i(-1)^j (\Phi\circ_j E_V^{(2)})\circ_i E_V^{(2)}& \textrm{(III)}\\ 
& + \sum_{i=1}^{n+1}(-1)^{n+1+i}(E_W^{(0,1)} \circ_1 \Phi)\circ_i E_V^{(2)} +  \sum_{i=1}^n (-1)^{n+2+i} E_W^{(0,1)} \circ_1 (\Phi\circ_i E_V^{(2)})- E_W^{(0,1)} \circ_1 (E_W^{(0,1)} \circ_1 \Phi) & \textrm{(IV)}
\end{align*}
We argue that (I), (II), (III) and (IV) are all zero. 

For (I), we need the following lemma
\begin{lemma}
$\displaystyle{E_W^{(1,0)}\circ_2 (E_W^{(1, 0)} \circ_2 \Phi) = (E_W^{(1, 0)}\circ_2 \Phi)\circ_1 E_V^{2}}$
\end{lemma}
\begin{proof}
For any $v_1, ..., v_{n+2}\in V, (z_1, ..., z_{n+2})\in F_n \C$, we have 
\begin{align*}
& [E_W^{(1, 0)}\circ_2 (E_W^{(1, 0)} \circ_2 \Phi)(v_1\otimesdots v_{n+2})](z_1, ..., z_{n+2})\\
= & [E_W^{(1, 0)} (v_1; [(E_W^{(1, 0)}\circ_2 \Phi)(v_2\otimesdots v_{n+2})](z_2,...,z_{n+2}) )](z_1)\\
= & E(Y_W^L(v_1, z_1) [E_W^{(1, 0)}(v_2; [\Phi(v_3\otimesdots v_{n+2})](z_3,...,z_{n+2}))](z_2)\\
= & E(Y_W^L(v_1, z_1) Y_W^L(v_2, z_2) [\Phi(v_3\otimesdots v_{n+2})](z_3,..., z_{n+2}),
\end{align*}
and
\begin{align*}
& [(E_W^{(1, 0)}\circ_2 \Phi)\circ_1 E_V^{(2)}](v_1\otimesdots v_{n+2})(z_1, ..., z_{n+2})\\
= & [(E_W^{(1,0)}\circ_2 \Phi)(Y_V(v_1, z_1-\zeta)Y_V(v_2, z_2-\zeta)\one \otimes v_3\otimesdots v_{n+2})](\zeta, z_3,..., z_{n+2}) \\
= & [E_W^{(1, 0)}(Y_V(v_1, z_1-\zeta)Y_V(v_2, z_2-\zeta)\one ; [\Phi(v_3\otimesdots v_{n+2})](z_3, ..., z_{n+2}))](\zeta)\\
= & E(Y_W^L(Y_V(v_1, z_1-\zeta)Y_V(v_2, z_2-\zeta)\one, \zeta)[\Phi(v_3\otimesdots v_{n+2})](z_3, ..., z_{n+2})) %\\
% = & E(Y_W^L(v_1, z_1)Y_W^L(v_2, z_2)Y_W(\one, \zeta)[\Phi(v_3\otimesdots v_{n+2})](z_3, ..., z_{n+2})) \\
% = & E(Y_W^L(v_1, z_1) Y_W^L(v_2, z_2) [\Phi(v_3\otimesdots v_{n+2})](z_3,..., z_{n+2}),
\end{align*}
It follows from Theorem \ref{ConvRef} Part (2) and the identity property of vacuum that these rational functions are equal. 
\end{proof}

We also need the following lemma
\begin{lemma}
$E_W^{(1,0)} \circ_2 (\Phi\circ_i E_V^{(2)}) = (E_W^{(1,0)} \circ_2 \Phi )\circ_{i+1} E_V^{(2)}$
\end{lemma}

\begin{proof}
For any $v_1, ..., v_{n+2}\in V, (z_1, ..., z_{n+2})\in F_n \C$, we have 
\begin{align*}
& [E_W^{(1,0)} \circ_2 (\Phi\circ_i E_V^{(2)})(v_1\otimesdots v_{n+2})](z_1, ..., z_{n+2}) \\
= & [E_W^{(1,0)}(v_1; [(\Phi\circ_i E_V^{(2)})(v_2\otimesdots v_{n+2})](z_2, ..., z_{n+2}))](z_1)\\
= & [E_W^{(1,0)}(v_1; [\Phi(v_2\otimesdots [E_V^{(2)}(v_{i+1}, v_{i+2})](z_{i+1}-\zeta, z_{i+2}-\zeta)\otimesdots  v_{n+2})](z_2, ..., \zeta, ... , z_{n+2}))](z_1)\\
= & E(Y_W^L(v_1, z_1) [\Phi(v_2 \otimesdots Y_V(v_{i+1}, z_{i+1}-\zeta)Y_V(v_{i+2}, z_{i+2}-\zeta)\one \otimesdots v_{n+2})](z_2, ..., \zeta, ..., z_{n+2})
\end{align*}
and 
\begin{align*}
& [(E_W^{(1,0)} \circ_2 \Phi)\circ_{i+1} E_V^{(2)}(v_1\otimesdots v_{n+2})](z_1, ..., z_{n+2}) \\
= & [(E_W^{(1,0)}\circ_2 \Phi)(v_1\otimesdots [E_V^{(2)}(v_{i+1}, v_{i+2})](z_{i+1}-\zeta, z_{i+2}-\zeta)\otimesdots v_{n+2})](z_1, ..., \zeta, ..., z_{n+2})\\
= & [E_W^{(1,0)}(v_1; [\Phi(v_2\otimesdots [E_V^{(2)}(v_{i+1}, v_{i+2})](z_{i+1}-\zeta, z_{i+2}-\zeta)\otimesdots  v_{n+2})](z_2, ..., \zeta, ... , z_{n+2}))](z_1)\\
= & E(Y_W^L(v_1, z_1) [\Phi(v_2 \otimesdots Y_V(v_{i+1}, z_{i+1}-\zeta)Y_V(v_{i+2}, z_{i+2}-\zeta)\one \otimesdots v_{n+2})](z_2, ..., \zeta, ..., z_{n+2})
\end{align*}
So they are equal. 
\end{proof}
So the second sum and the third sum without $i=1$ differs by an index shift and a $(-1)$ factor. That way they cancels out.

For (II), we need the following lemma:
\begin{lemma}
$$E_W^{(1,0)} \circ_2 (E_W^{(0,1)}\circ_1 \Phi) = E_W^{(0,1)} \circ_1 (E_W^{(1,0)}\circ_2 \Phi)$$
\end{lemma}
\begin{proof}
For any $v_1, ..., v_{n+2}\in V, (z_1, ..., z_{n+2})\in F_n \C$, we have 
\begin{align*}
& [E_W^{(1,0)}\circ_2 (E_W^{(0,1)} \circ_1 \Phi)(v_1 \otimes \cdots \otimes v_{n+2})](z_1, ..., z_{n+2})\\
= & [E_W^{(1, 0)}(v_1; [(E_W^{(0,1)}\circ_1 \Phi)(v_2 \otimesdots v_{n+2})](z_2, ..., z_{n+2}) ](z_1) \\
= & [E_W^{(1, 0)}(v_1; [E_W^{(0,1)}([\Phi(v_2 \otimesdots v_{n+1})](z_2, ..., z_{n+1}); v_{n+2})(z_{n+2}) ](z_1) \\
= & E(Y_W^L(v_1, z_1)Y_W^{s(R)}(v_{n+2}, z_{n+2}) [\Phi(v_2\otimesdots v_{n+1})](z_2, ..., z_{n+1}), 
\end{align*}
and 
\begin{align*}
& [E_W^{(0,1)}\circ_1 (E_W^{(1,0)} \circ_2 \Phi)(v_1 \otimes \cdots \otimes v_{n+2})](z_1, ..., z_{n+2})\\
= & [E_W^{(0, 1)}([(E_W^{(1,0)}\circ_2 \Phi)(v_1 \otimesdots v_{n+1})](z_1, ..., z_{n+1}); v_{n+2} )](z_{n+2}) \\
= & [E_W^{(0, 1)}([E_W^{(1,0)}(v_1; [\Phi(v_2 \otimesdots v_{n+1})](z_2, ..., z_{n+1}))(z_1); v_{n+2})(z_{n+2}) \\
= & E(Y_W^{s(R)}(v_{n+2}, z_{n+2})Y_W^L(v_1, z_1) [\Phi(v_2\otimesdots v_{n+1})](z_2, ..., z_{n+1}), 
\end{align*}
It follows from Theorem \ref{ConvRef} Part (1) that these rational functions are equal. 
%They are equal because of the \textcolor{red}{commutativity} of $Y_W^L$ and $Y_W^{s(R)}$. 
\end{proof}
So the two terms in (II) add up to zero. 

For (III), We need the following lemmas
\begin{lemma}
If $j\leq i-1$, then 
$$(\Phi \circ_j E_V^{(2)})\circ_i E_V^{(2)}= (\Phi\circ_{i-1} E_V^{(2)})\circ_j E_V^{(2)}. $$
\end{lemma}

\begin{proof}
Consider the case when $j<i-1$. Then for any $v_1, ..., v_{n+2}\in V, (z_1, ..., z_{n+2})\in F_n \C$, we have 
\begin{align*}
& [(\Phi \circ_j E_V^{(2)})\circ_i E_V^{(2)}(v_1 \otimes \cdots \otimes v_{n+2})](z_1, ..., z_{n+2})\\
= & [\Phi \circ_j E_V^{(2)}(v_1 \otimesdots [E_V^{(2)}(v_i, v_{i+1})](z_i - \zeta, z_{i+1}-\zeta) \otimesdots v_{n+2})](z_1, ..., \zeta, ..., z_{n+2})\\
= & [\Phi(v_1 \otimesdots [E_V^{(2)}(v_j, v_{j+1})](z_{j}-\eta, z_{j+1} -\eta)\\
& \qquad \otimesdots [E_V^{(2)}(v_i, v_{i+1})](z_i - \zeta, z_{i+1}-\zeta) \otimesdots v_{n+2})](z_1, ..., \eta, ...,  \zeta, ..., z_{n+2})\\
= & E([\Phi(v_1 \otimesdots Y_V(v_j, z_j - \eta)Y_V(v_{j+1}, z_{j+1} -\eta)\one\\
& \qquad \otimesdots Y_V(v_i, z_i - \zeta)Y_V(v_{i+1}, z_{i+1} -\zeta)\one \otimesdots v_{n+2})](z_1, ..., \eta, ...,  \zeta, ..., z_{n+2}))
\end{align*}
and
\begin{align*}
& [(\Phi\circ_{i-1} E_V^{(2)})\circ_j E_V^{(2)}(v_1 \otimes \cdots \otimes v_{n+2})](z_1, ..., z_{n+2})\\
= & [\Phi \circ_{i-1} E_V^{(2)}(v_1 \otimesdots [E_V^{(2)}(v_j, v_{j+1})](z_j - \zeta, z_{j+1}-\zeta) \otimesdots v_{n+2})](z_1, ..., \zeta, ..., z_{n+2})\\
= & [\Phi(v_1 \otimesdots [E_V^{(2)}(v_j, v_{j+1})](z_{j}-\zeta, z_{j+1} -\zeta)\\
& \qquad \otimesdots [E_V^{(2)}(v_i, v_{i+1})](z_i - \eta, z_{i+1}-\eta) \otimesdots v_{n+2})](z_1, ..., \zeta, ...,  \eta, ..., z_{n+2})\\
= & E([\Phi(v_1 \otimesdots Y_V(v_j, z_j - \zeta)Y_V(v_{j+1}, z_{j+1} -\zeta)\one\\
& \qquad \otimesdots Y_V(v_i, z_i - \eta)Y_V(v_{i+1}, z_{i+1} -\eta)\one \otimesdots v_{n+2})](z_1, ..., \zeta, ...,  \eta, ..., z_{n+2}))
\end{align*}
Since the resulting $\overline{W}$-valued rational functions are independent of the choice of $\zeta$ and $\eta$, they are equal. 
%They are equal because the resulting $\overline{W}$-element is independent of $\zeta$ and $\eta$. 

Now consider the case when $j=i-1$. Then for any $v_1, ..., v_{n+2}\in V, (z_1, ..., z_{n+2})\in F_n \C$, we compute the left-hand-side as follows:
\begin{align*}
& [(\Phi \circ_{i-1} E_V^{(2)})\circ_i E_V^{(2)}(v_1 \otimes \cdots \otimes v_{n+2})](z_1, ..., z_{n+2})\\
= & [\Phi \circ_{i-1} E_V^{(2)}(v_1 \otimesdots [E_V^{(2)}(v_i, v_{i+1})](z_i - \zeta, z_{i+1}-\zeta) \otimesdots v_{n+2})](z_1, ..., \zeta, ..., z_{n+2})\\
= & [\Phi(v_1 \otimesdots [E_V^{(2)}(v_{i-1}, [E_V^{(2)}(v_i, v_{i+1})](z_i - \zeta, z_{i+1}-\zeta))](z_{i-1}-\eta, \zeta -\eta)\\
& \qquad\,   \otimesdots v_{n+2})](z_1, ..., z_{i-2}, \eta, z_{i+2}..., z_{n+2})\\
= & E([\Phi(v_1 \otimesdots Y_V(v_{i-1}, z_{i-1} - \eta)Y_V(Y_V(v_i, z_i-\zeta)Y_V(v_{i+1}, z_{i+1}-\zeta)\one, \zeta-\eta)\one\\
& \qquad\quad\, \otimes v_{i+2} \otimesdots v_{n+2})](z_1, ..., z_{i-2}, \eta, z_{i+2}..., z_{n+2})),\\
= & E([\Phi(v_1 \otimesdots Y_V(v_{i-1}, z_{i-1} - \eta)Y_V(v_i, z_i-\eta)Y_V(v_{i+1}, z_{i+1}-\eta)Y_V(\one, \zeta-\eta)\one\\
& \qquad\quad\, \otimes v_{i+2} \otimesdots v_{n+2})](z_1, ..., z_{i-2}, \eta, z_{i+2}..., z_{n+2})),\\
= & E([\Phi(v_1 \otimesdots Y_V(v_{i-1}, z_{i-1} - \eta)Y_V(v_i, z_i-\eta)Y_V(v_{i+1}, z_{i+1}-\eta)\one \otimesdots v_{n+2})]\\
& \quad (z_1, ..., z_{i-2}, \eta, z_{i+2}..., z_{n+2}))
\end{align*}
%\textcolor{red}{Need to argue the independence of $\zeta$ and $\eta$ well. }
where the fourth equality follows from the associativity in $V$, the fifth equality follows from the identity property of the vacuum. Also by Definition \ref{Composable}, the resulting rational function is independent of $\eta$. 

Now we compute the right-hand-side as follows:
\begin{align*}
& [(\Phi\circ_{i-1} E_V^{(2)})\circ_{i-1} E_V^{(2)}(v_1 \otimes \cdots \otimes v_{n+2})](z_1, ..., z_{n+2})\\
= & [\Phi \circ_{i-1} E_V^{(2)}(v_1 \otimesdots [E_V^{(2)}(v_{i-1}, v_{i})](z_{i-1} - \zeta, z_{i}-\zeta) \otimesdots v_{n+2})](z_1, ..., z_{i-2}, \zeta, z_{i+1}, ..., z_{n+2})\\
= & [\Phi(v_1 \otimesdots [E_V^{(2)}([E_V^{(2)}(v_{i-1}, v_{i})](z_{i-1}-\zeta, z_{i} -\zeta), v_{i+1})](\zeta-\eta, z_{i+1}-\eta)\\
& \qquad\, \otimesdots v_{n+2})](z_1, ..., z_{i-2}, \eta, z_{i+2}, ..., z_{n+2})\\
= & E([\Phi(v_1 \otimesdots Y_V(Y_V(v_{i-1}, z_{i-1} - \zeta)Y_V(v_i, z_i -\zeta)\one, \zeta-\eta)Y_V(v_{i+1}, z_{i+1} -\eta)\one\\
& \qquad\quad\, \otimesdots v_{n+2})](z_1, ..., z_{i-2}, \eta, z_{i+2}, ..., z_{n+2}))\\
= & E([\Phi(v_1 \otimesdots Y_V(v_{i-1}, z_{i-1} - \eta)Y_V(v_i, z_i -\eta)Y_V(\one, \zeta-\eta)Y_V(v_{i+1}, z_{i+1} -\eta)\one\\
& \qquad\quad\, \otimesdots v_{n+2})](z_1, ..., z_{i-2}, \eta, z_{i+2}, ..., z_{n+2}))\\
= & E([\Phi(v_1 \otimesdots Y_V(v_{i-1}, z_{i-1} - \eta)Y_V(v_i, z_i -\eta)Y_V(v_{i+1}, z_{i+1} -\eta)\one\otimesdots v_{n+2})]\\
& \quad (z_1, ..., z_{i-2}, \eta, z_{i+2}, ..., z_{n+2}))
\end{align*}
where the fourth equality follows from the associativity extended to $\overline{V}$-valued rational functions (see Theorem \ref{ConvRef} Part (2)), the fifth equality follows from the identity property of the vacuum. Also by Definition \ref{Composable}, the resulting rational function is independent of $\eta$. So the left-hand-side and the right-hand-side are equal. 

\end{proof}

\begin{lemma}
If $j\geq i$, then
$$(\Phi\circ_j E_V^{(2)})\circ_i E_V^{(2)} = (\Phi \circ_i E_V^{(2)})\circ_{j+1} E_V^{(2)}.$$
\end{lemma}

\begin{proof}
Consider the case when $j> i$. Then for any $v_1, ..., v_{n+2}\in V, (z_1, ..., z_{n+2})\in F_n \C$, we have 
\begin{align*}
& [(\Phi \circ_j E_V^{(2)})\circ_i E_V^{(2)}(v_1 \otimes \cdots \otimes v_{n+2})](z_1, ..., z_{n+2})\\
= & [\Phi \circ_j E_V^{(2)}(v_1 \otimesdots [E_V^{(2)}(v_i, v_{i+1})](z_i - \zeta, z_{i+1}-\zeta) \otimesdots v_{n+2})](z_1, ..., \zeta, ..., z_{n+2})\\
= & [\Phi(v_1 \otimesdots [E_V^{(2)}(v_i, v_{i+1})](z_i - \zeta, z_{i+1}-\zeta) \\
& \qquad \otimesdots [E_V^{(2)}(v_j, v_{j+1})](z_{j}-\eta, z_{j+1} -\eta) \otimesdots v_{n+2})](z_1, ..., \zeta, ...,  \eta, ..., z_{n+2})\\
= & E([\Phi(v_1 \otimesdots Y_V(v_i, z_i - \zeta)Y_V(v_{i+1}, z_{i+1} -\zeta)\one\\
& \qquad \otimesdots Y_V(v_j, z_j - \eta)Y_V(v_{j+1}, z_{j+1} -\eta)\one \otimesdots v_{n+2})](z_1, ..., \zeta, ...,  \eta, ..., z_{n+2}))
\end{align*}
and
\begin{align*}
& [(\Phi\circ_{i-1} E_V^{(2)})\circ_j E_V^{(2)}(v_1 \otimes \cdots \otimes v_{n+2})](z_1, ..., z_{n+2})\\
= & [\Phi \circ_{i-1} E_V^{(2)}(v_1 \otimesdots [E_V^{(2)}(v_j, v_{j+1})](z_j - \zeta, z_{j+1}-\zeta) \otimesdots v_{n+2})](z_1, ..., \zeta, ..., z_{n+2})\\
= & [\Phi(v_1 \otimesdots [E_V^{(2)}(v_i, v_{i+1})](z_i - \zeta, z_{i+1}-\zeta) \\
& \qquad \otimesdots [E_V^{(2)}(v_j, v_{j+1})](z_{j}-\eta, z_{j+1} -\eta) \otimesdots v_{n+2})](z_1, ..., \zeta, ...,  \eta, ..., z_{n+2})\\
= & E([\Phi(v_1 \otimesdots Y_V(v_i, z_i - \zeta)Y_V(v_{i+1}, z_{i+1} -\zeta)\one\\
& \qquad \otimesdots Y_V(v_j, z_j - \eta)Y_V(v_{j+1}, z_{j+1} -\eta)\one \otimesdots v_{n+2})](z_1, ..., \zeta, ...,  \eta, ..., z_{n+2}))
\end{align*} 
They are equal because the resulting $\overline{W}$-valued rational functions are independent of $\zeta$ and $\eta$. 

Now consider the case when $j=i$. Then for any $v_1, ..., v_{n+2}\in V, (z_1, ..., z_{n+2})\in F_n \C$, we compute the left-hand-side as follows: 
\begin{align*}
& [(\Phi\circ_{i} E_V^{(2)})\circ_{i} E_V^{(2)}(v_1 \otimes \cdots \otimes v_{n+2})](z_1, ..., z_{n+2})\\
= & [\Phi \circ_{i} E_V^{(2)}(v_1 \otimesdots [E_V^{(2)}(v_{i}, v_{i+1})](z_{i} - \zeta, z_{i+1}-\zeta) \otimesdots v_{n+2})](z_1, ..., z_{i-1}, \zeta, z_{i+2}, ..., z_{n+2})\\
= & [\Phi(v_1 \otimesdots [E_V^{(2)}([E_V^{(2)}(v_{i}, v_{i+1})](z_{i}-\zeta, z_{i+1} -\zeta), v_{i+2})](\zeta-\eta, z_{i+2}-\eta)\\
& \qquad\, \otimesdots v_{n+2})](z_1, ..., z_{i-1}, \eta, z_{i+3}, ..., z_{n+2})\\
= & E([\Phi(v_1 \otimesdots Y_V(Y_V(v_{i}, z_{i} - \zeta)Y_V(v_{i+1}, z_{i+1} -\zeta)\one, \zeta-\eta)Y_V(v_{i+2}, z_{i+2} -\eta)\one\\
& \qquad\quad\, \otimesdots v_{n+2})](z_1, ..., z_{i-1}, \eta, z_{i+3}, ..., z_{n+2}))\\
= & E([\Phi(v_1 \otimesdots Y_V(v_{i}, z_{i} - \eta)Y_V(v_{i+1}, z_{i+1} -\eta)Y_V(\one, \zeta-\eta)Y_V(v_{i+2}, z_{i+2} -\eta)\one\\
& \qquad\quad\, \otimesdots v_{n+2})](z_1, ..., z_{i-1}, \eta, z_{i+3}, ..., z_{n+2}))\\
= & E([\Phi(v_1 \otimesdots Y_V(v_{i}, z_{i} - \eta)Y_V(v_{i+1}, z_{i+1} -\eta)Y_V(v_{i+2}, z_{i+2} -\eta)\one\otimesdots v_{n+2})]\\
& \quad (z_1, ..., z_{i-1}, \eta, z_{i+3}, ..., z_{n+2}))
\end{align*}
%\textcolor{red}{Need explanations} 
where the fourth equality follows from the associativity extended to $\overline{V}$-valued rational functions (see Theorem \ref{ConvRef} Part (2)), the fifth equality follows from the identity property of the vacuum. Also by Definition \ref{Composable}, the resulting rational function is independent of $\eta$. 

Now we compute the right-hand-side as follows
\begin{align*}
& [(\Phi \circ_{i} E_V^{(2)})\circ_{i+1} E_V^{(2)}(v_1 \otimes \cdots \otimes v_{n+2})](z_1, ..., z_{n+2})\\
= & [\Phi \circ_{i-1} E_V^{(2)}(v_1 \otimesdots [E_V^{(2)}(v_{i+1}, v_{i+2})](z_{i+1} - \zeta, z_{i+2}-\zeta) \otimesdots v_{n+2})](z_1, ..., \zeta, ..., z_{n+2})\\
= & [\Phi(v_1 \otimesdots [E_V^{(2)}(v_{i}, [E_V^{(2)}(v_{i+1}, v_{i+2})](z_{i+1} - \zeta, z_{i+2}-\zeta))](z_{i}-\eta, \zeta -\eta)\\
& \qquad\,   \otimesdots v_{n+2})](z_1, ..., z_{i-1}, \eta, z_{i+3}..., z_{n+2})\\
= & E([\Phi(v_1 \otimesdots Y_V(v_{i}, z_{i} - \eta)Y_V(Y_V(v_{i+1}, z_{i+1}-\zeta)Y_V(z_{i+2}-\zeta)\one, \zeta-\eta)\one\\
& \qquad\quad\, \otimes v_{i+3} \otimesdots v_{n+2})](z_1, ..., z_{i-1}, \eta, z_{i+3}..., z_{n+2})),\\
= & E([\Phi(v_1 \otimesdots Y_V(v_{i}, z_{i} - \eta)Y_V(v_{i+1}, z_{i+1}-\eta)Y_V(z_{i+2}-\eta)Y_V(\one, \zeta-\eta)\one\\
& \qquad\quad\, \otimes v_{i+3} \otimesdots v_{n+2})](z_1, ..., z_{i-1}, \eta, z_{i+3}..., z_{n+2})),\\
= & E([\Phi(v_1 \otimesdots Y_V(v_{i}, z_{i} - \eta)Y_V(v_{i+1}, z_{i+1}-\eta)Y_V(z_{i+2}-\eta)\one\otimes v_{i+3} \otimesdots v_{n+2})]\\
& \quad (z_1, ..., z_{i-1}, \eta, z_{i+3}..., z_{n+2})).
\end{align*}
%\textcolor{red}{Need to argue the independence of $\zeta$ and $\eta$ well. }
where the fourth equality follows from the associativity in $V$, the fifth equality follows from the identity property of the vacuum. The resulting $\overline{W}$-valued rational function is independent of $\eta$. So the left-hand-side and the right-hand-side are equal. 
\end{proof}

Once we proved these two lemmas, we write (III) as $$\sum_{i=2}^{n+1}\sum_{j=1}^{i-1} (-1)^{i+j} (\Phi\circ_j E_V^{(2)})\circ_i E_V^{(2)} + \sum_{i=1}^n \sum_{j=i}^{n} (-1)^{i+j} (\Phi\circ_j E_V^{(2)})\circ_i E_V^{(2)}$$
Here the first sum starts from $i=2$ because when $i=1$, the inner sum does not exist. Similarly the second sum ends at $i=n$ because when $i=n+1$, the inner sum does not exist. The first sum is computed as follows
\begin{align*}
& \sum_{i=2}^{n+1} \sum_{j=1}^{i-1} (-1)^{i+j} (\Phi\circ_j E_V^{(2)})\circ_i E_V^{(2)} \\
= & \sum_{i=2}^{n+1} \sum_{j=1}^{i-1} (-1)^{i+j} (\Phi\circ_{i-1}  E_V^{(2)})\circ_j E_V^{(2)} & \text{use the identity above}\\
= &  \sum_{j=1}^n\sum_{i=j+1}^{n+1} (-1)^{i+j} (\Phi\circ_{i-1} E_V^{(2)})\circ_j E_V^{(2)} & \text{change the order of summation}\\
= &  \sum_{i=1}^n\sum_{j=i+1}^{n+1}(-1)^{i+j} (\Phi\circ_{j-1}  E_V^{(2)})\circ_i E_V^{(2)} & \text{interchange $i$ and $j$}\\
= &  \sum_{i=1}^n\sum_{j=i}^{n}(-1)^{i+j+1} (\Phi\circ_{j}  E_V^{(2)})\circ_i E_V^{(2)} & \text{shift the index $j$}
\end{align*}
So the first sum is precisely the negative of the second sum. Thus the two sums add up to be zero.

For (IV), we need the following lemma 
\begin{lemma}
$$(E_W^{(0, 1)}\circ_1 \Phi)\circ_{n+1} E_V^{(2)} = E_W^{(0, 1)} \circ_1 (E_W^{(0, 1)} \circ_1 \Phi)$$
\end{lemma}
\begin{proof}
For any $v_1, ..., v_{n+2}\in V, (z_1, ..., z_{n+2})\in F_n \C$, we compute the left-hand-side as follows: 
\begin{align*}
&[(E_W^{(0, 1)}\circ_1 \Phi)\circ_{n+1} E_V^{(2)}(v_1\otimesdots v_{n+2})](z_1, ..., z_{n+2}) \\ 
= &  [E_W^{(0,1)}\circ_1 \Phi(v_1\otimesdots v_{n} \otimes [E_V^{(2)}(v_{n+1}, v_{n+2})](z_{n+1}-\zeta, z_{n+2}-\zeta))](z_1, ..., z_n, \zeta)\\
= & [E_W^{(0, 1)}([\Phi(v_1\otimesdots v_n)](z_1, ..., z_n); [E_V^{(2)}(v_{n+1}, v_{n+2})](z_{n+1}-\zeta, z_{n+2}-\zeta))](\zeta)\\
= & E(Y_W^{s(R)}(Y_V(v_{n+1}, z_{n+1}-\zeta)Y_V(v_{n+2}, z_{n+2}-\zeta)\one, \zeta)[\Phi(v_1, ..., v_n)](z_1,...,z_n))\\
%\textcolor{red}{=} & E(Y_W^{s(R)}(Y_V^s(v_{n+2}, z_{n+2}-\zeta)Y_V^s(v_{n+1}, z_{n+1}-\zeta)\one, \zeta)[\Phi(v_1, ..., v_n)](z_1,...,z_n))\\
= & E(Y_W^{s(R)}(v_{n+2}, z_{n+2})Y_W^{s(R)}(v_{n+1}, z_{n+1})Y_W^{s(R)}(\one, \zeta)[\Phi(v_1, ..., v_n)](z_1,...,z_n))\\
= & E(Y_W^{s(R)}(v_{n+2}, z_{n+2})Y_W^{s(R)}(v_{n+1}, z_{n+1})[\Phi(v_1, ..., v_n)](z_1,...,z_n)),
\end{align*}
where the fourth equality follows from the associativity of $Y_W^{s(R)}$ extended to $\overline{W}$-valued rational functions (see Theorem \ref{ConvRef} Part (3)). The fifth equality follows from the identity property of vacuum. 

Now we compute the right-hand-side as follows: 
\begin{align*}
&[E_W^{(0, 1)}\circ_1 (E_W^{(0,1)}\circ_1 \Phi)(v_1\otimesdots v_{n+2})](z_1, ..., z_{n+2}) \\ 
= &  [E_W^{(0,1)}([E_W^{(0,1)}\circ_1 \Phi(v_1\otimesdots v_{n+1})](z_1, ..., z_{n+1});, v_{n+2})](z_{n+2})\\
= &  [E_W^{(0,1)}([E_W^{(0,1)}([\Phi(v_1\otimesdots v_{n+1})](z_1, ..., z_n);, v_{n+1})](z_{n+1});, v_{n+2})](z_{n+2})\\
= &  E(Y_W^{s(R)} (v_{n+2}, z_{n+2}) Y_W^{s(R)}(v_{n+1}, z_{n+1})[\Phi(v_1\otimesdots v_n)](z_1,...z_n))
\end{align*}
So it is equal to the left-hand-side. 
\end{proof}
So the $(n+1)$-th term in the first sum cancels out with the third term. 

We also need the following lemma
\begin{lemma}
$(E_W^{(0,1)}\circ_1 \Phi) \circ_i E_V^{(2)} = E_W^{(0,1)}\circ_1 (\Phi\circ_{i} E_V^{(2)})$
\end{lemma}
\begin{proof}
For $v_1, ..., v_{n+2} \in V, (z_1, ..., z_{n+2})\in F_n \C$, 
\begin{align*}
& [(E_W^{(0,1)}\circ_1 \Phi)\circ_i E_V^{(2)}(v_1 \otimes \cdots \otimes v_{n+2})](z_1, ..., z_{n+2}) \\
= & [E_W^{(0,1)}\circ_1 \Phi(v_1 \otimesdots [E_V^{(2)}(v_i \otimes v_{i+1})](z_i-\zeta, z_{i+1}-\zeta)\otimesdots v_{n+2}](z_1, ..., \zeta, ..., z_{n+2})\\
= & [E_W^{(0, 1)}]([\Phi(v_1 \otimesdots [E_V^{(2)}(v_i \otimes v_{i+1})](z_i-\zeta, z_{i+1}-\zeta)\otimesdots v_{n+1})](z_1, ..., \zeta, ..., z_{n+1}); v_{n+2})(z_{n+2})\\
= & E(Y_W^{s(R)}(v_{n+2}, z_{n+2}) [\Phi(v_1 \otimesdots Y_V(v_i, z_i -\zeta) Y_V(v_{i+1}, z_{i+1}- \zeta) \one \otimesdots v_{n+1})](z_1, ..., \zeta, ..., z_{n+1}))
\end{align*}
and 
\begin{align*}
& [E_W^{(0, 1)}\circ_1 (\Phi \circ_i E_V^{(2)})(v_1 \otimes \cdots \otimes v_{n+2})](z_1, ..., z_{n+2}) \\
= & [E_W^{(0,1)}([\Phi\circ_i E_V^{(2)} (v_1\otimesdots v_{n+1})](z_1,..., z_{n+1}); v_{n+2})](z_{n+2})\\
= & [E_W^{(0, 1)}]([\Phi(v_1 \otimesdots [E_V^{(2)}(v_i \otimes v_{i+1})](z_i-\zeta, z_{i+1}-\zeta)\otimesdots v_{n+1})](z_1, ..., \zeta, ..., z_{n+1}); v_{n+2})(z_{n+2})\\
= & E(Y_W^{s(R)}(v_{n+2}, z_{n+2}) [\Phi(v_1 \otimesdots Y_V(v_i, z_i -\zeta) Y_V(v_{i+1}, z_{i+1}- \zeta) \one \otimesdots v_{n+1})](z_1, ..., \zeta, ..., z_{n+1}))
\end{align*}
So they are equal. 
\end{proof}
Therefore, the rest of the first sum cancels out with the second sum.

\end{proof}

% \textcolor{red}{What remains to be said: \\
% 1. The image of $\hat{\delta}$ of $\hat{C}^n_m(V, W)$ sits in the space $\hat{C}^{n+1}_{m-1}(V, W)$. \\
% 2. The 1/2-complex should also be defined here, especially if we want to use it to develop the deformation theory. }

\begin{rema}
We remind the readers again that all the equalities in the lemmas above are in the space of $\overline{W}$-valued rational functions. The only requirements on the parameters $z_1, ..., z_{n+1}$ is that they are mutually distinct to each other. 
\end{rema}

We have given the definitions of $\hat{C}_m^n(V, W)$ and $\hat{\delta}_m^n$ for all integers $m\geq 1, n\geq 1$. Here we discuss the case $n=0$. 

% \begin{defn}
% We define $\hat{C}^0(V, W)$ to be the set of vectors $w\in W$, such that for every $u\in V$, the series 
% $$Y_W^L(u, x)w - Y_W^{s(R)}(u, x)w$$
% in $W[[x, x^{-1}]]$ has no negative powers, and map
% $$u \mapsto E(Y_W^L(u, z)w-Y_W^{s(R)}(u, z)w)$$
% is in $\ker \hat{\delta}^1$. 
% \end{defn}
\begin{defn}
We define $\hat{C}^0(V, W)$ to be the set of \textit{vaccum-like} vectors $w\in W$, i.e., $w\in W_{(0)}$ and $D_W w=0$
\end{defn}

\begin{prop}
Let $w\in W$ be a vaccum-like vector. Then for every $v\in V, Y_W^L(v, x)w \in W[[x]], Y_W^{s(R)}(v, x)w\in W[[x]]$. 
\end{prop}

\begin{proof}
Fix $v\in V$. From the $D$-commutator formula, we have
$$\frac d {dx} Y_W^L(v, x)w = D_WY_W^L(v, x).$$
Thus for the series $e^{-xD_W}Y_W^L(v, x)w$, we have
$$\frac d {dx}\left(e^{-xD_W}Y_W^L(v, x)w \right)= -e^{-xD_W}D_WY_W^L(v, x)w + e^{-xD_W} D_WY_W^L(v, x) =0, $$
which then shows $e^{-xD_W}Y_W^L(v, x)w$ has only constant term. Thus $Y_W^L(v, x)w \in W[[x]]$. One similarly proves the conclusion for $Y_W^{s(R)}$. 
\end{proof}

% \begin{exam}
% If $w\in W$ is ``vacuum-like'', i.e., 
% $$w \in W_{[0]}, D_W w = 0, $$
% then $w\in \hat{C}^0(V, W)$. 
% \end{exam}

% \begin{exam}
% Let $V$ be a vertex algebra and $W$ be a $V$-module. If we view $V$ as a MOSVA and $W$ as a $V$-bimodule, then 
% $\hat{C}^0(V, W) = W$. 
% \end{exam}

\begin{defn}
We define $\hat{\delta}^0: \hat{C}^0(V, W)\to \hom(V, \widetilde{W}_z)$ by the following: for $w\in \hat{C}^0(V, W)$
$$[(\hat{\delta}^0(w))(v)](z) = E(Y_W^L(v, z)w-Y_W^{s(R)}(v, z)w)$$
%By definition of $\hat{C}^0(V, W)$, $\hat{\delta}^1 \circ \hat{\delta}^0 = 0$. 
\end{defn}

\begin{prop}
For every $w\in \hat{C}^0(V, W)$, $\hat{\delta}^0(w) \in \hat{C}_\infty^1(V, W)$, and $\hat{\delta}^1(\hat{\delta}^0(w))=0$
\end{prop}

\begin{proof}
It is easy to check that $\hat{\delta}^0(w)$ satisfies the $\d$-conjugation property and $D$-derivative property. From the arguments in Example \ref{exam-n-cochain}, $\hat{\delta}^0(w)$ is composable with any numbers of vertex operators. The last conclusion follows from a computation that is essentially the same as those in Theorem \ref{delta-square-zero}. 
\end{proof}

\begin{rema}
It can be proved that the map $v \mapsto [(\hat{\delta}^0(w)](v))(0)$ is an ``inner-derivation''. The set of inner-derivations from $V$ to $W$ is isomorphic as vector space to $\hat{\delta}^0(\hat{C}^0(V, W))$. Please see details in \cite{HQ-Red} and \cite{Q-Thesis}. 
\end{rema}

Thus we proved the following theorem:

\begin{thm}
For any $m\in \Z_+$, the following sequence
$$\hat{C}^0(V, W) \xrightarrow{\hat{\delta}^0} \hat{C}_m^1(V, W) \xrightarrow{\hat{\delta}_m^1} \hat{C}_{m-1}^2(V, W) \xrightarrow{\hat{\delta}^2_m} \hat{C}_{m-2}^3(V, W) \rightarrow \cdots \rightarrow \hat{C}_0^{m+1}(V, W) $$
forms a cochain complex. For every $n\in \N$, define $\hat{C}_\infty^n(V, W) = \bigcap_{m=0}^\infty \hat{C}_m^n(V, W)$ and $\hat{\delta}_\infty^n$ to be the restriction of $\hat{\delta}^n_1$ on $\hat{C}_\infty^n(V, W)$. Obviously, the image of $\hat{\delta}_\infty^n$ is in $\hat{C}_\infty^{n+1}(V, W)$, and the following sequence
$$\hat{C}^0(V, W) \xrightarrow{\hat{\delta}^0} \hat{C}_\infty^1(V, W) \xrightarrow{\hat{\delta}^1_\infty} \hat{C}_\infty^2(V, W) \xrightarrow{\hat{\delta}^2_\infty} \hat{C}_\infty^3(V, W) \rightarrow \cdots \rightarrow \hat{C}_\infty^m(V, W) \rightarrow \cdots $$
forms a cochain complex. 
\end{thm}

%================================================================

\subsection{Cohomology groups}

\begin{defn} For every $n\in \N$, the \textit{$n$-th cohomology group} is defined as 
$$\hat{H}_\infty^n(V, W)=\text{ker} \hat{\delta}_\infty^n / \text{im} \hat{\delta}_\infty^{n-1}$$
\end{defn}

\begin{rema}
We can similarly define the cohomology groups $\hat{H}_m^n(V, W)$ with $\hat{C}_m^n(V, W)$. 
\end{rema}

% \begin{exam}
% When $n=0$, we know that $H^0(V, W)$ consists of vectors $w\in W$ such that 
% $$Y_W^L(u, x)w = Y_W^{s(R)}(u, x)w \in W[[x]]$$
% \end{exam}

\begin{exam}
When $n=1$, we know that $\ker \hat{\delta}_m^1$ consists of maps $f: V \to \widetilde{W}_{z}$ that are composable with $m$ vertex operators, and 
\begin{align*}E[Y_W^L(u, z_1)(f(v))(z_2)] - E[(f(Y_V(u,z_1-\zeta)Y_V(v, z_2-\zeta)\one))( \zeta)] \\
+ E[Y_W^{s(R)}(v, z_2)(f(v))(z_1)]  = 0
\end{align*}
It can be proved that for every $m\in \Z_+$, $\ker \hat{\delta}_m^1 = \ker \hat{\delta}_\infty^1$ and is linearly isomorphic to the space of derivations from $V$ to $W$. Thus 
$$H^1(V, W) \simeq \{\text{Derivation } V \to W\} / \{\text{Inner derivation } V \to W\}$$
Please see details in \cite{HQ-Red} and in \cite{Q-Thesis}. 
\end{exam}

\noindent {\small \sc Department of Mathematics, Yale University, New Haven, CT 06511}

\noindent {\em E-mail address}: fei.qi@yale.edu

\end{document}